\documentclass{article}
\usepackage{amsmath, amsthm, amssymb, amsfonts}
\usepackage{bm}
\usepackage{empheq}
\usepackage[T1]{fontenc}
\usepackage{xcolor}

\numberwithin{equation}{section}

\usepackage{geometry}
\usepackage[shortlabels]{enumitem}
\usepackage[normalem]{ulem}

\usepackage{algorithm}
\usepackage{algorithmic}
\usepackage{adjustbox}
\usepackage{booktabs}
\usepackage{multirow}
\usepackage{siunitx}
\usepackage{array}
\usepackage{graphicx}
\usepackage{caption}
\usepackage{subcaption}
\usepackage{tikz}

\graphicspath{ {./img/} }

\usepackage{cite}

\usepackage{hyperref}
\hypersetup{
    colorlinks=true,
    linkcolor=blue
}

\usepackage{thmtools}
\usepackage{thm-restate}

\usepackage[capitalize, noabbrev, nameinlink]{cleveref}

\declaretheorem[name=Theorem, numberwithin=section]{theorem}
\declaretheorem[name=Proposition, sibling=theorem]{proposition}

\declaretheorem[name=Lemma,       sibling=theorem]{lemma}

\declaretheorem[name=Definition,  style=definition, sibling=theorem]{definition}
\declaretheorem[name=Example,     style=definition, sibling=theorem]{example}
\declaretheorem[name=Remark,      style=remark,     sibling=theorem]{remark}

\crefdefaultlabelformat{#2\textup{#1}#3}

\crefformat{equation}{\textup{#2(#1)#3}}
\crefrangeformat{equation}{\textup{#3(#1)#4--#5(#2)#6}}
\crefmultiformat{equation}{\textup{#2(#1)#3}}{ and \textup{#2(#1)#3}}{, \textup{#2(#1)#3}}{, and \textup{#2(#1)#3}}

\Crefformat{equation}{#2Equation~\textup{(#1)}#3}
\Crefrangeformat{equation}{Equations~\textup{#3(#1)#4--#5(#2)#6}}
\Crefmultiformat{equation}{Equations~\textup{#2(#1)#3}}{ and \textup{#2(#1)#3}}{, \textup{#2(#1)#3}}{, and \textup{#2(#1)#3}}

\crefname{section}{section}{sections}
\Crefname{section}{Section}{Sections}

\begin{document}

\begin{center}
{ \bf \large 
    \textsc{An Optimization Approach to Weight Collocation} \\
    \textsc{for Scattered Spherical Data}
}

\vspace{0.3cm}
\scshape
\renewcommand{\thefootnote}{\fnsymbol{footnote}}

Congpei An\footnote{School of Mathematics and Statistics, Guizhou University, Guiyang 550025, China. This author was supported by NSFC (No. 12371099). Email: andbachcp@gmail.com}
\qquad
Xiannan Hu\footnote{Department of Mathematics, The University of Hong Kong, Hong Kong, China. Email: hans0711@connect.hku.hk}
\qquad
Xiaoming Yuan\footnote{Department of Mathematics, The University of Hong Kong, Hong Kong, China. This author was supported by the Croucher Senior Fellowship and the GRF 17305825. Email: xmyuan@hku.hk}

\renewcommand{\thefootnote}{\arabic{footnote}}
\setcounter{footnote}{0}

\vspace{0.3cm}
\today
\end{center}

\begin{abstract}
We introduce an optimization approach for constructing spherical quadrature rules on arbitrarily scattered data. Rather than designing node placements, the new approach focuses on optimally computing the weights for fixed configurations. Motivated by P\'olya's necessary and sufficient conditions for quadrature convergence in 1933, we argue that pursuing weight positivity and high algebraic exactness for scattered data approximation is not necessary. To align the quadrature design with the underlying theory of approximation, we construct convex optimization models with suitable objective functionals by examining the accuracy of numerical integration with reproducing kernels of Sobolev spaces and the performance of hyperinterpolation with Marcinkiewicz–Zygmund (MZ) inequalities. The resulting optimization models encode the spatial distribution of the scattered sites and the analytic properties of the target function spaces. The proposed approach enables the derivation of rigorous theoretical stability bounds, and the resulting quadrature weights are efficiently computable by modern convex optimization techniques. Numerical results are reported to demonstrate the performance of the optimization approach for fundamental approximation tasks such as numerical integration and hyperinterpolation for scattered spherical data.
\end{abstract}

\phantom{\textbf{\textsc{Keywords}}\;\;\,}inequality, convex optimization.

\vspace{0.3cm}


\section{Introduction}
Numerical integration over the unit sphere $\mathbb{S}^{2}: = \{\mathbf{x}\in\mathbb{R}^{3}\,|\,\|\mathbf{x}\|_2 = 1\}$ is a ubiquitous task in disciplines ranging from computer graphics \cite{ramamoorthi2004signal}, geophysics \cite{freeden1998constructive, freeden2008spherical, sambridge1995geophysical, hofmann2006physical, simons2006spatiospectral}, planetary science \cite{turcotte1981role, wieczorek2007gravity, wieczorek1998potential, gorski2005healpix}, quantum chemistry \cite{choi1999rapid, ritchie1999fast}, astrophysics \cite{bennett1996four, jarosik2011seven} and machine learning \cite{cohen2018spherical, passaro2023reducing,montufar2022distributed}. A quadrature rule $Q[X_N, \mathbf{w}]$ defined by
\begin{equation*}
    Q[X_N, \mathbf{w}](f) := \sum_{j=1}^N w_j f(\mathbf{x}_j)
\end{equation*}
consists of two components: a set of nodes $X_N=\{\mathbf{x}_1, \ldots, \mathbf{x}_N\}\subseteq\mathbb{S}^2$ and a corresponding weight vector $\mathbf{w} = (w_1, \ldots, w_N)^\top \in \mathbb{R}^N$. Historically, the approximation theory literature has focused primarily on optimizing node configurations. This is well exemplified by the theory of spherical designs \cite{Delsarte1977SphericalCA, bannai2009survey}, which seeks ideal node placements that achieve high algebraic exactness using equal weights. 

While theoretically profound, the emphasis on ideal node placements frequently conflicts with the constraints of practical applications.
In many real-world scenarios, function evaluations are restricted to a prescribed set of arbitrarily scattered sites. For instance, in geomagnetism and satellite gravimetry, missions such as MAGSAT, CHAMP, and GRACE sample the sphere along orbits governed by celestial mechanics~\cite{langel1982geomagnetic, mandea2006magnetic, olsen2006chaos, swenson2002methods, whaler1994downward}. Cosmology often utilizes pixelation schemes such as HEALPix, where the grid geometry is driven by resolution-efficiency trade-offs rather than quadrature optimality~\cite{gorski2005healpix}. In computer graphics, environment maps and omnidirectional cameras yield unstructured spherical samples for rendering and inference tasks~\cite{cohen2018spherical}. In these contexts, sampling locations are determined by physical, instrumental, or logistical constraints. Reconfiguring the observation network is generally financially and operationally prohibitive. Consequently, when the nodes are fixed and potentially poorly distributed, the burden of ensuring both integration accuracy and numerical stability rests entirely on the design of the quadrature weights.

Classical quadrature design pursues high algebraic exactness, often accompanied by the requirement of strictly positive weights. While effective for well-distributed nodes, irregular node distributions lack the spatial resolution required to appropriately sample high-frequency spherical harmonics, so forcing the exact integration of these components becomes numerically unstable and overfits the quadrature to the discrete geometry. This manifests as weights with mixed signs and a significantly inflated $\ell_1$-norm, which amplifies floating-point cancellation errors and degrades the integration accuracy for general non-polynomial functions. According to P\'olya's theorem \cite{polya1933konvergenz} (cf. \cref{thm:polya}), quadrature convergence requires only uniform boundedness of quadrature weights and an asymptotic exactness on polynomials. The classical criteria of strict positivity and finite-stage high exactness are therefore sufficient, but not necessary, for convergence. Enforcing them on poorly distributed data can instead be counterproductive. 

Motivated by this fact, we propose to relax these rigid classical requirements in favor of moderate algebraic exactness and quadrature weight design. We formalize this shift of  methodologies and term it \textit{weight collocation}, by framing the construction of quadrature weights on scattered data as an optimization problem (cf. \Cref{sec:collocation}). Instead of solving a nearly square and ill-conditioned exactness system, we determine the weight vector by minimizing an objective functional subject to a prescribed moderate algebraic exactness. We cast the abstract model as standard and tractable convex optimization problems, such as quadratic programming (QP) and semidefinite programming (SDP) problems, both of which can be solved efficiently with modern convex optimization techniques. 

Our more specific aim is to develop the following two complementary collocation methodologies for numerical integration and hyperinterpolation, respectively:
\begin{itemize}
    \item \textit{Kernel Collocation} (cf. \Cref{sec:kernel}): We construct quadrature weights by minimizing the \textit{generalized discrepancy}, also known as the worst-case integration error in Sobolev spaces. This formulation reduces the quadrature design to QP. To practically compute the associated matrix in the absence of closed-form kernels, we derive theoretically justified series truncation strategies. For highly clustered nodes, we further introduce a bandlimited formulation that filters unsupported spherical harmonics. Theoretically, our spectral analysis of the QP matrix reveals two asymptotic behaviors governed by functional smoothness: as smoothness approaches zero, the framework recovers classical $\ell_2$-minimization; as it increases, the condition number grows exponentially.
    \item \textit{Marcinkiewicz--Zygmund (MZ) Collocation} (cf. \Cref{sec:MZ}): Targeting hyperinterpolation, we first prove, underpinned by MZ inequalities, a stability-accuracy decomposition of hyperinterpolation error under the general setting where neither exactness nor weight positivity is assumed. To control stability, we introduce a geometry-aware regularizer called the $\chi^2$-divergence that penalizes weight deviations from a Voronoi geometric prior, a choice theoretically justified by its $2$-optimality. To improve accuracy, we consider optimization on the spectrum of the Gram matrix. Fusing these objectives yields a unified model that directly reflects the error decomposition. To solve this efficiently, we develop a customized interior-point method that exploits the sum-of-rank-one structure of the Gram matrix to reduce the per-iteration complexity.
\end{itemize}

Finally, \Cref{sec:num_exp} provides a comprehensive empirical study utilizing two distinct types of node sets: artificial low-discrepancy sequences (Halton points) and geometrically challenging real-world satellite trajectories (MAGSAT points). We first validate our theoretical stability bounds, confirming the spectral conditioning limits of the discrepancy matrix and demonstrating the practical efficacy of the proposed geometry-aware regularizer. Then, we proceed to evaluate the collocated weights on the fundamental tasks of numerical integration and hyperinterpolation, from which notable improvements in performance are observed. 

\section{Preliminaries} \label{sec:preliminaries}

\subsection{Geometry of the Sphere}
We equip $\mathbb{S}^2$ with the Lebesgue surface measure $\omega$, normalized so that $|\mathbb{S}^2| = \int_{\mathbb{S}^2} d\omega(\mathbf{x}) = 4\pi$. We parametrize points $\mathbf{x} \in \mathbb{S}^2$ using spherical coordinates $(\theta, \phi) \in [0, \pi] \times [0, 2\pi)$ via
\begin{equation*}
    \mathbf{x} = (\sin\theta\cos\phi, \sin\theta\sin\phi, \cos\theta)^\top.
\end{equation*}
The intrinsic geometry of $\mathbb{S}^2$ is given by the geodesic distance $\mathrm{dist}(\mathbf{x}, \mathbf{y}) = \arccos(\langle \mathbf{x}, \mathbf{y}\rangle)$, where $\langle \cdot, \cdot\rangle$ denotes the standard Euclidean inner product in $\mathbb{R}^3$. For any $\mathbf{x} \in \mathbb{S}^2$ and $r \in [0, \pi]$, we denote the spherical cap of angular radius $r$ centered at $\mathbf{x}$ by
\begin{equation*}
    \mathcal{C}(\mathbf{x}, r) := \{\mathbf{y} \in \mathbb{S}^2 \,|\, \text{dist}(\mathbf{x}, \mathbf{y}) \leq r \}.
\end{equation*}
Its surface area is
\begin{equation} \label{eq:area-sph-cap}
    |\mathcal{C}(\mathbf{x}, r)| = 2\pi \int_0^r \sin\phi\, d\phi = 4\pi\sin^2(r/2).
\end{equation}

We are primarily concerned with the spatial distribution of finite point sets $X_N = \{\mathbf{x}_1, \ldots, \mathbf{x}_N\} \subseteq \mathbb{S}^2$. The quality of the distribution is commonly captured by the \textit{mesh norm} $h_{X_N}$ and \textit{separation distance} $q_{X_N}$, which are defined, respectively, by  
\begin{equation*}
    h_{X_N} := \max_{\mathbf{x} \in \mathbb{S}^2} \min_{\mathbf{x}_j \in X_N} \mathrm{dist}(\mathbf{x}, \mathbf{x}_j)\quad\hbox{and} \quad q_{X_N} := \min_{i \neq j} \mathrm{dist}(\mathbf{x}_i, \mathbf{x}_j).
\end{equation*}

\subsection{Spherical Harmonics}
Let $C(\mathbb{S}^2)$ denote the space of real-valued continuous functions on $\mathbb{S}^2$ equipped with the uniform norm $\|f\|_{\infty} := \sup_{\mathbf{x} \in \mathbb{S}^2} |f(\mathbf{x})|$. For $p \in [1, \infty)$, we denote the standard Lebesgue spaces by $L^p(\mathbb{S}^2)$. In particular, $L^2(\mathbb{S}^2)$ is a Hilbert space with the inner product
\begin{equation*}
    \langle f, g \rangle_{L^2} := \int_{\mathbb{S}^2} f(\mathbf{x})g(\mathbf{x})\,d\omega(\mathbf{x})
\end{equation*}
and the induced norm $\|f\|_{L^2} :=\sqrt{\langle f, f \rangle_{L^2}}$.

\textit{Spherical harmonics} \cite{muller2006spherical} are the restrictions of harmonic homogeneous polynomials in $\mathbb{R}^3$ to $\mathbb{S}^2$. Let $\mathbb{H}_\ell$ denote the space of spherical harmonics of degree $\ell \in \mathbb{N}_0$. It is known that $\mathrm{dim}(\mathbb{H}_\ell) = 2\ell + 1$. We choose an orthonormal basis for each $\mathbb{H}_\ell$
\begin{equation*}
    \{Y_{\ell, k} \,|\, k = -\ell, \ldots, -1, 0, 1, \ldots, \ell\}.
\end{equation*}
The spherical harmonics satisfy the \textit{addition theorem} \cite{muller2006spherical}
\begin{equation} \label{eq:addition-theorem}
    \sum_{k = -\ell}^{\ell} Y_{\ell, k}(\mathbf{x})Y_{\ell, k}(\mathbf{y}) = \frac{2\ell + 1}{4\pi}P_\ell(\langle \mathbf{x}, \mathbf{y}\rangle) \quad \forall\, \mathbf{x}, \mathbf{y} \in \mathbb{S}^2,
\end{equation}
where $P_{\ell}$ is the Legendre polynomial of degree $\ell$ normalized such that $P_\ell(1) = 1$.

Let $\mathbb{P}_n(\mathbb{S}^2) = \bigoplus_{\ell=0}^n \mathbb{H}_\ell$ be the space of spherical polynomials of degree at most $n$, and let $\mathbb{P}(\mathbb{S}^2) = \bigoplus_{\ell=0}^\infty \mathbb{H}_\ell$ be the space of all spherical polynomials. Because spherical harmonics of different degrees are mutually orthogonal, the collection $\{Y_{\ell, k}\,|\, 0 \leq \ell \leq n, |k| \leq \ell\}$ forms an orthonormal basis for $\mathbb{P}_n(\mathbb{S}^2)$, implying $\mathrm{dim}(\mathbb{P}_n(\mathbb{S}^2)) = (n+1)^2$. Furthermore, the set of all spherical harmonics forms a complete orthonormal system for $L^2(\mathbb{S}^2)$. Consequently, any $f \in L^2(\mathbb{S}^2)$ can be represented by its Fourier expansion
\begin{equation*}
f = \sum_{\ell=0}^\infty \sum_{k=-\ell}^{\ell} \hat{f}_{\ell, k} Y_{\ell, k} \quad \text{with} \quad \hat{f}_{\ell, k} := \langle f, Y_{\ell, k}\rangle_{L^2}.
\end{equation*}

\subsection{Hyperinterpolation}
The orthogonal projection operator $T_n: L^2(\mathbb{S}^2) \to \mathbb{P}_n(\mathbb{S}^2)$ onto $\mathbb{P}_n(\mathbb{S}^2)$ is
\begin{equation*}
    T_nf (\mathbf{x}):= \sum_{\ell=0}^n \sum_{k=-\ell}^{\ell} \hat{f}_{\ell,k} Y_{\ell, k} = \int_{\mathbb{S}^2} f(\mathbf{y}) G_n(\mathbf{x}, \mathbf{y}) \, d\omega(\mathbf{y}),
\end{equation*}
where
\begin{equation} \label{eq:reproducing-kernel-Pn}
    G_n(\mathbf{x}, \mathbf{y}) := 
    \sum_{\ell=0}^n \sum_{k = -\ell}^{\ell} Y_{\ell, k}(\mathbf{x})Y_{\ell, k}(\mathbf{y}) = \sum_{\ell=0}^n \frac{2\ell+1}{4\pi} P_\ell(\langle \mathbf{x}, \mathbf{y}\rangle).
\end{equation}
In practice, the Fourier coefficients $\hat{f}_{\ell, k}$ are often approximated via a quadrature $Q[X_N, \mathbf{w}]$. This induces a discrete bilinear form on $C(\mathbb{S}^2)$ given by
\begin{equation*}
    \langle f, g \rangle_Q := Q[X_N, \mathbf{w}](fg) = \sum_{j=1}^N w_j f(\mathbf{x}_j)g(\mathbf{x}_j), \quad f, g \in C(\mathbb{S}^2).
\end{equation*}
Following \cite{sloan1995polynomial}, we define the \textit{hyperinterpolation} operator $L_n: C(\mathbb{S}^2) \to \mathbb{P}_n(\mathbb{S}^2)$ by approximating $\hat{f}_{\ell, k} \approx \langle f, Y_{\ell, k}\rangle_Q$, yielding
\begin{equation} \label{eq:hyperinterpolation}
    L_nf (\mathbf{x}):= \sum_{\ell=0}^n \sum_{k=-\ell}^{\ell} \langle f, Y_{\ell, k}\rangle_Q Y_{\ell, k}(\mathbf{x}) = \langle f, G_n(\mathbf{x}, \cdot)\rangle_Q,
\end{equation}
where the last equality follows from the addition theorem \eqref{eq:addition-theorem}. In our settings, the quadrature weights $w_j$ are not restricted to being positive. Consequently, $\langle \cdot, \cdot\rangle_Q$ does not necessarily define an inner product. Historically, theoretical properties and error bounds of hyperinterpolation have been primarily developed for quadrature rules with strictly positive weights \cite{sloan1995polynomial, hesse2006hyperinterpolation, reimer2000hyperinterpolation, an2024bypassing}.
Finally, let us prove an identity that will be used frequently later.
\begin{lemma} \label{lem:hyperinterpolation-lemma}
    Let $L_n$ be the hyperinterpolation operator associated with a quadrature $Q[X_N, \mathbf{w}]$. For any $f \in C(\mathbb{S}^2)$ and $p \in \mathbb{P}_n(\mathbb{S}^2)$, it holds that
    \begin{equation}
        \langle L_n f, p \rangle_{L^2} = \langle f, p \rangle_Q.
    \end{equation}
\end{lemma}
\begin{proof}
    Fix arbitrary $f \in C(\mathbb{S}^2)$ and $p \in \mathbb{P}_n(\mathbb{S}^2)$. By the definition of the hyperinterpolation operator \eqref{eq:hyperinterpolation}, we have
    \begin{equation*}
        L_nf(\mathbf{x}) = \sum_{j=1}^N w_jf(\mathbf{x}_j)G_n(\mathbf{x}, \mathbf{x}_j).
    \end{equation*}
    Taking the $L^2$-inner product with $p$ and exchanging the order of summation and integration yields:
    \begin{equation*}
        \langle L_nf, p \rangle_{L^2} = \int_{\mathbb{S}^2} \left(\sum_{j=1}^N w_jf(\mathbf{x}_j)G_n(\mathbf{x}, \mathbf{x}_j)\right) p(\mathbf{x})\,d\omega(\mathbf{x}) = \sum_{j=1}^N w_jf(\mathbf{x}_j) \int_{\mathbb{S}^2} G_n(\mathbf{x}, \mathbf{x}_j)p(\mathbf{x})\,d\omega(\mathbf{x}).
    \end{equation*}
    Because $p \in \mathbb{P}_n(\mathbb{S}^2)$, the reproducing property of $G_n$ implies that
    \begin{equation*}
        \langle L_nf, p\rangle_{L^2} = \sum_{j=1}^N w_jf(\mathbf{x}_j)p(\mathbf{x}_j) = \langle f, p \rangle_Q.
    \end{equation*}
    The proof is complete.
\end{proof}

\subsection{Sobolev Spaces}
\label{sec:sob}

A \textit{reproducing kernel Hilbert space} (RKHS) $\mathbb{H}$ on $\mathbb{S}^2$ is a Hilbert space in which all point evaluation functionals are bounded. Denote its associated inner product by $\langle \cdot, \cdot\rangle_{\mathbb{H}}$. By the Riesz representation theorem, there exists a unique symmetric and positive definite kernel $K:\mathbb{S}^2\times \mathbb{S}^2\to \mathbb{R}$ such that $K(\mathbf{x},\cdot)\in \mathbb{H}$ and
\begin{equation*}
  \langle f,K(\mathbf{x},\cdot)\rangle_{\mathbb{H}} = f(\mathbf{x})\quad \forall\, f\in \mathbb{H}, \,\mathbf{x}\in \mathbb{S}^2.
\end{equation*}
For example, the polynomial kernel $G_n$ given in \eqref{eq:reproducing-kernel-Pn} is the reproducing kernel of $\mathbb{P}_n(\mathbb{S}^2)$ equipped with the $L^2$ inner product.

To measure the regularity of functions and to quantify integration errors, we introduce a family of Sobolev spaces on the sphere.
For a smoothness index $s \ge 0$, the \textit{Sobolev space} $\mathbb{H}^s(\mathbb{S}^2)$ is defined as the set of all $f\in L^2(\mathbb{S}^2)$ whose Fourier coefficients satisfy
\begin{equation*}
  \sum_{\ell = 0}^{\infty}\sum_{k = -\ell}^{\ell} \frac{1}{a_\ell^{(s)}}\,|\hat{f}_{\ell,k}|^2 < \infty,
\end{equation*}
where $\{a_\ell^{(s)}\}_{\ell \in \mathbb{N}_0}$ is a sequence of positive numbers defining                                 the norm. Unless otherwise stated, a canonical choice we adopt is
\begin{equation} \label{eq:sobolev-norm-constant}
  a_\ell^{(s)} := \left(\ell + \frac{1}{2}\right)^{-2s}, \qquad \ell \ge 0.
\end{equation}
The space $\mathbb{H}^s(\mathbb{S}^2)$ is a Hilbert space with the inner product
\begin{equation*}
  \langle f,g\rangle_{\mathbb{H}^s} := \sum_{\ell = 0}^{\infty} \frac{1}{a_\ell^{(s)}} \sum_{k = -\ell}^{\ell} \hat{f}_{\ell,k}\,\hat{g}_{\ell,k},
\end{equation*}
and it induces the norm
\begin{equation*}
  \|f\|_{\mathbb{H}^s} := \left( \sum_{\ell = 0}^{\infty} \Bigl(\ell + \frac{1}{2}\Bigr)^{2s} \sum_{k = -\ell}^{\ell} |\hat{f}_{\ell,k}|^2 \right)^{1/2}.
\end{equation*}
When $s = 0$, the norm reduces to the standard $L^2$ norm, yielding $\mathbb{H}^0(\mathbb{S}^2) = L^2(\mathbb{S}^2)$.

\begin{remark}
It could be  a mathematical advantage to utilize an equivalent norm by replacing \eqref{eq:sobolev-norm-constant} with a sequence $\{a_\ell^{(s)}\}$ satisfying 
\begin{equation} \label{eq:sobolev-norm-constant-alt}
    a_\ell^{(s)} \asymp (\ell + 1/2)^{-2s}. 
\end{equation}
While this does not alter the underlying function space, it can yield a reproducing kernel with a closed-form expression, which simplifies computations as exploited in \Cref{sec:closed_form}.
\end{remark}

When $s > 1$, the Sobolev embedding theorem ensures that $\mathbb{H}^s(\mathbb{S}^2) \subseteq C(\mathbb{S}^2)$. In this case, $\mathbb{H}^s(\mathbb{S}^2)$ is an RKHS possessing a continuous reproducing kernel $K_s$ given by
\begin{align} \label{eq:reproducing-kernel-sobolev}
    K_s(\mathbf{x}, \mathbf{y}) = \sum_{\ell=0}^\infty \sum_{k=-\ell}^{\ell} a_\ell^{(s)} Y_{\ell,k}(\mathbf{x}) Y_{\ell, k}(\mathbf{y}) = \sum_{\ell=0}^\infty \frac{2\ell+1}{4\pi} a_\ell^{(s)} P_{\ell}(\langle\mathbf{x}, \mathbf{y}\rangle),
\end{align}
with the infinite series converging uniformly and absolutely.
\section{Weight Collocation} \label{sec:collocation}
To practically construct a numerical quadrature rule $Q[X_N, \mathbf{w}]$ for scattered data on the unit sphere $\mathbb{S}^2$, one needs to determine  an appropriate weight vector $\mathbf{w} \in \mathbb{R}^N$ for a prescribed set of nodes $X_N = \{\mathbf{x}_1, \dots, \mathbf{x}_N\} \subseteq \mathbb{S}^2$. Because $X_N$ can be arbitrarily scattered or even poorly distributed in real-world scenarios, how to ensure  accuracy and numerical stability rests entirely upon the proper design of the weights $\mathbf{w}$. We use the term \textit{weight collocation} to denote the process of optimizing these weights, where the discrete quadrature $Q[X_N, \mathbf{w}]$ is driven to better approximate the integration functional
\begin{equation*}
    I(f) = \int_{\mathbb{S}^2} f(\mathbf{x})\, d\omega(\mathbf{x}).
\end{equation*}

\subsection{From the Exactness System Back to Pólya's Condition}
Classically, quadrature accuracy is evaluated primarily through the lens of algebraic exactness. A quadrature $Q[X_N,\mathbf{w}]$ is said to be exact of degree $n$ if
\begin{equation*}
    Q[X_N,\mathbf{w}](p) = I(p) \quad \forall\,p \in \mathbb{P}_n(\mathbb{S}^2).
\end{equation*}
Using the orthonormal basis $\{Y_{\ell, k} \,|\, 0 \leq \ell \leq n, |k| \leq \ell\}$, the requirement of exactness of degree $n$ yields the linear system
\begin{equation} \label{eq:exactness-system}
\mathbf{Y}_n \mathbf{w} = \mathbf{b}_n,
\end{equation}
where the basis matrix $\mathbf{Y}_n \in \mathbb{R}^{(n+1)^2 \times N}$ and the moment vector $\mathbf{b}_n \in \mathbb{R}^{(n+1)^2}$ are defined as
\begin{equation*}
    \mathbf{Y}_n := [Y_{\ell,k}(\mathbf{x}_j)]_{(\ell,k), j} \in \mathbb{R}^{(n+1)^2 \times N} \quad \text{and} \quad \mathbf{b}_n := [I(Y_{\ell, k})]_{(\ell, k)} \in \mathbb{R}^{(n+1)^2},
\end{equation*}
respectively.

From a purely algebraic perspective, for an arbitrary scattered set $X_N$, the exactness system \eqref{eq:exactness-system} is generally only feasible if it is underdetermined. This establishes a crude algebraic capacity of $(n+1)^2 \leq N$, or
\begin{equation} \label{eq:max-exact}
    n \leq \lfloor \sqrt{N} - 1 \rfloor.
\end{equation}
In the context of scattered data, pursuing maximal exactness by pushing $n$ to the upper limit in \eqref{eq:max-exact} poses significant computational challenges. Exhausting the available degrees of freedom leaves the nearly square basis matrix $\mathbf{Y}_n$ highly ill-conditioned. Consequently, the computed weights exhibit severe oscillations with large negative components, triggering significant cancellation errors in floating-point arithmetic. This instability is reminiscent of the divergence observed for high-degree Newton--Cotes formulas on the interval $[-1, 1]$, which arises from enforcing high algebraic exactness at equally spaced nodes \cite{trefethen2022exactness}.

To circumvent this instability, the approximation theory community has historically prioritized quadratures with strictly positive weights. The theoretical appeal is straightforward: if all $w_j > 0$ and the quadrature integrates the constant function exactly, i.e., 
\begin{equation} \label{eq:exact-for-const}
    \sum_{j=1}^N w_j = 4\pi,
\end{equation}
the weights are inherently uniformly bounded. Foundational results, such as Tchakaloff's theorem \cite{tchakaloff1957formules}, which guarantees the existence of positive quadrature weights for finite-dimensional spaces, have influenced decades of quadrature design.

The design principles of high algebraic exactness and strict positivity for quadrature formulas, however, may be inconsistent with practical approximation goals. As articulated  in the recent review by Trefethen \cite{trefethen2022exactness}, forcing the exact integration of high-frequency components often overfits the discrete geometry, thereby sacrificing general approximation power for non-polynomial functions. For scattered data with deficient geometry, strict weight positivity could also be overly restrictive. Relaxing these constraints is, in fact, directly justified by P\'olya's foundational theorem on quadrature convergence \cite{polya1933konvergenz}. We here present a natural generalization to the sphere.

\newpage
\begin{theorem} \label{thm:polya}
    Given a sequence of spherical quadratures $Q[X_N, \mathbf{w}^{(N)}]: C(\mathbb{S}^2) \to \mathbb{R}$, $N \in \mathbb{N}$, of the form $Q[X_N,\mathbf{w}^{(N)}](f) := \sum_{j=1}^N w_j^{(N)} f(\mathbf{x}_j^{(N)})$, suppose the sequence possesses the asymptotic approximation property for polynomials
    \begin{equation} \label{eq:approximation-property}
        \lim_{N \to \infty} Q_N(p) = \int_{\mathbb{S}^2} p(\mathbf{x})\, d\omega(\mathbf{x}) \quad \forall \, p \in \mathbb{P}(\mathbb{S}^2).
    \end{equation}
    Then, the sequence converges for all continuous functions,
    \begin{equation*}
        \lim_{N \to \infty} Q_N(f) = \int_{\mathbb{S}^2} f(\mathbf{x})\, d\omega(\mathbf{x}) \quad \forall \, f \in C(\mathbb{S}^2)
    \end{equation*}
    if and only if the absolute sum of the weights is uniformly bounded
    \begin{equation} \label{eq:polya-cond}
        \sup_{N \in \mathbb{N}} \|\mathbf{w}^{(N)}\|_1 < \infty.
    \end{equation}
\end{theorem}
\cref{thm:polya} demonstrates that finite-stage maximal exactness and strict positivity are sufficient for quadrature convergence, but not necessary. Negative weights are theoretically permissible provided that their overall $\ell_1$-norm remains bounded. Furthermore, the requirement for polynomial exactness is asymptotic. Consequently, designing robust numerical quadratures requires a delicate balance between the approximation accuracy for polynomials and the $\ell_1$-uniform boundedness of the weights \eqref{eq:polya-cond}.

\subsection{An Optimization Approach} \label{sec:framework}
Motivated by Pólya's theorem \cite{polya1933konvergenz}, we propose an optimization approach to focus on computing weights for fixed node configurations rather than designing node placements.

In the optimization models to follow, the objective function $P(\mathbf{w})$ is chosen to promote either approximation accuracy or numerical stability, while the constraints enforce a basic level of exactness on spherical polynomials. The constraints should be underdetermined enough to provide sufficient degrees of freedom for optimizing the weights. Therefore, rather than setting the exactness $n$ close to the algebraic capacity $\lfloor \sqrt{N} - 1 \rfloor$, we restrict it to a lower degree. While various heuristics exist for choosing this degree, we adopt a data-dependent approach. Specifically, we define the \textit{critical degree} $n^+$ as the highest exactness that the specific spatial distribution of $X_N$ can support using nonnegative weights
\begin{equation} \label{eq:critical-exactness}
    n^+ := \max\{n \in \mathbb{N}_0 \,|\, \exists\,\mathbf{w} \in \mathbb{R}^N, \mathbf{w} \ge \mathbf{0},\; \text{s.t.}\, \mathbf{Y}_n\mathbf{w} = \mathbf{b}_n\}.
\end{equation}
This threshold serves two purposes. First, it provides a data-driven measure of the geometric capacity of the node distribution. Since this degree is typically lower than the algebraic limit, the resulting system is not severely ill-conditioned and leaves sufficient degrees of freedom in the null space of $\mathbf{Y}_{n^+}$ for optimization. Second, fixing the exactness at $n^+$ provides a consistent basis for comparing our collocation methods with classical positive quadratures. At this degree, classical positive quadratures are at their limit, so any improvement in performance can be attributed to the choice of the optimization objective.

With the exactness constrained to $n^+$, we seek a weight vector that minimizes the chosen penalty functional $P(\mathbf{w})$, and it is represented by \begin{align} \label{eq:framework}
\begin{aligned}
\min \quad & P(\mathbf{w}) \\
\text{s.t.} \quad & \mathbf{Y}_{n^+} \mathbf{w} = \mathbf{b}_{n^+}, \quad \mathbf{w} \in \mathbb{R}^N.
\end{aligned}
\end{align}
With suitable convex penalty functionals $P(\mathbf{w})$, as mentioned, the model  \eqref{eq:framework} reduces to standard convex optimization problems, such as QP and SDP problems. For (\ref{eq:framework}), a  critical question is how to specify $P(\mathbf{w})$. One possible choice is to consider standard regularizers such as the $\ell_p$-norms. For instance, penalizing the $\ell_1$-norm ($P(\mathbf{w}) = \|\mathbf{w}\|_1$) directly reflects the uniform boundedness requirement in \cref{thm:polya}; this approach is utilized by \cite{devore2019computing}, albeit motivated by a different optimal recovery perspective. However, the $\ell_1$-norm promotes sparsity, so its minimization often yields many zero weights. This discards available observations and may not fully exploit the geometric information of the node set. Alternatively, $\ell_2$-norm minimization is a common heuristic for computing quadrature weights (see, e.g., \cite{le2004galerkin}). Minimizing the Euclidean norm penalizes large individual weights, thereby promoting uniformity. While this strategy works well for uniform point sets, it is less suitable for scattered sites. In the presence of local clusters, a common feature in real-world satellite trajectories, a geometrically faithful quadrature should assign smaller weights to clustered nodes and larger weights to isolated ones. The $\ell_2$-norm inherently resists this necessary geometric variation. These regularizers fall short because they treat $\mathbf{w}$ solely as an abstract algebraic vector in $\mathbb{R}^N$, independent of the actual quadrature $Q[X_N, \mathbf{w}]$. They encode neither the spatial distribution of the scattered sites $X_N$ nor the analytic properties of the target function spaces, offering no theoretical guarantee on performance.

To overcome these limitations, $P(\mathbf{w})$ should be elevated to capture the underlying structures of the approximation problem. To address two fundamental approximation tasks, namely numerical integration and hyperinterpolation, we develop respective methodologies to specify $P(\mathbf{w})$, with all proposed strategies summarized in \cref{tab:strats}. In the subsequent sections, we rigorously formulate these methodologies, derive their theoretically stability guarantees, and provide practical computational realizations.

\begin{table}[htbp]
    \centering
    \caption{Overview and taxonomy of our weight collocation strategies for spherical scattered data.}
    \label{tab:strats}
    \renewcommand{\arraystretch}{1.25}
    \begin{adjustbox}{max width=\textwidth}
    \begin{tabular}{@{}l p{2.1cm} l p{6cm} l@{}}
        \toprule
        Methodology & Strategy & $P(\mathbf{w})$ & Target & Reference \\
        \midrule
        -- & -- & $\|\mathbf{w}\|_2^2$, $\|\mathbf{w}\|_1$, etc. & Weight vector behavior & \Cref{sec:framework} \\
        \midrule
        \multirow{8}{*}{Kernel}  & Discrepancy & $\mathbf{w}^\top \mathbf{K}_s \mathbf{w}$ & Worst-case integration error in smooth Sobolev spaces & \Cref{sec:minimum-discrepancy} \\
        & Truncated discrepancy & $\mathbf{w}^\top \mathbf{K}_s^L \mathbf{w}$ & Worst-case integration error in Sobolev spaces when closed-form kernels are unavailable & \Cref{sec:series-trunc} \\
        & Bandlimited & $\|\bm{\Gamma} (\mathbf{Y}_L \mathbf{w} - \mathbf{b}_L) \|_1 + \lambda R(\mathbf{w})$ & Worst-case integration error in a bandlimited space for highly clustered scattered sites & \Cref{sec:bandlimited} \\
        \midrule
        \multirow{6}{*}{MZ}  & $\chi^2$-divergence & $\|\mathbf{w}\|_{\mathcal{R},2}$ & Stability of hyperinterpolation & \Cref{sec:geometric-aware-regularization} \\
         & Spectral & $\|\mathbf{I}-\mathbf{G}_n[X_N](\mathbf{w})\|_2$ & Accuracy of hyperinterpolation & \Cref{sec:minimal-eta} \\
         & $D$-optimal & $\log\det \mathbf{G}_n[X_N](\mathbf{w})^{-1}$ & Computationally efficient surrogate for spectral collocation & \Cref{sec:well-conditioned-Gram-matrix} \\
         & Unified & $J_n(\mathbf{w}) + (\lambda/2)\|\mathbf{w}\|_{\mathcal{V}, 2}^2$ & Accuracy-stability decomposition of hyperinterpolation error & \Cref{sec:sdp} \\
        \bottomrule
    \end{tabular}
    \end{adjustbox}
\end{table}
\section{Kernel Collocation} \label{sec:kernel}

Our first weight collocation strategy addresses the fundamental task of numerical integration. To ensure generalization beyond polynomials, we evaluate the integration error within Sobolev spaces $\mathbb{H}^s(\mathbb{S}^2)$. In this setting, the reproducing kernel serves as the foundational tool, structurally linking the underlying Sobolev space to the discrete geometry of the scattered sites.

For $s > 1$, both the quadrature $Q[X_N, \mathbf{w}]$ and the integration functional $I$ are bounded linear functionals on $\mathbb{H}^s(\mathbb{S}^2)$. By the Riesz representation theorem, their representers are, respectively
\begin{equation*}
  R_Q = \sum_{j=1}^{N} w_j K_s(\mathbf{x}_j,\cdot)  \qquad
  \text{and} \qquad R_I = \int_{\mathbb{S}^2} K_s(\mathbf{x},\cdot)\,d\omega(\mathbf{x}).
\end{equation*}
For any $f \in \mathbb{H}^s(\mathbb{S}^2)$, the integration error is bounded via the Cauchy--Schwarz inequality
\begin{equation} \label{eq:int-err-cauchy-schwarz}
    |Q[X_N, \mathbf{w}](f) - I(f)| = |\langle R_Q - R_I, f\rangle_{\mathbb{H}^s}| \leq \|f\|_{\mathbb{H}^s} \|R_Q - R_I\|_{\mathbb{H}^s}.
\end{equation}
Assuming the quadrature integrates constant functions exactly, i.e., \eqref{eq:exact-for-const}, it is a standard result in approximation theory \cite{hesse2005worst,brauchart2007numerical} that the norm of the representer difference evaluates to
\begin{equation} \label{eq:discrepancy}
    D_s(X_N, \mathbf{w}) := \|R_Q - R_I\|_{\mathbb{H}^s} = \left(\sum_{i=1}^N \sum_{j=1}^N w_i w_j K_s'(\mathbf{x}_i, \mathbf{x}_j)\right)^{1/2},
\end{equation}
where $K_s'$ is the reproducing kernel $K_s$ \eqref{eq:reproducing-kernel-sobolev} excluding the zero-degree harmonic
\begin{equation} \label{eq:reproducing-kernel-sobolev-trunc}
    K_s'(\mathbf{x}, \mathbf{y}) = \sum_{\ell = 1}^\infty \frac{2\ell+1}{4\pi} a_\ell^{(s)} P_\ell(\langle \mathbf{x}, \mathbf{y}\rangle).
\end{equation}
The expression $D_s(X_N, \mathbf{w})$ is defined as the \textit{generalized discrepancy} of quadrature $Q[X_N, \mathbf{w}]$ in $\mathbb{H}^s(\mathbb{S}^2)$ \cite{cui1997equidistribution,freeden1998constructive}. Substituting \eqref{eq:discrepancy} back into \eqref{eq:int-err-cauchy-schwarz} yields the celebrated Koksma--Hlawka inequality in Sobolev spaces \cite[Theorem~3.1]{cui1997equidistribution}
\begin{equation*}
    |Q[X_N, \mathbf{w}](f) - I(f)| \leq \|f\|_{\mathbb{H}^s} D_s(X_N, \mathbf{w}) \quad \forall\, f \in \mathbb{H}^s(\mathbb{S}^2).
\end{equation*}
This inequality establishes the generalized discrepancy $D_s(X_N, \mathbf{w})$ as the worst-case integration error over the unit ball in $\mathbb{H}^s(\mathbb{S}^2)$. The notion of generalized discrepancy has proven to be useful in the quasi-Monte Carlo (QMC) settings \cite{sloan2026qmc}. This motivates us to develop a deterministic kernel collocation methodology that minimizes this discrepancy for numerical integration within Sobolev spaces.

\subsection{Minimum Discrepancy Quadratures in Sobolev Spaces} \label{sec:minimum-discrepancy}

To improve the performance of quadratures, we seek to directly minimize this worst-case error. The square of the generalized discrepancy \eqref{eq:discrepancy} is a quadratic form that separates the components of the approximation problem: the sites $X_N$ enter through the kernel evaluations $K_s'(\mathbf{x}_i, \mathbf{x}_j)$, with the smoothness parameter $s$ controls the weighting of high-frequency components, and the weight vector $\mathbf{w}$ is the optimization variable. To formalize this, we define the \textit{discrepancy matrix} of $X_N$ as
\begin{equation*}
    \mathbf{K}_s := [K_s'(\mathbf{x}_i, \mathbf{x}_j)]_{i, j = 1}^N.
\end{equation*}
Because the infinite series \eqref{eq:reproducing-kernel-sobolev-trunc} defining $K_s'$ is uniformly and absolutely convergent for $s > 1$, the matrix $\mathbf{K}_s$ is well-defined and real-valued in these cases. The treatment for low-smoothness cases $0 \leq s \leq 1$, which do not admit continuous reproducing kernels, will be recovered later via series truncation.

Using the discrepancy matrix, we pose the following \textit{discrepancy collocation} model:
\begin{align} \label{eq:discrepancy-collocation}
  \begin{aligned}
    \min \quad & \mathbf{w}^\top \mathbf{K}_s \mathbf{w} \\
    \text{s.t.} \quad & \mathbf{Y}_{n^+} \mathbf{w} = \mathbf{b}_{n^+}, \quad \mathbf{w} \in \mathbb{R}^N.
  \end{aligned}
\end{align}
The quadrature derived from the discrepancy collocation \eqref{eq:discrepancy-collocation} is optimal in the worst-case sense within the Sobolev spaces $\mathbb{H}^s(\mathbb{S}^2)$ among all quadratures exact of degree $n^+$. Because the Hessian matrix $\mathbf{K}_s$ is symmetric positive definite (cf. \cref{prop:positive-definiteness-discrepancy-matrix}), discrepancy collocation \eqref{eq:discrepancy-collocation} is a standard strictly convex QP with linear equality constraints, which can be solved efficiently using off-the-shelf solvers.

\begin{proposition} \label{prop:positive-definiteness-discrepancy-matrix}
    Suppose $s > 1$. The discrepancy matrix $\mathbf{K}_s$ is symmetric positive definite.
\end{proposition}
\begin{proof}
    The symmetry is trivial in view of the definition \eqref{eq:reproducing-kernel-sobolev-trunc}. Note that
    \begin{equation*}
        K_s'(\mathbf{x}, \mathbf{y}) = \sum_{\ell=0}^\infty \frac{2\ell+1}{4\pi} \tilde{a}_\ell^{(s)} P_\ell(\langle \mathbf{x}, \mathbf{y}\rangle),
    \end{equation*}
    where
    \begin{equation*}
        \tilde{a}_\ell^{(s)} := \left\{
            \begin{array}{ll}
               0, & \text{if $\ell = 0$}, \\
               (\ell+1/2)^{-2s}, & \text{if $\ell > 0$}. 
            \end{array}
        \right.
    \end{equation*}
    The set of positive indices is $\mathcal{A} := \{\ell \in \mathbb{N}_0 : \tilde{a}_\ell^{(s)} > 0\} = \mathbb{N}$. Since $\mathbb{N}$ contains arbitrarily long sequences of consecutive even and odd integers, a result in \cite[Corollary~6.9]{ron1996strictly} implies that $K_s'$ is a positive definite kernel\footnote{For historical reasons, some scholars call this a strictly positive definite kernel instead.} on $\mathbb{S}^2$. Consequently, the discrepancy matrix $\mathbf{K}_s$ is positive definite for any set of distinct sites $X_N$.
\end{proof}

\subsection{Spectral Stability} \label{sec:spectral-stability}

While \cref{prop:positive-definiteness-discrepancy-matrix} establishes the strict convexity and well-posedness of \eqref{eq:discrepancy-collocation}, the practical conditioning of the discrepancy matrix is also important. Ill-conditioning can affect the reliability of optimization solvers and the resulting weights. We therefore analyze the condition number of $\mathbf{K}_s$, which depends on both the node geometry and the smoothness index $s$. The main tool used in the spectral analysis is the B-spline kernels.
\begin{definition}[\!\!{\cite[Definition~3]{keiner2007efficient}}] \label{def:bspline}
    Let $\beta \in \mathbb{N}$ be given. The \textit{normalized B-spline} of order $\beta$, $g_\beta:[-1/2, 1/2] \to \mathbb{R}$ is defined by $g_\beta(t) := \beta N_\beta(\beta t + \beta/2)$, 
    where $N_\beta$ denotes the \textit{cardinal B-spline} of order $\beta$. The cardinal B-splines are given by
    \begin{equation*}
        N_{\beta+1}(t) = \int_{t-1}^t N_\beta(\tau)\, d\tau, \;\beta \in \mathbb{N}, \quad\quad N_1(t) = \left\{
            \begin{array}{ll}
                1, & 0 < t < 1, \\
                0, & \text{otherwise}.
            \end{array}
        \right.
    \end{equation*}
    Moreover, we define for $\beta \in \mathbb{N}$ and $n \in \mathbb{N}$ the \textit{B-spline kernel} $B_{\beta, n}:[-1,1]\to\mathbb{R}$ by
    \begin{equation*}
        B_{\beta, n}(t) = \frac{1}{\|g_\beta\|_{1, n}}\sum_{\ell=0}^n (2 - \delta_{\ell, 0})g_\beta\left(\frac{\ell}{2(n+1)}\right)T_\ell(t),
    \end{equation*}
    where $T_\ell(t):=\cos(\ell \arccos(t))$ denotes the Chebyshev polynomial of degree $\ell$, and $\|\cdot\|_{1,n}$ denotes the discrete norm
    \begin{equation*}
        \|g_\beta\|_{1,n}:=\sum_{\ell=-n}^n g_\beta\left(\frac{\ell}{2(n+1)}\right).
    \end{equation*}
\end{definition}

The following lemma establishes the positive semidefiniteness of B-spline kernels and a useful localization property.

\begin{lemma}[\!\!{\cite[Lemma~7]{keiner2007efficient}}] \label{lem:local}
    The B-spline kernel $B_{\beta, n}$ satisfies, for $n \geq \beta - 1$ and $\theta \in (0, \pi]$, the localization property
    \begin{equation} \label{eq:local}
        |B_{\beta, n}(\cos \theta)| \leq c_\beta|(n+1)\theta|^{-\beta}  \quad \hbox{with} \quad c_\beta := \frac{(2^\beta - 1)\zeta(\beta)\beta^\beta}{2^{\beta-1} - \zeta(\beta)\pi^{-\beta}},
    \end{equation}
    where $\zeta(\beta) := \sum_{r=1}^\infty r^{-\beta}$ is the Riemann zeta function. Moreover, it is normalized by $B_{\beta, n}(1) = 1$ and can be represented as 
    \begin{equation*}
        B_{\beta, n}(t) = \sum_{\ell=0}^n \frac{2\ell+1}{4\pi}\alpha_\ell P_\ell(t), \quad t \in [-1, 1]
    \end{equation*}
    with the \textit{positive} Fourier--Legendre coefficients
    \begin{equation*}
        \alpha_\ell := 2\pi \int_{-1}^1 P_\ell(t)B_{\beta, n}(t)\,dt > 0, \quad \ell = 0, \ldots, n.
    \end{equation*}
\end{lemma}

The localization bound \eqref{eq:local} is not defined for $\theta = 0$ and is not integrable over $[0, \pi]$ for $\beta > 1$. To desingularize the estimate at $\theta = 0$, we establish the following corollary using the uniform boundedness of B-spline kernels.
\begin{lemma} \label{lem:shifted-local}
    For $n \geq \beta - 1$ and $\theta \in [0, \pi]$, it holds that
   \begin{equation} \label{eq:shifted-local}
    |B_{\beta,n}(\cos\theta)| \leq c_\beta' \, |1+(n+1)\theta|^{-\beta}, \quad \text{where } c_\beta' := \bigl(1 + c_\beta^{1/\beta}\bigr)^\beta.
\end{equation}
\end{lemma}
\begin{proof}
    For any $\theta \in [0, \pi]$, we have
    \begin{equation} \label{eq:boundedness-of-B-spline-kernel}
        |B_{\beta, n}(\cos\theta)| \leq \sum_{\ell=0}^n \frac{2\ell+1}{4\pi} \alpha_\ell |P_\ell(\cos\theta)| \leq \sum_{\ell=0}^n \frac{2\ell+1}{4\pi} \alpha_\ell P_\ell(1) = B_{\beta, n}(1) = 1,
    \end{equation}
    where we use the positivity of $\alpha_\ell$ and $|P_\ell(t)|\leq P_\ell(1)$. Combining the bounds \eqref{eq:local} and \eqref{eq:shifted-local} gives
    \begin{equation*}
        |B_{\beta,n}(\cos\theta)| \leq \min\{1, c_\beta|(n+1)\theta|^{-\beta}\}.
    \end{equation*}
    Note that the choice of $c_\beta'$ in \eqref{eq:shifted-local} is made such that
    \begin{equation*}
        c_\beta' = \max_{\theta \in [0, \pi]} \min\{1, c_\beta|(n+1)\theta|^{-\beta}\} |1+(n+1)\theta|^{\beta} \geq |B_{\beta, n}(\cos\theta)||1+(n+1)\theta|^\beta.
    \end{equation*}
    The proof is complete.
\end{proof}

The modified localization bound \eqref{eq:shifted-local} is crucial for establishing bounds of the quadratic form associated with the B-spline kernels. Our proof uses the same Gershgorin argument as in \cite[Theorem~2.4]{narcowich1998stability}.
\begin{theorem} \label{thm:B-sp-ker}
    Let $\beta \in \mathbb{N}$ with  $\beta \geq 3$ and $n \in \mathbb{N}$ with $n \geq \max\{\beta - 1, 1/q_{X_N}\}$. For any set $X_N = \{\mathbf{x}_1, \ldots, \mathbf{x}_N\} \subseteq \mathbb{S}^2$, it holds that
    \begin{align*}
        (1-\gamma) \|\mathbf{c}\|_2^2 \leq \sum_{i=1}^N \sum_{j=1}^N c_ic_j B_{\beta, n}(\langle \mathbf{x}_i, \mathbf{x}_j\rangle) \leq (1+\gamma)\|\mathbf{c}\|_2^2
    \end{align*}
    for every vector $\mathbf{c} = (c_1, \ldots, c_N)^\top \in \mathbb{R}^N$, where
    \begin{equation} \label{eq:gamma}
        \gamma := 25c_\beta' \left[\frac{1}{((n+1)q_{X_N})^\beta} + \frac{1}{(n+1)^2q_{X_N}^2(\beta-1)(\beta-2)}\right].
    \end{equation}
\end{theorem}
\begin{proof}
    By a Gershgorin argument, the quadratic form
    \begin{equation*}
        \sum_{i=1}^N \sum_{j=1}^N c_ic_j B_{\beta,n}(\langle \mathbf{x}_i, \mathbf{x}_j \rangle)
    \end{equation*}
    is bounded above and below, respectively, by the quantities
    \begin{equation} \label{eq:gersh-bdd}
        (B_{\beta,n}(1) \pm \gamma_0)\|\mathbf{c}\|_2^2 = (1 \pm \gamma_0)\|\mathbf{c}\|_2^2, \quad \text{where} \quad \gamma_0 := \max_{1 \leq j \leq N} \left[\sum_{\substack{i=1 \\ i \neq j}}^N |B_{\beta, n}(\langle \mathbf{x}_i, \mathbf{x}_j\rangle)| \right].
    \end{equation}
    The rest of proof bounds $\gamma_0$ from above. Without loss of generality, by reordering the points in $X_N$ and exploiting the rotational invariance of the kernel on $\mathbb{S}^2$, we may assume that the maximum in $\gamma$ is attained at $j = 1$ and that $\mathbf{x}_1 = (0, 0,1 )^\top$ corresponds to the north pole. 
    
    We start by partitioning the remaining points $X_N \setminus \{\mathbf{x}_1\}$ into disjoint subsets $S_k$ based on their geodesic distance from the north pole $\mathbf{x}_1$. Let $M := \lfloor \pi/q_{X_N}\rfloor$. We define
    \begin{align*}
    S_k &:= \{\mathbf{x} \in X_N \mid kq_{X_N} \leq \text{dist}(\mathbf{x}, \mathbf{x}_1) < (k+1)q_{X_N}\}, \quad 1 \leq k \leq M - 1, \\
    S_M &:= \{\mathbf{x} \in X_N \mid Mq_{X_N} \leq \text{dist}(\mathbf{x}, \mathbf{x}_1) \leq \pi\}.
    \end{align*}
    A standard packing argument \cite[Lemma~5]{kunis2007stability} demonstrates that the cardinalities of these sets satisfy $|S_k| \leq 25k$ for $k = 1, \ldots, M$. Splitting the summation according to the partition and using the localization property \eqref{eq:shifted-local}, we have
    \begin{align}
        \gamma_0 &= \sum_{i=2}^N |B_{\beta, n}(\langle \mathbf{x}_i, \mathbf{x}_1 \rangle)| 
       = \sum_{k=1}^{M} \sum_{\mathbf{x}_i \in S_k} |B_{\beta, n}(\langle \mathbf{x}_i, \mathbf{x}_1\rangle)| \notag \\
       &\leq \sum_{k=1}^M |S_k| \max_{\mathbf{x} \in S_k} |B_{\beta, n}(\langle \mathbf{x}, \mathbf{x}_1\rangle)| 
       \leq \sum_{k=1}^M 25k \, c_\beta' \, | 1+(n+1)kq_{X_N} |^{-\beta}. \label{eq:estimate-for-gamma}
\end{align}
    Under the assumptions $nq_{X_N} \geq 1$ and $\beta \geq 3$, the function $x \mapsto x/|1+(n+1)q_{X_N}x|^\beta$ is monotonically decreasing on $[1, \infty)$, so the sum (excluding the first term) can be bounded by comparison with its corresponding integral as the following:
    \begin{align*}
        \sum_{k=1}^M \frac{k}{|1+(n+1)kq_{X_N}|^\beta} & = \frac{1}{|1+(n+1)q_{X_N}|^\beta} + \sum_{k=2}^M \frac{k}{|1+(n+1)kq_{X_N}|^\beta} \\
        & \leq \frac{1}{((n+1)q_{X_N}))^\beta} + \int_0^\infty \frac{x}{|1+(n+1)q_{X_N}x|^\beta} \,dx \\
        & = \frac{1}{((n+1)q_{X_N})^\beta} + \frac{1}{(n+1)^2q_{X_N}^2(\beta-1)(\beta-2)}.
    \end{align*}
    Substituting this estimate in \eqref{eq:estimate-for-gamma} and applying it to \eqref{eq:gersh-bdd}, we obtain the desired $\gamma$ in \eqref{eq:gamma} and the proof is complete.
\end{proof}

Suppose $s > 1$. Let $0 < \lambda_1 \leq \lambda_2 \leq \cdots \leq \lambda_N$ denote the eigenvalues of $\mathbf{K}_s$. To bound the condition number $\mathrm{cond}(\mathbf{K}_s) = \lambda_N/\lambda_1$, we analyze the extremal eigenvalues separately. Let us begin with the largest eigenvalue $\lambda_N$.
\begin{proposition} \label{prop:lambda-N}
We have
    \begin{equation*}
     \frac{3}{4\pi}\left(\frac{4}{9}\right)^s  \leq \frac{1}{2\pi} \sum_{\ell=1}^\infty \left(\ell+\frac{1}{2}\right)^{1-2s} \leq \lambda_N \leq \frac{N}{2\pi} \sum_{\ell=1}^\infty \left(\ell+\frac{1}{2}\right)^{1-2s} \leq \frac{\pi}{4q_{X_N}^2} \frac{4^s}{s-1}.
\end{equation*}
\end{proposition}
\begin{proof}
    Since $P_\ell(1) = 1$ for all $\ell \in \mathbb{N}$, we have
    \begin{equation} \label{eq:trace-exp}
        \sum_{k=1}^N \lambda_k = \mathrm{Tr}(\mathbf{K}_s) = \sum_{k=1}^N K_s'(\mathbf{x}_k, \mathbf{x}_k) = \frac{N}{2\pi} \sum_{\ell=1}^\infty \left(\ell+\frac{1}{2}\right)^{1-2s}.
    \end{equation}
    Given that the eigenvalues are positive, the inequality $\lambda_N \leq \sum_{k=1}^N \lambda_k$ trivially establishes the upper bound. Conversely, the lower bound follows from the fact that the maximum eigenvalue must be at least as large as the mean of the eigenvalues, $\lambda_N \geq (\sum_{k=1}^N \lambda_k)/N$. The leftmost estimate is then obtained by retaining only the initial term ($\ell = 1$) of the resulting series.

    To derive the rightmost estimate, we first bound $N$ from above via a packing argument. Since interiors of the spherical caps $\{\mathcal{C}(\mathbf{x}_j, q_{X_N}/2)\}_{j=1}^N$ are pairwise disjoint, summing their areas (cf. \eqref{eq:area-sph-cap}) yields
    \begin{equation*}
        \sum_{j=1}^N |\mathcal{C}(\mathbf{x}_j, q_{X_N}/2)| = N \cdot 4
        \pi\sin^2(q_{X_N}/4) \leq |\mathbb{S}^2| = 4\pi.
    \end{equation*}
    We note that $0 \leq q_{X_N}/4 \leq \pi/4$. Applying the half-angle identity and Jordan's inequality, $\sin \theta \geq (2/\pi)\theta$ for $\theta \in [0, \pi/2]$, we find that
    \begin{equation} \label{eq:N-upper-bound}
        N \leq \frac{1}{\sin^2(q_{X_N}/4)} \leq \frac{4\pi^2}{q_{X_N}^2}.
    \end{equation}
    Finally, we bound the infinite series by its corresponding integral
    \begin{equation*}
        \sum_{\ell=1}^\infty \left(\ell + \frac{1}{2}\right)^{1-2s} \leq \int_0^\infty \left(x + \frac{1}{2}\right)^{1-2s} \,dx = \frac{2^{2s-3}}{s-1} .
    \end{equation*}
    Combining the bounds for $N$ and the series into \eqref{eq:trace-exp}, we obtain
    \begin{equation*}
        \lambda_N \leq \frac{1}{2\pi} \left( \frac{4\pi^2}{q_{X_N}^2} \right) \left( \frac{2^{2s-3}}{s-1} \right) = \frac{\pi}{4q_{X_N}^2} \frac{4^s}{s-1},
    \end{equation*}
    which completes the proof.
\end{proof}

To bound the smallest eigenvalue $\lambda_1$ from below, we employ the surrogate B-spline kernel argument introduced in \cite{narcowich1998stability}. Utilizing the refined localization estimates of the B-spline kernels \eqref{eq:shifted-local}, we are able to establish a tighter bound dependent on the separation distance $q_{X_N}$. 
\begin{proposition} \label{prop:lambda-1-lower-bound}
Denote $L:=\lceil 55/q_{X_N} \rceil$. It holds that
\begin{equation*}
    \lambda_1 \geq \frac{1}{32\pi} a_{L}^{(s)}.
\end{equation*}
\end{proposition}
\begin{proof}
    The choice of $L$ is made specifically so that 
    \begin{equation*}
        25c_3'\left[\frac{1}{((L+1)q_{X_N})^3} + \frac{1}{(L+1)^2q_{X_N}^2(3-1)(3-2)}\right] \leq \frac{25c_3'}{55^2}\left[\frac{1}{55}+\frac{1}{2}\right] \leq \frac{1}{2}.
    \end{equation*}
    Consequently, $B_{3,L}$ becomes a positive definite kernel for which 
    \begin{equation*}
        \sum_{i=1}^N \sum_{j=1}^N c_ic_j B_{3,L}(\langle \mathbf{x}_i, \mathbf{x}_j \rangle) \geq \frac{1}{2} \|\mathbf{c}\|_2^2, \quad \forall\, \mathbf{c} \in \mathbb{R}^N.
    \end{equation*}
    Recall from \cref{lem:local} that $B_{3, L}(t) = \sum_{\ell=0}^{L} \frac{2\ell+1}{4\pi} \alpha_\ell P_\ell(t)$ with $\alpha_\ell > 0$. Because $K_s'$ excludes the zero-degree harmonic, analogously we define
    \begin{equation*}
        B_{3, L}'(t) := \sum_{\ell=1}^{L} \frac{2\ell+1}{4\pi} \alpha_\ell P_\ell(t) = B_{3, L}(t) - \frac{\alpha_0}{4\pi}.
    \end{equation*}
    By the Cauchy--Schwarz inequality, the quadratic form associated with $B_{3, L}'$ can be bounded below as the following:
    \begin{equation*}
        \sum_{i=1}^N \sum_{j=1}^N c_ic_j B_{3,L}'(\langle \mathbf{x}_i, \mathbf{x}_j \rangle) \geq \frac{1}{2} \|\mathbf{c}\|_2^2 - \frac{\alpha_0}{4\pi}\left(\sum_{j=1}^N c_j\right)^2 \geq \left(\frac{1}{2} - \frac{N\alpha_0}{4\pi}\right) \|\mathbf{c}\|_2^2, \quad \forall\, \mathbf{c} \in \mathbb{R}^N.
    \end{equation*}
    It remains to bound $\alpha_0$. By definition, it holds that
    \begin{equation*}
        \alpha_0 = 2\pi \int_{-1}^1 B_{3,L}(t)\,dt = 2\pi \int_0^\pi B_{3,L}(\cos\theta)\sin\theta\,d\theta.
    \end{equation*}
    Utilizing the localization property \eqref{eq:shifted-local} and an elementary approximation to sine, we have
    \begin{equation*}
        \alpha_0 \leq 2\pi c_3'\int_0^\pi \frac{\theta}{(1 + (L + 1)\theta)^3}\,d\theta \leq \frac{2\pi c_3'}{(L+1)^2} \int_0^\infty \frac{u}{(1+u)^3}\,du = \frac{\pi c_3'}{(L+1)^2}.
    \end{equation*}
    Given $N \leq 4\pi^2/q_{X_N}^2$ (cf. \eqref{eq:N-upper-bound}) and $Lq_{X_N} \geq 55$, we can bound the amount of shift as
    \begin{equation*}
        \frac{N\alpha_0}{4\pi} \leq \frac{Nc_3'}{4(L+1)^2} \leq \frac{(4\pi^2/q_{X_N}^2)c_3'}{4(55/q_{X_N})^2} = \frac{\pi^2 c_3'}{55^2} \leq \frac{3}{8}.
    \end{equation*}
    Consequently, the quadratic form for $B_{3,L}$ is bounded below by
    \begin{equation*}
        \sum_{i=1}^N \sum_{j=1}^N c_ic_j B_{3,L}'(\langle \mathbf{x}_i, \mathbf{x}_j \rangle) \geq \left(\frac{1}{2} - \frac{3}{8}\right) \|\mathbf{c}\|_2^2 = \frac{1}{8} \|\mathbf{c}\|_2^2.
    \end{equation*}
    We now compare the coefficients $\alpha_\ell$ of $B_{3,L}'$ with $a_\ell^{(s)}$ of $K_s'$. Since $|P_\ell(t)| \leq 1$ and $|B_{3, L}(t)| \leq 1$ (cf. \eqref{eq:boundedness-of-B-spline-kernel}), for all $1 \leq \ell \leq L$, it holds that
    \begin{equation*}
        \alpha_\ell = 2\pi \int_{-1}^1 P_\ell(t) B_{3, L}(t) \, dt \leq 4\pi.
    \end{equation*}
    Therefore, $a_\ell^{(s)} \geq a_{L}^{(s)} \alpha_\ell/(4\pi)$ for all $1 \leq \ell \leq L$. It follows that
    \begin{equation} \label{eq:surrogate-step}
        \sum_{i=1}^N \sum_{j=1}^N c_ic_j K_s'(\mathbf{x}_i, \mathbf{x}_j) \geq \frac{a_{L}^{(s)}}{4\pi} \sum_{i=1}^N \sum_{j=1}^N c_ic_j B_{3,L}'(\langle \mathbf{x}_i, \mathbf{x}_j \rangle) \geq \frac{a_{L}^{(s)}}{32\pi} \|\mathbf{c}\|_2^2,
    \end{equation}
    which demonstrates a lower bound.
\end{proof}

\begin{remark} \label{rem:reduction}
    The original localization argument necessitates choosing a B-spline kernel of degree $L = O(q_{X_N}^{-2} \log q_{X_N}^{-2})$ \cite[(2.61)]{narcowich1998stability}, resulting in $\lambda_1 = \Omega((q_{X_N}^{-2} \log q_{X_N}^{-2})^{-2s})$. In contrast, our refined analysis reduces the required degree to $L = O(q_{X_N}^{-1})$, yielding the tighter lower bound $\lambda_1 = \Omega(q_{X_N}^{2s})$.
\end{remark}

To complete the spectral estimates, we also need to establish an upper bound for $\lambda_1$.
\begin{proposition} \label{prop:lambda-1-upper-bound}
Suppose $0 < h_{X_N} \leq 2$. Denote $H := \lfloor 2h_{X_N}^{-1} \rfloor \geq 1$. It holds that
\begin{equation*}
    \lambda_1 \leq \frac{6\pi}{q_{X_N}^3}a_{H}^{(s/2)} \frac{2^s}{s-1}.
\end{equation*}
\end{proposition}
\begin{proof}
    We start by noticing that the spherical caps $\{\mathcal{C}(\mathbf{x}_j, h_{X_N})\}_{j=1}^N$ cover $\mathbb{S}^2$, so that by \eqref{eq:area-sph-cap} we have
    \begin{equation*} 
        \sum_{j=1}^N |\mathcal{C}(\mathbf{x}_j, h_{X_N})| = N \cdot 4\pi \sin^2(h_{X_N}/2) \geq |\mathbb{S}^2| = 4\pi.
    \end{equation*}
    Since $\sin\theta < \theta$ for all $\theta \in (0, \pi]$ and $h_{X_N} > 0$, the preceding inequality implies $N > 4/h_{X_N}^2$. Consequently, as $\mathrm{dim}(\mathbb{P}_{H-1}(\mathbb{S}^2)) = H^2 \leq 4/h_{X_N}^2 < N$, a standard dimensionality argument implies the existence of a nonzero vector $\mathbf{c}^* \in \mathbb{R}^N$ that annihilates all harmonics of degree less than $H$. That is, we have
    \begin{equation} \label{eq:annihilation}
        \sum_{j=1}^N c_j^* Y_{\ell, k}(\mathbf{x}_j) = 0 \quad \forall\, \ell = 0, 1, \ldots, H -1 , \quad |k| \leq \ell.
    \end{equation}
    In particular, since $H > 0$, $\mathbf{c}^*$ annihilates the zero-degree harmonic, i.e., $\sum_{j=1}^N c_j^* = 0$, so that the quadratic form of $K_s'$ evaluates identically to that of $K_s$ with $\mathbf{c}^*$. Using also the addition theorem \eqref{eq:addition-theorem} and the annihilation properties of $\mathbf{c}^*$ \eqref{eq:annihilation} yields
    \begin{equation*}
        \sum_{i=1}^N \sum_{j=1}^N c_i^* c_j^* K_s'(\mathbf{x}_i, \mathbf{x}_j) = \sum_{i=1}^N \sum_{j=1}^N c_i^* c_j^* K_s(\mathbf{x}_i, \mathbf{x}_j) = \sum_{\ell=H}^\infty a_\ell^{(s)} \sum_{k=-\ell}^\ell \left( \sum_{j=1}^N c_j^* Y_{\ell,k}(\mathbf{x}_j) \right)^2.
    \end{equation*}
    Because $s > 1$, we can define $t:=(s+1)/2$, which satisfies $1 < t < s$. For all $\ell \geq H$, we have
    \begin{equation*}
        a_\ell^{(s)} = \left(\ell+\frac{1}{2}\right)^{-(s-1)}\left(\ell+\frac{1}{2}\right)^{-(s+1)} \leq \left(H + \frac{1}{2}\right)^{-(s-1)}a_\ell^{(t)} = \left(H + \frac{1}{2}\right)a_{H}^{(s/2)}a_\ell^{(t)}.
    \end{equation*}
    Applying this bound and the annihilation properties of $\mathbf{c}^*$ \eqref{eq:annihilation} again, we obtain
    \begin{align*}
        \sum_{i=1}^N\sum_{j=1}^N c_i^*c_j^*K_s'(\mathbf{x}_i, \mathbf{x}_j) & \leq \sum_{\ell=H}^\infty \left(H+\frac{1}{2}\right) a_H^{(s/2)} a_\ell^{(t)} \sum_{k=-\ell}^\ell \left(\sum_{j=1}^N c_j^* Y_{\ell,k}(\mathbf{x}_j)\right)^2 \\
        & = \left(H + \frac{1}{2}\right)a_{H}^{(s/2)} \sum_{i=1}^N \sum_{j=1}^N c_i^*c_j^*K_t'(\mathbf{x}_i, \mathbf{x}_j).
    \end{align*}
    Now, we have
    \begin{equation*}
        \lambda_1(\mathbf{K}_s) \leq \frac{(\mathbf{c}^*)^\top \mathbf{K}_s\mathbf{c}^*}{\|\mathbf{c}^*\|_2^2} \leq \left(H + \frac{1}{2}\right)a_{H}^{(s/2)}\frac{(\mathbf{c}^*)^\top \mathbf{K}_t\mathbf{c}^*}{\|\mathbf{c}^*\|_2^2} \leq \left(H + \frac{1}{2}\right)a_{H}^{(s/2)}\lambda_N(\mathbf{K}_t).
    \end{equation*}
    From \cref{prop:lambda-N}, we know
    \begin{equation*}
        \lambda_N(\mathbf{K}_t) \leq \frac{\pi}{4q_{X_N}^2}\frac{4^t}{t-1} =\frac{\pi}{q_{X_N}^2}\frac{2^s}{s-1}.
    \end{equation*}
    By definition, we have $2h_{X_N} \geq q_{X_N}$. With $q_{X_N} \leq \pi < 4$, we can then bound 
    $$
    H + 1/2 \leq 2/h_{X_N} + 1/2 \leq 4/q_{X_N} + 2/q_{X_N} = 6/q_{X_N}.
    $$
    Combining these estimates yields
    \begin{equation*}
        \lambda_1(\mathbf{K}_s) \leq \frac{6}{q_{X_N}} a_{H}^{(s/2)} \left(\frac{\pi}{q_{X_N}^2} \frac{2^s}{s-1}\right) = \frac{6\pi}{q_{X_N}^3} a_{H}^{(s/2)} \frac{2^s}{s-1},
    \end{equation*}
    which completes the proof.
\end{proof}

Piecing these extremal eigenvalue estimates together, we obtain a comprehensive bound for $\mathrm{cond}(\mathbf{K}_s)$ in the smooth regime ($s > 1$).
\begin{theorem} \label{thm:conditioning}
    Suppose $0 < h_{X_N} \leq 2$. Define $L := \lceil 55/q_{X_N} \rceil$ and $H := \lfloor 2/h_{X_N}\rfloor$. For $s > 1$, the condition number of the discrepancy matrix $\mathbf{K}_s$ satisfies
    \begin{equation*}
        \frac{q_{X_N}^3(s-1)}{8\pi^2}\left(\frac{2H+1}{9}\right)^s 
        \leq \mathrm{cond}(\mathbf{K}_s) 
        \leq \frac{8\pi^2}{q_{X_N}^2(s-1)}(2L+1)^{2s}.
    \end{equation*}
\end{theorem}

\begin{proof}
    Recall that the condition number is defined by $\mathrm{cond}(\mathbf{K}_s) = \lambda_N/\lambda_1$. Bounding the numerator from below (cf. \cref{prop:lambda-N}) and the denominator from above (cf. \cref{prop:lambda-1-upper-bound}) gives
    \begin{equation*}
        \mathrm{cond}(\mathbf{K}_s) \geq \dfrac{\dfrac{3}{4\pi}\left(\dfrac{4}{9}\right)^s}{\dfrac{6\pi}{q_{X_N}^3}\left(H+\dfrac{1}{2}\right)^{-s} \dfrac{2^s}{s-1}} = \dfrac{q_{X_N}^3 (s-1)}{8\pi^2} \left(\dfrac{2H+1}{9}\right)^s.
    \end{equation*}
     Bounding the numerator from above (cf. \cref{prop:lambda-N}) and the denominator from below (cf. \cref{prop:lambda-1-lower-bound}) gives
    \begin{equation*}
        \mathrm{cond}(\mathbf{K}_s) \leq \frac{\dfrac{\pi}{4q_{X_N}^2} \dfrac{4^s}{s-1}}{\dfrac{1}{32\pi}\left(L+\dfrac{1}{2}\right)^{-2s}} = \frac{8\pi^2}{q_{X_N}^2(s-1)}(2L+1)^{2s}.
    \end{equation*}
The proof is complete.
\end{proof}

This theorem reveals that the condition number grows exponentially as the assumed smoothness $s \to \infty$.
\begin{remark}[Exponential Growth of Conditioning] \label{rem:exponential-conditioning}
    While the existing literature (e.g., \cite{narcowich1998stability,levesley1999norm}) has focused primarily on upper bounds for the condition number, the upper bound for the smallest eigenvalue in \cref{prop:lambda-1-upper-bound} and the lower bound for the condition number in \cref{thm:conditioning} appear to be less commonly addressed. When $h_{X_N} < 1/2$ (as is typically the case when $N$ is large), it holds that $(2H+1)/9 > 1$. For any such $X_N$, the conditioning grows exponentially with $s$
    \begin{equation*}
        \lim_{s \to \infty} \mathrm{cond}(\mathbf{K}_s) \geq \lim_{s \to \infty} \frac{q_{X_N}^3(s-1)}{24\pi^2} \left(\frac{2H+1}{9}\right)^s = \infty.
    \end{equation*}
    This suggests that $s$ should be chosen moderately to avoid ill-conditioning for practical computation.
\end{remark}

\subsection{Computation Aspects} \label{sec:computation-realization-kernel}
To implement the proposed weight collocation models, we need to compute the entries of the matrix $\mathbf{K}_s$. The entries of $\mathbf{K}_s$ are defined by infinite series and, in general, cannot be evaluated in closed forms. In this section, we present two computational strategies. First, we exploit closed-form kernel expressions whenever they are available. Otherwise, including in the low-smoothness regime ($0 \leq s \leq 1$), we approximate the kernel via series truncation. We discuss two approaches for choosing the truncation degree $L$: a theoretically rigorous threshold that guarantees positive definiteness (cf. \Cref{sec:high-truncation-level}) and a bandlimited heuristic designed for highly clustered data (cf. \Cref{sec:bandlimited}).

\subsubsection{Closed-Form Expressions} \label{sec:closed_form}
The infinite series \eqref{eq:reproducing-kernel-sobolev-trunc} does not generally admit a closed-form expression. However, when such an expression exists (typically through a specific choice of the Sobolev norm \eqref{eq:sobolev-norm-constant-alt}), we can gain significant computational advantages in evaluating $\mathbf{K}_s$. A notable example is the \textit{Cui and Freeden kernel} ($K_\text{CF}$) in \cite{cui1997equidistribution}, formulated for $\mathbb{H}^{3/2}(\mathbb{S}^2)$. By reweighting
\begin{align*}
  a_\ell^{(3/2)} := \left\{
  \begin{array}{ll}
     4\pi, & \text{if $\ell = 0$} \\
     4\pi/[(2\ell+1)\ell(\ell+1)], & \text{if $\ell > 0$}
  \end{array}\right. \asymp \left(\ell + \frac{1}{2}\right)^{-3},
\end{align*}
$K_\text{CF}$ admits a closed-form logarithmic expression
\begin{equation*}
  K_{\text{CF}}(\mathbf{x}, \mathbf{y}) := 1 + \sum_{\ell = 1}^\infty \frac{1}{\ell(\ell+1)}P_\ell(\langle \mathbf{x}, \mathbf{y}\rangle) =  2 - 2\log \left(1 + \sqrt{\frac{1 - \langle \mathbf{x}, \mathbf{y} \rangle}{2}}\right).
\end{equation*}
Excluding the zero-degree harmonic from $K_\text{CF}$, we obtain
\begin{equation*}
  K_{\text{CF}}'(\mathbf{x}, \mathbf{y}) := \sum_{\ell = 1}^\infty \frac{1}{\ell(\ell+1)}P_\ell(\langle \mathbf{x}, \mathbf{y}\rangle) = 1 - 2\log \left(1 + \sqrt{\frac{1 - \langle \mathbf{x}, \mathbf{y} \rangle}{2}}\right).
\end{equation*}
Consequently, we can formulate the following \textit{discrepancy collocation} model:
\begin{align} \label{eq:cui-freeden-discrepancy-collocation}
  \begin{aligned}
    \min \quad & \mathbf{w}^\top \mathbf{K}_{\text{CF}} \mathbf{w} \\
    \text{s.t.} \quad & \mathbf{Y}_{n^+}\mathbf{w} = \mathbf{b}_{n^+}, \quad \mathbf{w} \in \mathbb{R}^N,
  \end{aligned}
\end{align}
with $\mathbf{K}_\text{CF} := [K_\text{CF}'(\mathbf{x}_i, \mathbf{x}_j)]_{i, j = 1}^N$. Other reproducing kernels also admit closed-form expressions for specific smoothness indices; we summarize several examples in \cref{tab:closed-form-kernels}. 

\begin{table}[htbp]
    \centering
    \caption{Some closed-form expressions of $K_s'(\mathbf{x}, \mathbf{y})$ for Sobolev spaces $\mathbb{H}^s(\mathbb{S}^2)$ across varying smoothness indices $s$. Here, $Q_k(z) := \int_0^1 (1-x)^k (1-2xz+x^2)^{-1/2}\, dx$.
    }
    \label{tab:closed-form-kernels}
    \begin{adjustbox}{max width=\textwidth}
    \begin{tabular}{@{}lccc@{}}
        \toprule
        Name & $s$ & $a_\ell^{(s)}$ ($\ell > 0$) & $K_s'(\mathbf{x}, \mathbf{y})$  \\
        \midrule
        Cui and Freeden \cite{cui1997equidistribution} 
        & $s=3/2$ 
        & $\dfrac{4\pi}{(2\ell+1)\ell(\ell+1)}$ 
        & $1 - 2\log\left(1 + \sqrt{\dfrac{1-\langle\mathbf{x},\mathbf{y}\rangle}{2}}\right)$ \\[15pt]
        Generalized distance \cite[Sec.~5]{brauchart2014qmc} 
        & $1<s<2$ 
        & $-\dfrac{2^{2s}\pi}{s}\dfrac{(1-s)_\ell}{(1+s)_\ell}$ 
        & $\dfrac{2^{2s-2}}{s} - \operatorname{dist}(\mathbf{x},\mathbf{y})^{2s-2}$ \\[15pt]
        Wahba \cite{wahba1981spline} 
        & $s \in \mathbb{N}_+/2$ 
        & $\dfrac{1}{(\ell+1/2)(\ell+1)_{2s-1}}$ 
        & $\dfrac{1}{2\pi}\left[\dfrac{Q_{2s-2}(\langle\mathbf{x},\mathbf{y}\rangle)}{(2s-2)!} - \dfrac{1}{(2s-1)!}\right]$ \\
        \bottomrule
    \end{tabular}
    \end{adjustbox}
\end{table}

\subsubsection{Series Truncation} \label{sec:series-trunc}

When a closed-form expression is unavailable, we approximate $K_s'$ by truncating its Fourier-Legendre expansion at a finite degree $L$
\begin{equation} \label{eq:truncated-kernel}
  K_s'(\mathbf{x}, \mathbf{y}) \approx K_s^L(\mathbf{x}, \mathbf{y}) := \sum_{\ell=1}^L \frac{2\ell+1}{4\pi} a_\ell^{(s)} P_\ell(\langle \mathbf{x}, \mathbf{y}\rangle).
\end{equation}
As in the full kernel, the $\ell = 0$ harmonic is omitted since constant functions are assumed to be integrated exactly. Replacing $K_s'$ by $K_s^L$ in the discrepancy collocation yields the \textit{truncated discrepancy collocation}
\begin{align} \label{eq:truncated-kernel-collocation}
  \begin{aligned}
      \min \quad & \mathbf{w}^\top \mathbf{K}_s^L \mathbf{w} \\
      \text{s.t.} \quad & \mathbf{Y}_{n^+} \mathbf{w} = \mathbf{b}_{n^+}, \quad \mathbf{w} \in \mathbb{R}^N,
  \end{aligned}
\end{align}
where $\mathbf{K}_s^L := [K_s^L(\mathbf{x}_i, \mathbf{x}_j)]_{i,j=1}^N$ is the \textit{truncated discrepancy matrix}.

The truncated kernel $K_s^L$ is the reproducing kernel of the \textit{bandlimited Sobolev space} with bandwidth $L$
\begin{equation*}
    \mathbb{H}_L^s(\mathbb{S}^2) := \{f \in \mathbb{H}^s(\mathbb{S}^2)\,|\,\hat{f}_{\ell, k} = 0 \;\text{for}\;\ell > L\}.
\end{equation*}
Consequently, minimizing \eqref{eq:truncated-kernel-collocation} is equivalent to minimizing the worst-case integration error over the unit ball of $\mathbb{H}^s_L(\mathbb{S}^2)$. Applying the addition theorem \eqref{eq:addition-theorem} and using $I(Y_{\ell, k}) = 0$ for all $\ell \geq 1$, the quadratic form admits the decomposition
\begin{align} \label{eq:truncated-kernel-collocation-decomposition}
\begin{aligned}
    \mathbf{w}^\top \mathbf{K}_s^L \mathbf{w} & = \sum_{\ell=1}^L a_\ell^{(s)} \sum_{k=-\ell}^\ell \left| Q[X_N, \mathbf{w}](Y_{\ell, k}) - I(Y_{\ell, k}) \right|^2,
\end{aligned}
\end{align}
revealing the objective as a weighted sum of squared quadrature residuals for spherical harmonics. Although $\mathbb{H}_L^s(\mathbb{S}^2)$ coincides algebraically with the polynomial space $\mathbb{P}_L(\mathbb{S}^2)$, the two spaces are equipped with different inner products. The Sobolev inner product weights the residuals of harmonics in \eqref{eq:truncated-kernel-collocation-decomposition} using $a_\ell^{(s)} = (\ell+1/2)^{-2s}$, thereby penalizing low-frequency residuals more heavily than high-frequency ones. This frequency-dependent weighting reflects the regularity encoded by the Sobolev space $\mathbb{H}^s(\mathbb{S}^2)$, rather than treating all harmonic degrees equally. Finally, because $K_s^L$ is a finite polynomial sum, it is well defined for every $s \geq 0$. The truncated discrepancy collocation \eqref{eq:truncated-kernel-collocation} therefore extends naturally to the low-smoothness regime $0 \leq s \leq 1$, where the full Sobolev space no longer admits a reproducing kernel.

\begin{proposition} \label{prop:positive-definiteness-truncated-discrepancy-matrix}
    Suppose $s \geq 0$. For any $L \geq 1$, the truncated discrepancy matrix $\mathbf{K}_s^L$ is symmetric positive semidefinite.
\end{proposition}
\begin{proof}
    The symmetry is trivial in view of the definition \eqref{eq:truncated-kernel}. Note that
    \begin{equation*}
        K_s^L(\mathbf{x}, \mathbf{y}) = \sum_{\ell=0}^{\infty} \frac{2\ell+1}{4\pi} \tilde{a}_\ell^{(s)}P_\ell(\langle \mathbf{x}, \mathbf{y} \rangle),
    \end{equation*}
    where
    \begin{equation*}
        \tilde{a}_\ell^{(s)} := \left\{
            \begin{array}{ll}
               0, & \text{if $\ell = 0$ or $\ell > L$}, \\
               (\ell+1/2)^{-2s}, & \text{if $1 \leq \ell \leq L$}. 
            \end{array}
        \right.
    \end{equation*}
    Since $\tilde{a}_\ell^{(s)} \ge 0$ for all $\ell$, the characterization in \cite{schoenberg1942positive} implies that $K_s^L$ is a positive semidefinite kernel on $\mathbb{S}^2$. Hence, $\mathbf{K}_s^L$ is positive semidefinite.
\end{proof}

\cref{prop:positive-definiteness-truncated-discrepancy-matrix} shows that $\mathbf{K}_s^L$ is only positive \textit{semidefinite} and thus \eqref{eq:truncated-kernel-collocation} may have multiple solutions. This is indeed the primary theoretical challenge introduced by truncation. Then, we need to carefully select the bandwidth $L$. We present two distinct regimes: a high truncation level for theoretical definiteness, and a low truncation level for robust geometric recovery.

\subsubsection{High Truncation Level and Positive Definiteness} \label{sec:high-truncation-level}
To ensure the positive definiteness of $\mathbf{K}_s^L$, we use the spectral analysis from \Cref{sec:spectral-stability}. By setting the truncation degree to the threshold in \cref{prop:lambda-1-lower-bound}, we obtain strict convexity of the optimization model.

\begin{theorem}\label{thm:conditioning-truncation}
    Suppose $s \geq 0$. If the truncation level is chosen as $L := \lceil 55/q_{X_N} \rceil$, then $\mathbf{K}_s^L$ is symmetric positive definite. Furthermore, the smallest eigenvalue satisfies $\lambda_1 \geq a_L^{(s)}/(32\pi)$, and the condition number is bounded above by
    \begin{equation*}
        \mathrm{cond}(\mathbf{K}_s^L) \leq \frac{32\pi^2}{q_{X_N}^2} \frac{\sum_{\ell=1}^L (2\ell+1)a_\ell^{(s)}}{a_L^{(s)}}.
    \end{equation*}
\end{theorem}
\begin{proof}
    To establish the positive definiteness, it follows from  \eqref{eq:surrogate-step} that the quadratic form of $K_s^L$ is sandwiched by those of $K_s'$ and $B_{3,L}'$ in sense of
    \begin{equation*}
        \sum_{i=1}^N \sum_{j=1}^N c_ic_j K_s'(\mathbf{x}_i, \mathbf{x}_j) \geq \sum_{i=1}^N\sum_{j=1}^N c_ic_j K_s^L(\mathbf{x}_i, \mathbf{x}_j) \geq \frac{a_{L}^{(s)}}{4\pi} \sum_{i=1}^N \sum_{j=1}^N c_ic_j B_{3,L}'(\langle \mathbf{x}_i, \mathbf{x}_j \rangle).
    \end{equation*}
    Therefore, $\mathbf{K}_s^L$ shares the same lower bound on its smallest eigenvalue with $\mathbf{K}_s$ in \cref{prop:lambda-1-lower-bound}: $\lambda_1 \geq a_L^{(s)}/(32\pi)$.

    To bound the condition number from above, we estimate the largest eigenvalue via the trace
    \begin{equation*}
        \lambda_N \leq \mathrm{Tr}(\mathbf{K}_s^L) = \sum_{k=1}^N K_s^L(\mathbf{x}_k, \mathbf{x}_k) = N \sum_{\ell=1}^L \frac{2\ell+1}{4\pi} a_\ell^{(s)} P_\ell(1) = N \sum_{\ell=1}^L \frac{2\ell+1}{4\pi} a_\ell^{(s)}.
    \end{equation*}
    Using $N \leq 4\pi^2/q_{X_N}^2$ (cf. \eqref{eq:N-upper-bound}), we can bound
    \begin{equation*}
        \mathrm{cond}(\mathbf{K}_s^L) \leq \dfrac{N \displaystyle\sum_{\ell=1}^L \dfrac{2\ell+1}{4\pi} a_\ell^{(s)}}{\dfrac{1}{32\pi}a_L^{(s)}} \leq \frac{32\pi^2}{q_{X_N}^2} \frac{\sum_{\ell=1}^L (2\ell+1)a_\ell^{(s)}}{a_L^{(s)}},
    \end{equation*}
and the proof is complete.
\end{proof}

With this theoretically justified choice of $L$, the truncated discrepancy matrix $\mathbf{K}_s^L$ becomes strictly positive definite. Consequently, the optimization problem \eqref{eq:truncated-kernel-collocation} admits a unique optimal solution.

\begin{remark}[Evaluation of $\mathbf{K}_s^L$]

    To evaluate the expansion in \eqref{eq:truncated-kernel}, we employ Clenshaw summation (see \cite{clenshaw1955note} and \cite[Section~5.4]{press2007numerical}), which is numerically stable and efficient for Legendre series. This avoids the underflow, overflow, and accumulated round-off errors that can arise from evaluating the individual Legendre polynomials separately. Moreover, the reduction of the truncation degree to $L = O(q_{X_N}^{-1})$ discussed in \cref{rem:reduction} is essential for practical computation. The assembling of $\mathbf{K}_s^L$ using Clenshaw summation takes $O(N^2L)$ operations. Reducing $L$ therefore directly lowers the computational complexity and makes the truncated discrepancy collocation \eqref{eq:truncated-kernel-collocation} practical for larger point sets.
\end{remark}

\subsubsection{Low Truncation Level and Bandlimited Collocation} \label{sec:bandlimited}
When the scattered nodes are highly clustered, the separation distance $q_{X_N}$ can be very small, making the theoretically justified truncation level $L = \lceil 55/q_{X_N} \rceil$ prohibitively large. To address this difficulty, we can choose a much smaller bandwidth
\begin{equation*}
    n^+ < L := \lceil \gamma N \rceil \ll \lceil 55/q_{X_N} \rceil
\end{equation*}
where $\gamma > 0$ is a relaxation parameter. Rather than approximating the original kernel at this bandwidth, we view the problem as recovering low-frequency spherical harmonics. Motivated by the residual representation \eqref{eq:truncated-kernel-collocation-decomposition}, we replace the quadratic penalty on the harmonic residuals with its $\ell_1$-norm and incorporate a strictly convex regularizer $R(\mathbf{w})$ with strength $\lambda > 0$ to stabilize the solution. This yields the \textit{bandlimited collocation}
\begin{align} \label{eq:bandlimited-collocation}
  \begin{aligned}
    \min \quad & \|\bm{\Gamma}_s (\mathbf{Y}_L \mathbf{w} - \mathbf{b}_L) \|_1 + \lambda  R(\mathbf{w}) \\
    \text{s.t.} \quad & \mathbf{Y}_{n^+} \mathbf{w} = \mathbf{b}_{n^+}, \quad \mathbf{w} \in \mathbb{R}^N,
  \end{aligned}
\end{align}
where 
\begin{equation*}
    \bm{\Gamma}_s = \mathrm{diag}\left( \sqrt{a_\ell^{(s)}}\,\mathbf{I}_{2\ell+1}\right)_{\ell=0}^L
\end{equation*}
is a diagonal matrix and $\mathbf{I}_k \in \mathbb{R}^{k \times k}$ denotes the identity matrix. The choice of $R(\mathbf{w})$ could be flexible. For instances, one can utilize the simple $\ell_2$ norm or the geometry-aware regularizer (see \Cref{sec:perturbation}).

Replacing the quadratic penalty with an $\ell_1$-norm improves robustness to unresolved spherical harmonics. Rather than enforcing near-exact integration of all harmonics up to degree $L$, the $\ell_1$ objective permits the residual to concentrate on poorly resolved harmonics in the range $n^+ < \ell \leq L$. This prevents the optimizer from introducing large oscillations in the quadrature weights while maintaining accurate integration of the well-resolved low-frequency components, as illustrated in \cref{ex:bandlimited-duality}.

\begin{figure}[htbp]
\centering
\begin{subfigure}[t]{0.47\textwidth}
    \centering
    \includegraphics[width=\textwidth]{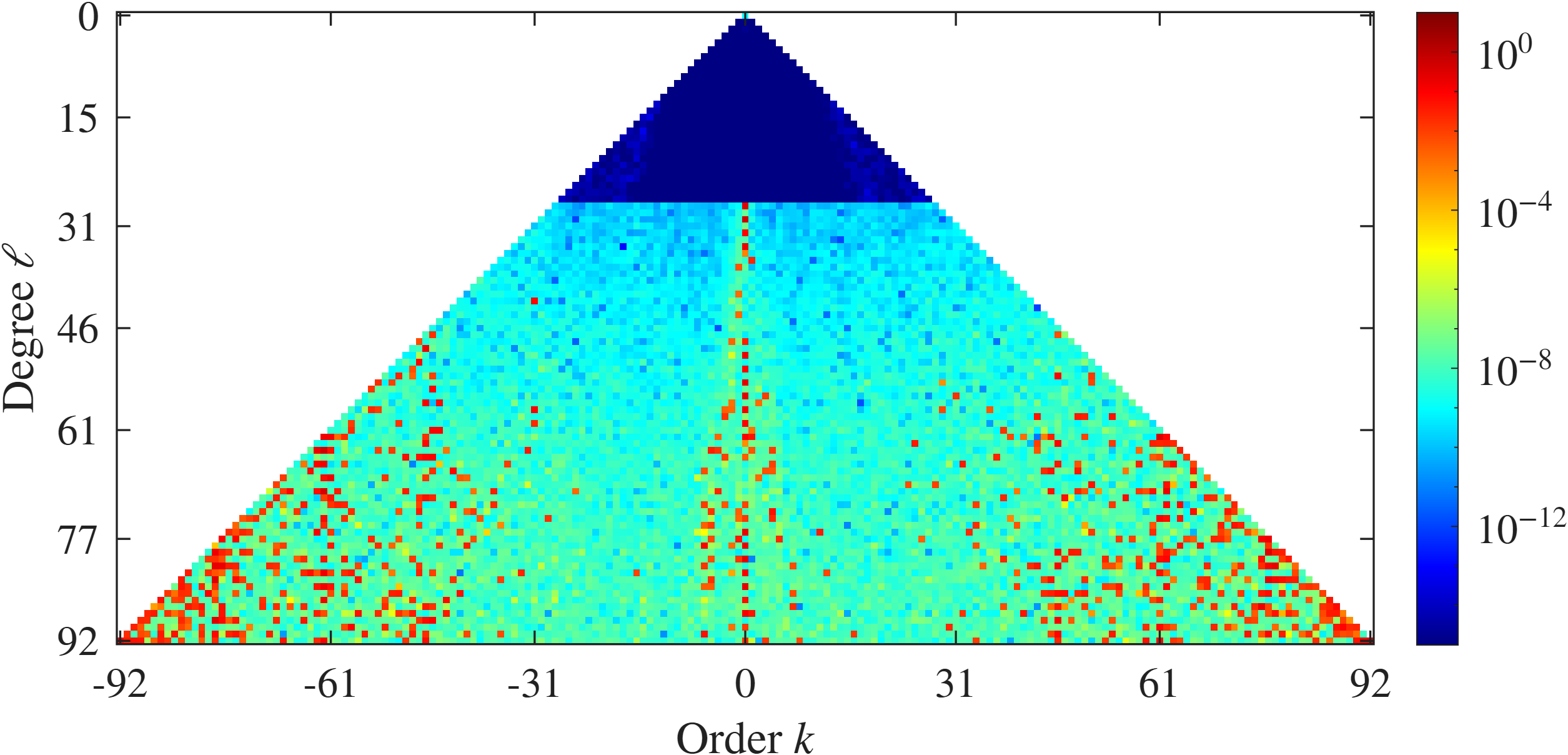}
    \caption{On 10,443 MAGSAT points.}
    \label{fig:magsat-l1}
\end{subfigure}%
~\hfill~
\begin{subfigure}[t]{0.47\textwidth}
    \centering
    \includegraphics[width=\textwidth]{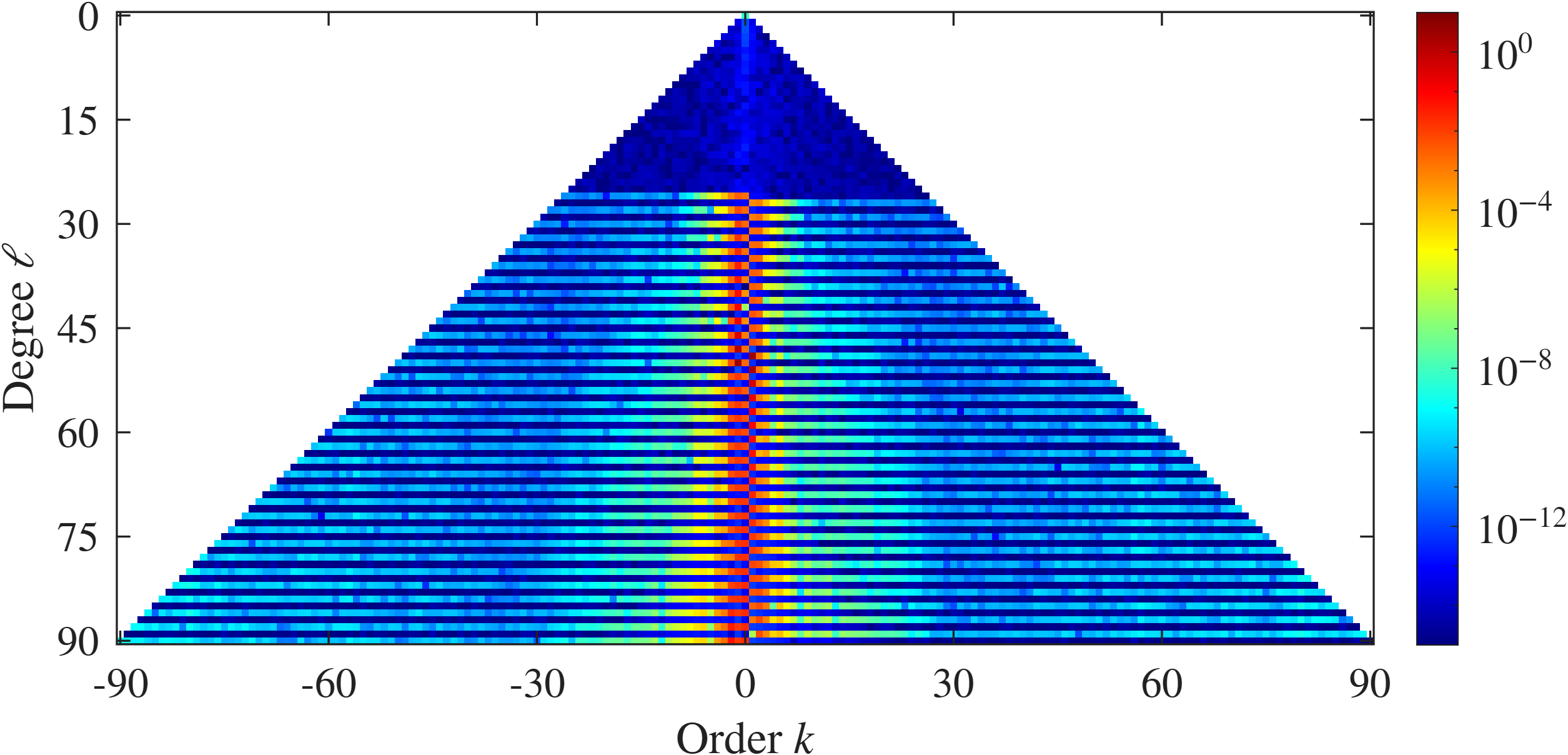}
    \caption{On 10,000 spiral points.} 
    \label{fig:spiral-l1}
\end{subfigure}
\caption{Absolute integration residuals of spherical harmonics $r_{\ell, k}$ for weights obtained via bandlimited collocation \eqref{eq:bandlimited-collocation}, plotted by degree $\ell$ (vertical axis) and order $k \in [-\ell, \ell]$ (horizontal axis).}
\label{fig:feature-selection}
\end{figure}

\begin{example} \label{ex:bandlimited-duality}
We consider the bandlimited collocation \eqref{eq:bandlimited-collocation} with smoothness parameter $s = 2.5$, quadratic regularizer $R(\mathbf{w}) = \frac{1}{2}\|\mathbf{w}\|_2^2$, regularization strength $\lambda = 0.01$, and relaxation parameter $\gamma = 0.9$. The optimization is solved separately for two node sets: the $10,443$-points MAGSAT data set and a synthetic set of $10,000$ spiral points\footnote{The spiral points are generated by evaluating the parametric curve $\mathbf{x}(t) = (\sin t \sin 50t, \sin t \cos 50t, \cos t)^\top$ at $10,000$ uniformly spaced intervals $t_j = (j-1)\pi/(N-1)$ for $j = 1, \ldots, 10,000$.}. We compare the resulting quadrature rules by examining the absolute integration residuals of the spherical harmonics, $$r_{\ell, k} := |Q[X_N,\mathbf{w}](Y_{\ell, k}) - I(Y_{\ell, k})|.$$

\cref{fig:feature-selection} shows the harmonic residuals for the two resulting quadrature rules. These two node sets possess complementary geometric limitations. Because the MAGSAT data arise from polar orbits, they contain large longitudinal gaps near the equator and may provide average-quality resolution of \textit{sectorial harmonics} ($|k| \approx \ell$). Accordingly, \cref{fig:magsat-l1} shows the concentration of residuals on these  resolved harmonics. In contrast, the spiral points provide uniform azimuthal coverage but exhibits substantial latitudinal gaps between successive coils, limiting its ability to resolve \textit{zonal harmonics} ($|k| \approx 0$). \cref{fig:spiral-l1} likewise shows that the residuals are concentrated on these unresolved harmonics. In both cases, the residual patterns faithfully reflect the geometric limitation of the underlying node set rather than attempting to enforce high accuracy on harmonics that are beyond the geometric explanations. 
\end{example}

\subsection{Revisiting \texorpdfstring{$\ell_2$--}{l2-}Minimization} \label{sec:asymptotic-recovery}
While \cref{rem:exponential-conditioning} establishes the exponential growth of the condition number as $s \to \infty$, the opposite limit $s \to 0^+$ provides a reinterpretation of the classical $\ell_2$-minimization. 

Recall the Fourier--Legendre expansion of the Dirac delta distribution on $\mathbb{S}^2$ \cite[(3.56)] {jackson1998classical} with the $\ell = 0$ term isolated
\begin{equation*}
 \delta(\mathbf{x}, \mathbf{y}) = \frac{1}{4\pi} + \sum_{\ell=1}^\infty \frac{2\ell+1}{4\pi} P_\ell(\langle \mathbf{x}, \mathbf{y}\rangle), \quad \mathbf{x}, \mathbf{y} \in \mathbb{S}^2.
\end{equation*}
As $s \to 0^+$, the Sobolev weights satisfy $a_\ell^{(s)} \to 1$. Consequently, as $L \to \infty$, the truncated kernel $K_s^L(\mathbf{x}, \mathbf{y})$ converges in sense of distributions to $\delta(\mathbf{x}, \mathbf{y}) - 1/(4\pi)$. It follows that, for $i \neq j$, we have $K_s^L(\mathbf{x}_i, \mathbf{x}_j) \to -1/(4\pi)$. Since $P_\ell(1) = 1$, every diagonal entry of the truncated discrepancy matrix is given by
\begin{equation*}
    K_s^L(\mathbf{x}_i, \mathbf{x}_i) = c_{L, s}, \quad c_{L, s} := \sum_{\ell=1}^L \frac{2\ell+1}{4\pi} a_\ell^{(s)},
\end{equation*}
which diverges as $s \to 0^+$ and $L \to \infty$. Therefore, the objective in \eqref{eq:truncated-kernel-collocation} is asymptotically dominated by
\begin{equation*}
    c_{L, s} \|\mathbf{w}\|_2^2 - \frac{1}{4\pi} (\mathbf{1}^\top \mathbf{w})^2.
\end{equation*}
Since the quadrature integrates constants exactly (cf. \eqref{eq:exact-for-const}), the second term is constant and therefore does not affect the optimization. Hence, in the limit $s \to 0^+$ and $L \to \infty$, the truncated discrepancy collocation \eqref{eq:truncated-kernel-collocation} is asymptotically equivalent to solving
\begin{align} \label{eq:l2-minimization}
    \begin{aligned}
    \min \quad & \|\mathbf{w}\|_2^2 \\
    \text{s.t.} \quad & \mathbf{Y}_{n^+} \mathbf{w} = \mathbf{b}_{n^+}, \quad \mathbf{w} \in \mathbb{R}^N,
    \end{aligned}
\end{align}
which is precisely the classical $\ell_2$-minimization. 

While the formulation \eqref{eq:l2-minimization} is often viewed as a heuristic for promoting weight uniformity, the above interpretation shows that it also arises as the limiting case of discrepancy minimization when the bandwidth tends to infinity and the underlying Sobolev space degenerates to $L^2(\mathbb{S}^2)$.

\section{Marcinkiewicz\texorpdfstring{--}{-}Zygmund Collocation} \label{sec:MZ}
While the kernel collocation in \Cref{sec:kernel} focuses on minimizing the discrepancy of a quadrature for numerical integration, our second collocation methodology targets the performance of function approximation, specifically the hyperinterpolation.

A key tool for analyzing hyperinterpolation is the Marcinkiewicz--Zygmund (MZ) inequality \cite{mhaskar2001spherical}. The constants in these inequalities are closely related to the stability and accuracy of hyperinterpolation. By shifting the emphasis from algebraic exactness to the optimization of these MZ constants, we construct quadrature weights that promote stable and efficient hyperinterpolation on arbitrarily scattered data.

\subsection{The Role of Marcinkiewicz\texorpdfstring{--}{-}Zygmund Constants} \label{sec:MZ-role}
\begin{definition}
    Let $X_N = \{\mathbf{x}_1, \ldots, \mathbf{x}_N\} \subseteq \mathbb{S}^2$ be a set of sites and $\mathbf{w} \in \mathbb{R}^N$ the associated weights. Denote its element-wise absolute value $|\mathbf{w}| := (|w_1|, |w_2|, \ldots, |w_N|) \in \mathbb{R}^N$. The quadrature $Q[X_N, \mathbf{w}]$ is said to satisfy
    \begin{enumerate}[(a)]
        \item an \textit{$L^2$ MZ inequality} \cite{mhaskar2001spherical} of degree $n \in \mathbb{N}_0$, if there exist constants $A \leq 1 \leq B$ such that
            \begin{equation} \label{eq:L2-MZ-inequality}
            A \, I(f^2) \leq Q[X_N, \mathbf{w}](f^2) \leq B \, I(f^2) \quad \forall\, f \in \mathbb{P}_n(\mathbb{S}^2);
        \end{equation}
        \item an \textit{$L^2$ MZ condition} \cite{gia2009localized} of degree $n \in \mathbb{N}_0$, if there exist constants $c \geq 0$ such that
        \begin{equation}
        \label{eq:L2-MZ-condition}
            Q[X_N, |\mathbf{w}|](f^2) \leq c\, I(f^2) \quad \forall\, f \in \mathbb{P}_n(\mathbb{S}^2).
        \end{equation}
    \end{enumerate}
\end{definition}

We begin by relating these conditions to P\'olya's condition. 
\begin{remark}[Connection to P\'olya's Condition] 
Taking $f \equiv 1 \in \mathbb{P}_0(\mathbb{S}^2)$ in an $L^2$ MZ condition \eqref{eq:L2-MZ-condition} of any degree yields
\begin{equation} \label{eq:L2-MZ-implies-Polya}
    \|\mathbf{w}\|_1 \leq 4\pi c.
\end{equation}
Therefore, the $L^2$ MZ condition \eqref{eq:L2-MZ-condition} implies P\'olya's condition \eqref{eq:polya-cond}. However, the same cannot be said for $L^2$ MZ inequalities, which are more prevalent in the literature, particularly for studying the strictly positive quadratures \cite{mhaskar2001spherical, filbir2024marcinkiewicz, an2026roleweakmarcinkiewiczzygmundconstants}. Taking $f \equiv 1$ in \eqref{eq:L2-MZ-inequality} gives
\begin{equation*}
    4\pi A \leq \sum_{j=1}^N w_j \leq 4\pi B,
\end{equation*}
offering no control over the absolute sum $\|\mathbf{w}\|_1$. This necessitates the distinction between the $L^2$ MZ inequalities (signed) and $L^2$ MZ conditions (absolute) when dealing with signed weights.
\end{remark}

The accuracy of hyperinterpolation is well documented in the literature, see, e.g.,  \cite{sloan1995polynomial,hesse2006hyperinterpolation ,an2024bypassing}. However, these classical bounds typically rely on strictly positive quadrature weights and algebraic exactness. We establish the corresponding bounds that bypass both of these requirements. By accommodating signed weights and the absence of quadrature exactness, this generalized bound underpins our MZ collocation framework.

\begin{theorem}
\label{thm:what-MZ-can-do}
    Let $Q[X_N, \mathbf{w}]$ be a quadrature. For any $f \in C(\mathbb{S}^2)$, let $$p^* := \operatorname*{argmin}_{p \in \mathbb{P}_n(\mathbb{S}^2)} \|f - p \|_\infty$$ be the best uniform approximation polynomial of $f$ in $\mathbb{P}_n(\mathbb{S}^2)$. If $Q[X_N, \mathbf{w}]$ satisfies an $L^2$ MZ condition of degree $n \in \mathbb{N}_0$ with constant $c \geq 0$, then the hyperinterpolation operator $L_n$ is stable in the sense that
    \begin{equation*}
        \|L_n\|_{C\to L^2} \leq c\sqrt{4\pi},
    \end{equation*}
    and the hyperinterpolation error satisfies
    \begin{equation} \label{eq:L2-hyperinterpolation-error}
        \|L_n f - f \|_{L^2} \leq (c+1) \sqrt{4\pi} \|f - p^*\|_\infty + \|L_n p^* - p^*\|_{L^2}.
    \end{equation}
\end{theorem}
\begin{proof}
    By setting $p = L_n f \in \mathbb{P}_n(\mathbb{S}^2)$ in \cref{lem:hyperinterpolation-lemma}, we obtain
    \begin{equation}
        \label{eq:L2-norm-of-hyperinterpolation}
        \|L_nf\|_{L^2}^2 = \sum_{j=1}^N w_jf(\mathbf{x}_j)L_nf(\mathbf{x}_j).
    \end{equation}
    We apply the Cauchy--Schwarz inequality to the sum \eqref{eq:L2-norm-of-hyperinterpolation}. Invoking the $L^2$ MZ condition \eqref{eq:L2-MZ-condition} and \eqref{eq:L2-MZ-implies-Polya}, we deduce
    \begin{align*} 
        \|L_nf\|_{L^2}^2 & \leq \left(\sum_{j=1}^N |w_j| |f(\mathbf{x}_j)|^2\right)^
        {1/2} \left(\sum_{j=1}^N |w_j| |L
        _nf(\mathbf{x}_j)|^2\right)^
        {1/2} \leq \sqrt{\|\mathbf{w}\|_1}\|f\|_\infty \sqrt{c}\|L_nf\|_{L^2} \\
        & \leq \sqrt{4\pi c} \|f\|_\infty \sqrt{c} \|L_nf\|_2 = c\sqrt{4\pi}\|f\|_\infty \|L_nf\|_2.
    \end{align*}
    Consequently, $\|L_n\|_{C \to L^2} \leq c\sqrt{4\pi}$.

    To establish the error bound \eqref{eq:L2-hyperinterpolation-error}, for any $p \in \mathbb{P}_n(\mathbb{S}^2)$, we split the hyperinterpolation error as follows using the triangle inequality:
    \begin{align*}
        \|L_n f - f\|_{L^2} & \leq \|L_n (f-p)\|_{L^2} + \|f-p\|_{L^2} + \|L_n p - p\|_{L^2} \\
        & \leq c\sqrt{4\pi}\|f - p\|_\infty +\sqrt{4\pi}\|f - p\|_\infty + \|L_n p - p\|_{L^2}\\
        & = (c+1)\sqrt{4\pi}\|f - p\|_\infty + \|L_n p - p\|_{L^2}.
    \end{align*}
    Setting $p = p^*$ gives the desired upper bound.
\end{proof}

The term $\|L_n p^* - p^*\|_{L^2}$ in \eqref{eq:L2-hyperinterpolation-error} is an accuracy term pertinent to $L^2$ MZ inequalities, as we now specify.

\begin{proposition} \label{lem:accuracy-term}
    Suppose the quadrature $Q[X_N, \mathbf{w}]$ satisfies an $L^2$ MZ inequality of degree $n$ with constants $A \leq B$. Let $\eta := \max\{|1 - A|, |1 - B|\}$. We have
    \begin{equation*}
        \|L_n p - p\|_{L^2} \leq \eta \|p \|_{L^2} \quad \forall\, p \in \mathbb{P}_n(\mathbb{S}^2).
    \end{equation*}
\end{proposition}
\begin{proof}
Since $L_np - p \in \mathbb{P}_n(\mathbb{S}^2)$, it follows from \cref{lem:hyperinterpolation-lemma} that
\begin{equation*}
    \|L_n p - p\|_{L^2}^2 = \langle L_n p - p, L_n p - p \rangle_{L^2} = \langle p, L_n p - p \rangle_Q - \langle p, L_n p - p \rangle_{L^2}.
\end{equation*}
Define now the symmetric bilinear form $B(u,v):=\langle u,v\rangle_Q - \langle u, v \rangle_{L^2}$ on $\mathbb{P}_n(\mathbb{S}^2) \times \mathbb{P}_n(\mathbb{S}^2)$. It follows from the definition of $L^2$ MZ inequalities that
\begin{equation} \label{eq:diagonal-bound}
    |B(p, p)| \leq \eta \|p\|_{L^2} \quad \forall\, p \in \mathbb{P}_n(\mathbb{S}^2).
\end{equation}
For any $u, v \in \mathbb{P}_n(\mathbb{S}^2)$, the polarization identity implies
\begin{equation*}
    B(u,v) = \frac{1}{4}[B(u+v, u+v) - B(u-v, u-v)].
\end{equation*}
Apply the triangle inequality, \eqref{eq:diagonal-bound}. Then, the parallelogram law yields
\begin{align} \label{eq:cross-term-bound}
    |B(u,v)| & \leq \frac{\eta}{4}[\|u+v\|_{L^2}^2 + \|u-v\|_{L^2}^2] = \frac{\eta}{2}[\|u\|_{L^2}^2 + \|v\|_{L^2}^2].
\end{align}
Without loss of generality, assuming that $u, v \neq 0$ (the case where $u = 0$ or $v = 0$ is trivial), we define the normalized polynomials $\tilde{u}:=u/\|u\|_{L^2}$ and $\tilde{v}:= v/\|v\|_{L^2}$. Since $\|\tilde{u}\|_{L^2} = \|\tilde{v}\|_{L^2} = 1$, \eqref{eq:cross-term-bound} implies $|B(\tilde{u}, \tilde{v})| \leq \eta$. Multiplying both sides by $\|u\|_{L^2}\|v\|_{L^2}$ gives
\begin{equation*}
    |B(u,v)| \leq \eta\|u\|_{L^2}\|v\|_{L^2} \quad \forall\, u, v \in \mathbb{P}_n(\mathbb{S}^2).
\end{equation*}
Setting $u = p \in \mathbb{P}_n(\mathbb{S}^2)$ and $v = L_np - p \in \mathbb{P}_n(\mathbb{S}^2)$, we obtain
\begin{equation*}
    |\langle p, L_np- p \rangle_Q - \langle p,  L_np- p \rangle_{L^2}| \leq \eta \|p\|_{L^2} \| L_np- p\|_{L^2}.
\end{equation*}
Substituting this bound back into the first equation gives $\|L_np- p\|_{L^2}^2 \leq \eta\|p\|_{L^2} \|L_np- p\|_{L^2}$, which completes the proof. 
\end{proof}

Combining \cref{thm:what-MZ-can-do,lem:accuracy-term} gives a structural decomposition of the hyperinterpolation error. Assume that the constants are known in the $L^2$ MZ inequality and the $L^2$ MZ condition. Then, the error bounds decompose into two components
\begin{equation} \label{eq:hyperinterpolation-error-decomposition}
    \|L_n f - f \|_{L^2} \quad\leq\quad \underbrace{(c+1) \sqrt{4\pi} \|f - p^*\|_\infty}_{\text{Stability Penalty}} \quad + \; \underbrace{\eta\|p^*\|_{L^2}}_{\text{Approximation Quality}}.
\end{equation}
The first term represents a \textit{stability penalty}, controlled by the constant $c$ in the $L^2$ MZ condition. The second term captures the \textit{approximation quality} on polynomials, governed by the constants $A$ and $B$ in the $L^2$ MZ inequality via
\begin{equation} \label{eq:eta}
    \eta := \max\{|1- A|, |1-B|\}.
\end{equation}
This decomposition implies that the optimal quadrature design for hyperinterpolation involves a balance between operator stability and integration accuracy on polynomials. In the context of our weight collocation framework, this balance is suited for a regularized optimization scheme.

The rest of the section is organized as follows. In \Cref{sec:perturbation}, we derive geometric-aware regularizers specifically designed to control the stability constant $c$. Subsequently, in \Cref{sec:spectrum-of-Gram-matrix}, we address the optimization of $\eta$ reflecting approximation quality. Finally, in \Cref{sec:sdp}, we synthesize these components and develop an efficient algorithm to solve the resulting collocation problem.

\subsection{Perturbation from a Geometric Prior} \label{sec:perturbation}
To construct a regularizer that stabilizes the $L^2$ MZ constant $c$, we adopt a perturbation perspective. The core idea is to establish a prior geometric baseline with a good MZ constant. We then quantify the deterioration of the MZ constant when weights deviate from this prior, leading to a regularizer that penalizes such deviations.

\subsubsection{The 2-Optimality of Voronoi Partitions}
\label{sec:2-optimality}
The extraction of spatial information from scattered sites on the sphere is typically formalized through the notion of compatible partitions.

\begin{definition}
    A finite collection $\mathcal{R} := \{R_1, R_2, \ldots, R_N\}$ of closed, nonoverlapping (i.e., having no common interior points) subsets such that $\mathbb{S}^2 = \bigcup_{j=1}^N R_j$ is called a \textit{partition} of $\mathbb{S}^2$. We say that $\mathcal{R}$ is \textit{$X_N$-compatible} if each \textit{patch} $R_j \in \mathcal{R}$ contains exactly one point $\mathbf{x}_j \in X_N$ in its interior. We denote the family of all $X_N$-compatible partitions by $\mathcal{P}(X_N)$. 
    The \textit{partition weight} of $\mathcal{R}$ is the vector $\mathbf{r} = (\omega(R_1), \omega(R_2), \ldots, \omega(R_N))^\top \in \mathbb{R}^N$. We call $Q[X_N, \mathcal{R}] := Q[X_N, \mathbf{r}]$ a \textit{geometric quadrature} when its quadrature weight is a partition weight. The \textit{partition norm} of $\mathcal{R}$ is defined by its largest patch diameter
    \begin{equation*}
        \|\mathcal{R}\| := \max_{R_j \in \mathcal{R}} \mathrm{diam}\, R_j \in [0, \pi],
    \end{equation*}
    where $\mathrm{diam}\, S := \sup_{\mathbf{x}, \mathbf{y} \in S} \mathrm{dist}(\mathbf{x}, \mathbf{y})$.
\end{definition}

Because partition weights are strictly nonnegative, $L^2$ MZ inequality and condition for them coincide. In the literature, e.g., \cite{mhaskar2001spherical, narcowich2006localized, filbir2024marcinkiewicz, keiner2007efficient}, MZ inequalities for geometric quadratures $Q[X_N, \mathcal{R}]$ associated with an $X_N$-compatible partition $\mathcal{R}$ have been studied. A standard result is that the constant $c$ scales as $1 + O(n\|\mathcal{R}\|)$.
\begin{theorem} \label{thm:MZ}
    Let $\mathcal{R}$ be an $X_N$-compatible partition of $\mathbb{S}^2$. Under mild conditions ensuring that the product $n\|\mathcal{R}\|$ is small, the quadrature $Q[X_N, \mathcal{R}]$ satisfies an $L^2$ MZ inequality (also an $L^2$ MZ condition) of degree $n \in \mathbb{N}_0$:
    \begin{equation*}
        A\, I(f^2) \leq Q[X_N, \mathcal{R}](f^2) \leq B \, I(f^2) \quad \forall\, f \in \mathbb{P}_n(\mathbb{S}^2),
    \end{equation*}
    where $A = 1 - O(n\|\mathcal{R}\|)$ and $B = 1 + O(n\|\mathcal{R}\|)$.
\end{theorem}
\begin{proof}
    See \cite[Theorem~1]{keiner2007efficient} for a proof with explicit bounds on the hidden constants. See also \cite{mhaskar2001spherical, narcowich2006localized, filbir2024marcinkiewicz}.
\end{proof}

Because the constant $c$ in the $L^2$ MZ condition for geometric quadrature is intrinsically governed by the partition norm $\|\mathcal{R}\|$, the ideal geometric prior is the $X_N$-compatible partition with minimal partition norm. This naturally motivates the \textit{minimum norm $X_N$-compatible partition problem}
\begin{equation} \label{eq:minimum-norm-partition-problem}
    R^*(X_N) := \inf_{\mathcal{R} \in \mathcal{P}(X_N)} \|\mathcal{R}\|.
\end{equation}
This problem is combinatorial in nature. We therefore seek an approximate solution. A natural candidate is the \textit{Voronoi partition} $\mathcal{V} = \{V_j\}_{j=1}^N$, where each patch $V_j$ is given by the Voronoi cell
\begin{equation} \label{eq:voronoi-cell}
  V_j := \left\{\mathbf{x} \in \mathbb{S}^2 \,|\, \mathrm{dist}(\mathbf{x}, \mathbf{x}_j) \leq \mathrm{dist}(\mathbf{x}, \mathbf{x}_i), \, i = 1, 2, \ldots, N \right\}.
\end{equation}

\newpage
We begin with the following lemma.
\begin{lemma} \label{lem:mesh-norm-and-partition-norm}
$h_{X_N} \leq \|\mathcal{R}\|$ for any $\mathcal{R} \in \mathcal{P}(X_N)$.
\end{lemma}
\begin{proof}
    By definition, any $\mathbf{x} \in \mathbb{S}^2$ belongs to some patch $R_j \in \mathcal{R}$, which by compatibility contains $\mathbf{x}_j \in X_N$. Therefore, it holds that
    \begin{equation*}
        h_{X_N} = \max_{\mathbf{x} \in \mathbb{S}^2} \min_{\mathbf{x}_j \in X_N} \mathrm{dist}(\mathbf{x}, \mathbf{x}_j) = \max_{R_j \in \mathcal{R}} \max_{\mathbf{x} \in R_j} \min_{\mathbf{x}_j \in X_N} \mathrm{dist}(\mathbf{x}, \mathbf{x}_j) \leq \max_{R_j \in \mathcal{R}} \max_{\mathbf{x} \in R_j} \max_{\mathbf{y} \in R_j} \mathrm{dist}(\mathbf{x}, \mathbf{y}) = \|\mathcal{R}\|.
    \end{equation*}
The proof is complete. 
\end{proof}
The Voronoi partition $\mathcal{V}$ is 2-optimal for the minimum norm $X_N$-compatible partition problem \eqref{eq:minimum-norm-partition-problem} in sense of the following proposition:
\begin{proposition} \label{prop:2-optimal}
The Voronoi partition $\mathcal{V}$ is an $X_N$-compatible partition and is $2$-optimal for the minimum norm $X_N$-compatible partition problem \cref{eq:minimum-norm-partition-problem} in the sense that 
\begin{equation} \label{eq:2-optimality-voronoi-partition-norm}
    R^*(X_N) \leq \Vert\mathcal{V}\Vert \leq 2R^*(X_N).
\end{equation}
\end{proposition}
\begin{proof}
The $X_N$-compatibility of $\mathcal{V}$ is obvious by the definition. To prove the $2$-optimality, it suffices to show $h_{X_N} \leq R^*(X_N) \leq \|\mathcal{V}\| \leq 2h_{X_N}$. The first inequality follows immediately from taking the infimum over $\mathcal{R} \in \mathcal{P}(X_N)$ in \cref{lem:mesh-norm-and-partition-norm}. The second inequality is by the definition of \cref{eq:minimum-norm-partition-problem}. For any $\mathbf{x}, \mathbf{y} \in V_i \in \mathcal{V}$, the defining property \eqref{eq:voronoi-cell} implies $\mathrm{dist}(\mathbf{x}, \mathbf{x}_i) = \min_{\mathbf{x}_j \in X_N} \mathrm{dist}(\mathbf{x}, \mathbf{x}_j) \leq h_{X_N}$ and similarly $\mathrm{dist}(\mathbf{y}, \mathbf{x}_i) \leq h_{X_N}$. Thus, by the triangle inequality, $\mathrm{dist}(\mathbf{x}, \mathbf{y}) \leq \mathrm{dist}(\mathbf{x}, \mathbf{x}_i) + \mathrm{dist}(\mathbf{y}, \mathbf{x}_i) \leq 2h_{X_N}$. This establishes the last inequality, thereby completing the proof.
\end{proof}

The upper bound in \cref{eq:2-optimality-voronoi-partition-norm} is remarkably \textit{sharp}, as illustrated by the following simple example.

\begin{example}[Equatorial Points] \label{ex:equatorial}
    For an even integer $N$, consider the set
    \begin{equation*}
        X_N := \{\mathbf{x}_j :=(\sin\theta\cos \phi_j, \sin\theta \sin \phi_j, \cos\theta)^\top\}_{j=0}^{N-1}, \quad\text{where}\quad \theta = \pi/2, \quad  \phi_j := 2\pi j / N,
    \end{equation*}
    uniformly distributed on the equator. The points farthest from $X_N$ on the sphere are the poles, yielding $h_{X_N} = \pi/2$. The Voronoi partition $\mathcal{V}_N$ associated with $X_N$ divides the sphere into $N$ identical longitudinal lunes, see \cref{fig:equatorial_vor}. Because the poles are equidistant to all points in $X_N$, every Voronoi cell spans exactly from the North pole to the South pole. Thus, for any even $N$, $\|\mathcal{V}_N\| = \pi = 2h_{X_N}$. 

    However, one can construct $X_N$-compatible partitions with significantly smaller partition norms. Fix any small $\epsilon > 0$, define the patches as
    \begin{equation*}
    \begin{aligned}
        R_{2k+1} &:= \{ \mathbf{x}(\theta, \phi) \in \mathbb{S}^2 \,|\, 0 \leq \theta \leq \pi/2 - \epsilon \cos(N\phi/2), \phi_{2k} \leq \phi \leq \phi_{2k+2} \}, \\
        R_{2k} &:= \{ \mathbf{x}(\theta, \phi) \in \mathbb{S}^2 \,|\, \pi/2 - \epsilon \cos(N\phi/2) \leq \theta \leq \pi, \phi_{2k-1} \leq \phi \leq \phi_{2k+1} \},
    \end{aligned}
    \end{equation*}
    with indices wrapping modulo $N$, see \cref{fig:equatorial_opt}. It is straightforward to check that $\mathcal{R}_N := \{R_k\}_{k=1}^N$ is an $X_N$-compatible partition. If $N$ is chosen sufficiently large, then the maximum distance between any two points in each patch $R_k$ is precisely the distance from the pole that it contains to the tip of the zigzag curve, giving $\mathrm{diam}(R_j) = \pi/2 + \epsilon$. By taking $\epsilon \to 0$, we find $\|\mathcal{R}_N\| \to \pi/2 = h_{X_N}$. This implies $R^*(X_N) = \pi/2 = h_{X_N} = \|\mathcal{V}_N\|/2$. Hence, the upper bound of \eqref{eq:2-optimality-voronoi-partition-norm} is attainable.

    \begin{figure}[htbp]
     \centering
     \begin{subfigure}[b]{0.49\textwidth}
         \centering
         \includegraphics[width=0.75\textwidth]{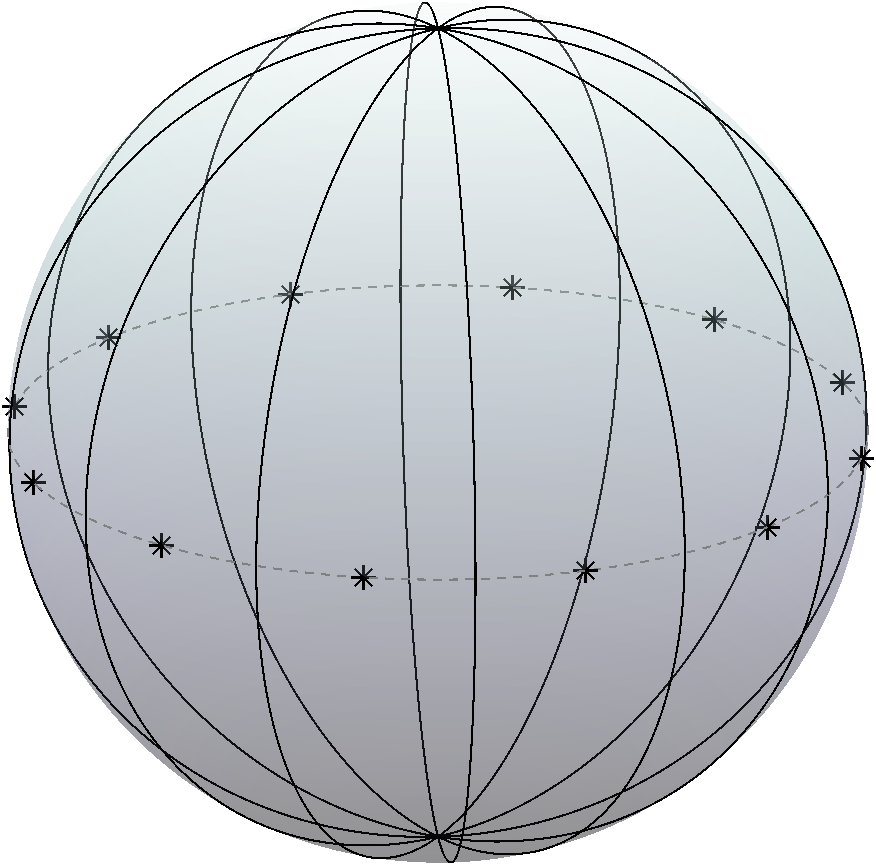}
         \caption{The Voronoi partition $\mathcal{V}_N$.}
         \label{fig:equatorial_vor}
     \end{subfigure}
     \hfill
     \begin{subfigure}[b]{0.49\textwidth}
         \centering
         \includegraphics[width=0.75\textwidth]{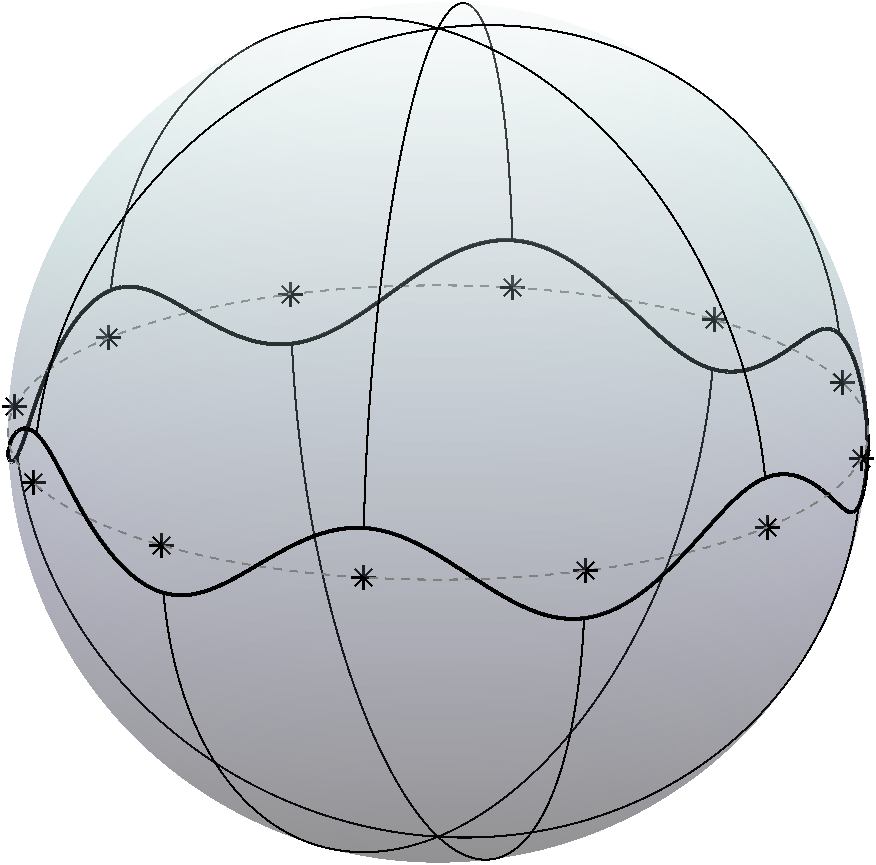}
         \caption{The constructed partition $\mathcal{R}_N$ with $\epsilon = 0.12$.}
         \label{fig:equatorial_opt}
     \end{subfigure}
     
     \caption{Two $X_N$-compatible partitions for the equatorial points ($N = 12$) as described in \cref{ex:equatorial}.}
\end{figure}

\end{example}

Recall from \cref{thm:MZ} that the theoretical stability of a geometric quadrature is fundamentally governed by $c = 1 + O(n\|\mathcal{R}\|)$. While directly minimizing the partition norm as in \eqref{eq:minimum-norm-partition-problem} is intractable, \cref{prop:2-optimal} confirms that the Voronoi partition provides a 2-optimal surrogate for the intractable problem \eqref{eq:minimum-norm-partition-problem}. This justifies the Voronoi weights as a suitable geometric prior.

Finally, we remark that the computation of Voronoi partitions is efficient; indeed algorithms running in $O(N \log N)$ operations are well-established in the literature \cite{boissonnat1998algorithmic, de2008computational, burkardt2019voronoi, na2002voronoi, renka1997algorithm, caroli2010robust}. Patch measures can be explicitly evaluated via the Gauss-Bonnet theorem (cf. \cite[(1.1)-(1.2)]{chern2024area}).

\subsubsection{Perturbation of Weights and Geometry-Aware Regularization}
\label{sec:geometric-aware-regularization}
To quantify the performance degradation as $\mathbf{w}$ deviates from this reference, we define the \textit{$\chi^2$-divergence} of $\mathbf{w}$ relative to the prior $\mathcal{R}$ as
\begin{equation} \label{eq:chi-square-divergence}
    \|\mathbf{w}\|_{\mathcal{R}, 2} := \left[\sum_{j=1}^N \frac{(w_j - r_j)^2}{r_j}\right]^{1/2}.
\end{equation}

\begin{proposition} \label{prop:perturb-III}
    Let $\mathcal{R} = \{R_j\}_{j=1}^N$ be an $X_N$-compatible partition. Suppose the geometric quadrature $Q[X_N, \mathcal{R}]$ satisfies an $L^2$ MZ inequality of degree $n \in \mathbb{N}_0$ with constants $A \leq B$. That is, it holds that
    \begin{equation} \label{eq:MZ-R}
        A \, I(f^2) \leq Q[X_N, \mathcal{R}](f^2) \leq B \, I(f^2), \quad \forall\, f \in \mathbb{P}_n(\mathbb{S}^2).
    \end{equation}
    Let $\mathbf{w} \in \mathbb{R}^N$ be an arbitrary weight vector. Then, the quadrature $Q[X_N, \mathbf{w}]$ satisfies the following inequalities for all $f \in \mathbb{P}_n(\mathbb{S}^2)$:
    \begin{equation*}
        \resizebox{\textwidth}{!}{
        $\left( A - \dfrac{(n+1)\sqrt{B}}{\sqrt{4\pi}} \|\mathbf{w}\|_{\mathcal{R}, 2} \right) I(f^2) \leq Q[X_N, \mathbf{w}](f^2) \leq Q[X_N, |\mathbf{w}|](f^2) \leq \left( B + \dfrac{(n+1)\sqrt{B}}{\sqrt{4\pi}} \|\mathbf{w}\|_{\mathcal{R}, 2} \right) I(f^2)$.
        }
    \end{equation*}
    This chain of inequalities is both an $L^2$ MZ inequality and an $L^2$ MZ condition.
\end{proposition}
\begin{proof}
    By the Cauchy--Schwarz inequality, for any $f \in \mathbb{P}_n(\mathbb{S}^2)$, it holds that
    \begin{align*}
        \sum_{j=1}^N |w_j - r_j| |f(\mathbf{x}_j)|^2 
        \leq \left[ \sum_{j=1}^N \frac{(w_j - r_j)^2}{r_j} \right]^{1/2} \left[ \sum_{j=1}^N r_j |f(\mathbf{x}_j)|^4 \right]^{1/2} = \|\mathbf{w}\|_{\mathcal{R}, 2} \left[ \sum_{j=1}^N r_j |f(\mathbf{x}_j)|^4 \right]^{1/2}.
    \end{align*}
    The fourth-order term can be handled as follows
    \begin{align*}
        \sum_{j=1}^N r_j |f(\mathbf{x}_j)|^4 
        & \leq \|f\|_\infty^2 \sum_{j=1}^N r_j |f(\mathbf{x}_j)|^2 = \|f\|_\infty^2 Q[X_N, \mathcal{R}](f^2).
    \end{align*}
    From the reproducing property of $G_n$ and the Cauchy--Schwarz inequality, we have, for any $\mathbf{x} \in \mathbb{S}^2$, that
    \begin{equation*}
        |f(\mathbf{x})|^2 \leq \left[\int_{\mathbb{S}^2} |G_n(\mathbf{x}, \mathbf{y})|^2\,d\omega(\mathbf{y})\right]\left[\int_{\mathbb{S}^2} |f(\mathbf{y})|^2\,d\omega(\mathbf{y})\right] = \frac{(n+1)^2}{4\pi} \|f\|_{L^2}^2.
    \end{equation*}
    Thus, $\|f\|_\infty^2 \leq (n+1)^2\|f\|_{L^2}^2/(4\pi)$. Combining these with the $L^2$ MZ inequality for $Q[X_N, \mathcal{R}]$ yields
    \begin{equation*}
        \sum_{j=1}^N r_j|f(\mathbf{x}_j)|^4 \leq \frac{(n+1)^2}{4\pi} B \|f\|_{L^2}^4.
    \end{equation*}
    Taking the square root and substituting it back into the perturbation bound yields
    \begin{equation*}
        \sum_{j=1}^N |w_j-r_j||f(\mathbf{x}_j)|^2 \leq \frac{(n+1)\sqrt{B}}{\sqrt{4\pi}}\|\mathbf{w}\|_{\mathcal{R}, 2} I(f^2).
    \end{equation*}
    To conclude the proof, it suffices to substitute the above estimate into
    \begin{multline*}
       \sum_{j=1}^N r_j |f(\mathbf{x}_j)|^2 - \sum_{j=1}^N |w_j - r_j| |f(\mathbf{x}_j)|^2 \leq \sum_{j=1}^N w_j|f(\mathbf{x}_j)|^2 \\
       \leq \sum_{j=1}^N |w_j| |f(\mathbf{x}_j)|^2 \leq \sum_{j=1}^N r_j |f(\mathbf{x}_j)|^2 + \sum_{j=1}^N |w_j - r_j| |f(\mathbf{x}_j)|^2.
    \end{multline*}
The proof is complete.
\end{proof}

\cref{prop:perturb-III} shows that $\|\mathbf{w}\|_{\mathcal{R}, 2}$ controls the deterioration of the constant $c$ in the $L^2$ MZ condition. Hence, it is a suitable regularizer for promoting stability in hyperinterpolation on scattered sites. In fact, one can consider the following \textit{$\chi^2$-divergence collocation} as an alternative to the standard $\ell_2$ minimization:
\begin{align} \label{eq:voronoi-collocation}
    \begin{aligned}
    \min \quad & \|\mathbf{w}\|_{\mathcal{R}, 2} \\ 
    \text{s.t.} \quad & \mathbf{Y}_{n^+} \mathbf{w} = \mathbf{b}_{n^+}, \quad \mathbf{w} \in \mathbb{R}^N.
    \end{aligned}
\end{align}
A particular interesting choice of $\mathcal{R}$ is the Voronoi partition $\mathcal{V}$, justified by its $2$-optimality in \Cref{sec:2-optimality}. This collocation model yields quadrature weights with tighter MZ constants, hence better hyperinterpolation stability.

\subsection{Spectrum of the Gram Matrix} \label{sec:spectrum-of-Gram-matrix}
While the regularizer in \Cref{sec:perturbation} controls the stability constant $c$, the error decomposition \eqref{eq:hyperinterpolation-error-decomposition} also requires controlling the accuracy term related to $\eta$. As established in \eqref{eq:eta}, it is governed by the constants $A$ and $B$ in the $L^2$ MZ inequality. These two constants admit a characterization in terms of the eigenvalues of the following Gram matrix:
\begin{equation*}
    \mathbf{G}_n[X_N](\mathbf{w}) := \mathbf{Y}_n  \mathrm{diag}(\mathbf{w})\mathbf{Y}_n^\top \in \mathbb{R}^{M \times M}, \quad M := (n+1)^2.
\end{equation*}

\begin{proposition}[\!\!\cite{an2026roleweakmarcinkiewiczzygmundconstants,filbir2011marcinkiewicz,grochenig2020sampling}]
    \label{prop:eigenvalue-characterization}
    For any quadrature $Q[X_N, \mathbf{w}]$, the sharpest constants $A$ and $B$ satisfying the $L^2$ MZ inequality of degree $n \in \mathbb{N}_0$
    \begin{equation*}
        A\,I(f^2) \leq Q[X_N, \mathbf{w}](f^2) \leq B\,I(f^2), \quad \forall\, f \in \mathbb{P}_n(\mathbb{S}^2),
    \end{equation*}
    are the smallest and largest eigenvalues of the Gram matrix $\mathbf{G}_n[X_N](\mathbf{w})$, respectively. That is, we have
    \begin{equation} \label{eq:eigenvalue-characterization}
        A = \lambda_1(\mathbf{G}_n[X_N](\mathbf{w})) \quad \;\text{and} \quad  \quad B = \lambda_{M}(\mathbf{G}_n[X_N](\mathbf{w})).
    \end{equation}
\end{proposition}
\cref{prop:eigenvalue-characterization} reduces the design of accurate quadrature weights to the optimization of the spectrum of the Gram matrix. We present two formulations for this spectral optimization.

\subsubsection{Minimal \texorpdfstring{$\eta$}{η}}
\label{sec:minimal-eta}
As demonstrated in \eqref{eq:hyperinterpolation-error-decomposition}, the error resulted by the accuracy term is bounded by $\eta \|p^*\|_{L^2}$. Therefore, one approach is to minimize $\eta$ directly. More specifically, let $\mathbf{I} \in \mathbb{R}^{M \times M}$ denote the identity matrix. In view of the eigenvalue characterization \cref{eq:eigenvalue-characterization}, the constant $\eta$ can be expressed as the spectral $2$-norm. That is, it holds that
\begin{equation*}
    \eta = \max\{|1-A|, |1-B|\} = \|\mathbf{I} - \mathbf{G}_n[X_N](\mathbf{w})\|_2.
\end{equation*}
This leads to the following model for \textit{spectral collocation}:
\begin{align} \label{eq:spectral-collocation}
    \begin{aligned}
        \min \quad & \|\mathbf{I} - \mathbf{G}_n[X_N](\mathbf{w})\|_2 \\
        \text{s.t.} \quad & \mathbf{Y}_{n^+} \mathbf{w} = \mathbf{b}_{n^+}, \quad \mathbf{w} \in \mathbb{R}^N.
    \end{aligned}
\end{align}
This optimization problem collocates quadrature weights that minimize $\eta$ for an $L^2$ MZ inequality of degree $n$ among all feasible weight vectors exact of degree $n^+$.

\subsubsection{Well-Conditioned Gram Matrix}
\label{sec:well-conditioned-Gram-matrix}
To solve the spectral collocation problem \eqref{eq:spectral-collocation}, it is typical to consider its SDP reformulation by introducing an auxiliary variable $\eta \in \mathbb{R}$ and replacing the spectral $2$-norm with a linear matrix inequality. That is, we could reformulate (\ref{eq:spectral-collocation}) as
\begin{align*}
    \begin{aligned}
    \min \quad & \eta \\
    \text{s.t.} \quad & \mathbf{Y}_{n^+} \mathbf{w} = \mathbf{b}_{n^+} \\
    & -\eta\, \mathbf{I} \preceq \mathbf{I} - \mathbf{G}_n[X_N](\mathbf{w}) \preceq \eta \,\mathbf{I}, \quad \mathbf{w} \in \mathbb{R}^N.
    \end{aligned}
\end{align*}
On the other hand, this SDP reformulation becomes computationally prohibitive for large $N$ due to the dense matrix inequalities. For this reason, instead of minimizing the spectral norm, we consider minimizing the condition number of the Gram matrix. Let us make two standard assumptions regarding the quadrature $Q[X_N, \mathbf{w}]$ as follows:
\begin{enumerate}[(1),noitemsep]
    \item it integrates the constant function exactly, i.e., \eqref{eq:exact-for-const}, which naturally forces $A \leq 1 \leq B$; and
    \item the resulting Gram matrix is positive definite, i.e., $\mathbf{G}_n[X_N](\mathbf{w}) \succ 0$, ensuring $A > 0$.
\end{enumerate}
The condition number is then given by
\begin{equation} \label{eq:cond-gram}
    \kappa := \mathrm{cond}(\mathbf{G}_n[X_N](\mathbf{w})) = \frac{\lambda_{M}(\mathbf{G}_n[X_N](\mathbf{w}))}{\lambda_1(\mathbf{G}_n[X_N](\mathbf{w}))} = \frac{B}{A} \geq 1.
\end{equation}

\begin{remark}[Relationship of $\kappa$ and $\eta$]
    Given the assumption $A \leq 1 \leq B$, we have $B = \kappa A \leq \kappa$ and $A = B/\kappa \geq 1/\kappa$. This yields $B - 1 \leq \kappa - 1$ and $1 - A \leq 1 - 1/\kappa$. Since $\kappa \geq 1$, it holds that
    \begin{equation*}
        1-1/\kappa \leq \eta \leq \kappa - 1.
    \end{equation*}
    This relationship shows that minimizing $\kappa$ effectively minimizes $\eta$, justifying the rationale of replacing the latter with the former.
\end{remark}
In the literature, there are some methodologies for minimizing the condition numbers of Gram matrices, e.g.,   \cite{marechal2009optimizing, chen2011minimizing}. These works, however, generally yield quasiconvex problems that are often more difficult than the SDP \cref{eq:spectral-collocation}. Instead, inspired by the classical experimental design theory \cite{huan2024optimal} and recent advances in \cite{an2010well}, we formulate a tractable surrogate by replacing the condition number with the negative log-determinant of a Gram matrix. This yields the following model of \textit{$D$-optimal collocation}:
\begin{align} \label{eq:opt-det}
\begin{aligned}
\min \quad & -\log\det \mathbf{G}_n[X_N](\mathbf{w}) \\
\text{s.t.} \quad & \mathbf{Y}_{n^+}\mathbf{w} = \mathbf{b}_{n^+}, \quad \mathbf{w} \in \mathbb{R}^N.
\end{aligned}
\end{align}
A practical advantage of this formulation is that the negative log-determinant acts as an implicit barrier, evading the need for handling explicit conic constraints. This reduces the model to a smooth and convex optimization problem over the positive definite cone, and it can be solved efficiently by standard optimization methods such as Newton-type methods.

To further understand the effect of the negative log-determinant objective, we utilize the addition theorem \eqref{eq:addition-theorem} to show that the trace of $\mathbf{G}_n[X_N](\mathbf{w})$ is inherently constant in the following sense:
\begin{equation*}
  \mathrm{Tr}(\mathbf{G}_n[X_N](\mathbf{w})) = \sum_{j=1}^N w_j  \sum_{\ell=0}^n \sum_{k=-\ell}^\ell Y_{\ell,k}(\mathbf{x}_j)Y_{\ell,k}(\mathbf{x}_j)= \sum_{j=1}^N w_j \sum_{\ell = 0}^n \frac{2\ell + 1}{4\pi}P_\ell(\langle \mathbf{x}_j, \mathbf{x}_j\rangle) = M,
\end{equation*}
where we have utilized \eqref{eq:exact-for-const} and the fact that $P_\ell(1) = 1$. Consider any weight vector $\mathbf{w} \in \mathbb{R}^N$ admitting a positive definite Gram matrix. Let the eigenvalues of the Gram matrix be $\lambda_1 \geq \lambda_2 \geq \cdots \geq \lambda_{M} > 0$. By the AM-GM inequality, we have
\begin{equation*}
  \det \mathbf{G}_n[X_N](\mathbf{w})^{1/M} = \left(\prod_{i=1}^{M} \lambda_i \right)^{1/M} \leq \frac{1}{M}\sum_{i=1}^{M} \lambda_i =\frac{1}{M} \mathrm{Tr}(\mathbf{G}_n[X_N](\mathbf{w})) = 1,
\end{equation*}
where the equality holds if and only if all $\lambda_i$ are equal to $1$. In this case, we have $\eta = 0$ and $\kappa = 1$. Therefore, minimizing the negative log-determinant encourages the eigenvalues to be close to $1$, which aligns the $D$-optimal objective with the minimization of $\kappa$ and $\eta$.

\subsection{Towards a Unified MZ Collocation Model} \label{sec:sdp}
The hyperinterpolation error in \eqref{eq:hyperinterpolation-error-decomposition} decomposes into two components: a stability penalty governed by $c$, and an approximation quality term governed by $\eta$. In \Cref{sec:perturbation,sec:spectrum-of-Gram-matrix}, we examine these components separately. To bound $c$, we propose the geometry-aware regularizer $\chi^2$-divergence. To minimize $\eta$, we propose two methods to optimize the spectrum of the Gram matrix. We now combine these two components into a single optimization model. We treat the accuracy objective as primary and incorporate stability via a geometry-aware regularizer, and the model reads as
\begin{align} \label{eq:MZ-collocation}
\begin{aligned}
    \min \quad & P(\mathbf{w}) :=\underbrace{J_n(\mathbf{w})\vphantom{\frac{\lambda}{2}}}_{\text{Accuracy}} + \underbrace{\frac{\lambda}{2} \|\mathbf{w}\|_{\mathcal{V},2}^2}_{\text{Stability}} \\
\text{s.t.} \quad & \mathbf{Y}_{n^+}\mathbf{w} = \mathbf{b}_{n^+}.
\end{aligned}
\end{align}
The primary objective $J_n(\mathbf{w})$ targets an $L^2$ MZ inequality of degree $n$. This can be chosen as either the spectral norm $\|\mathbf{I} - \mathbf{G}_n[X_N](\mathbf{w})\|_2$ from \eqref{eq:spectral-collocation}, or the computationally efficient $D$-optimal design surrogate $-\log\det \mathbf{G}_n[X_N](\mathbf{w})$ from \eqref{eq:opt-det}. The geometry-aware regularizer $\chi^2$-divergence is defined with respect to the Voronoi partition $\mathcal{V}$, theoretically known as being $2$-optimal (cf. \cref{prop:2-optimal}). The constant $\lambda > 0$ is the regularization strength.

Applying a generic SDP solver to \eqref{eq:MZ-collocation}, however, is computationally expensive for large $N$, because it treats $\mathbf{G}_n[X_N](\mathbf{w})$ as a dense matrix. We next exploit the sum-of-rank-one structure of the Gram matrix to obtain more efficient evaluations for the gradient and Hessian.

\subsubsection{Exploiting Sum-of-Rank-One Structure for \texorpdfstring{$D$}{D}-optimal Collocation}
Let us instantiate $J_n(\mathbf{w}) = -\log\det \mathbf{G}_n[X_N](\mathbf{w})$. Because the negative log-determinant also acts as a barrier function, the constraint $\mathbf{G}_n[X_N](\mathbf{w}) \succ 0$ is implicitly enforced. The objective function becomes
\begin{equation*}
    P(\mathbf{w}) = -\log\det\mathbf{G}_n[X_N](\mathbf{w}) + \frac{\lambda}{2} \|\mathbf{w}\|_{\mathcal{V}, 2}^2.
\end{equation*}
The Gram matrix is a linear combination of $N$ rank-$1$ matrices. That is, we have
\begin{equation*}
    \mathbf{G}_n[X_N](\mathbf{w}) = \mathbf{Y}_n\mathrm{diag}(\mathbf{w})\mathbf{Y}_n^\top = \sum_{j=1}^N w_j \mathbf{y}_n(\mathbf{x}_j)\mathbf{y}_n(\mathbf{x}_j)^\top,
\end{equation*}
where, for each $\mathbf{x} \in \mathbb{S}^2$, the vector $\mathbf{y}_n(\mathbf{x}) \in \mathbb{R}^{M}$ is the evaluation of all spherical harmonics up to degree $n$ at $\mathbf{x}$
\begin{equation*}
    \mathbf{y}_n(\mathbf{x}) := (Y_{0,0}(\mathbf{x}), Y_{1,-1}(\mathbf{x}), Y_{1,0}(\mathbf{x}), Y_{1,1}(\mathbf{x}), \ldots, Y_{n, n}(\mathbf{x}))^\top.
\end{equation*}
By taking derivatives, the $N \times 1$ gradient $\mathbf{g} = \nabla P(\mathbf{w})$ and the dense $N \times N$ Hessian $\mathbf{H} = \nabla^2 P(\mathbf{w})$ have the following components, respectively:
\begin{align*}
g_i &= -\mathrm{Tr} [\mathbf{G}_n[X_N]({\mathbf{w}})^{-1} \mathbf{y}_n(\mathbf{x}_i)\mathbf{y}_n(\mathbf{x}_i)^\top] + \nabla_{w_i} [(\lambda/2)\|\mathbf{w}\|_{\mathcal{V}, 2}^2] \\
&=-\mathbf{y}_n(\mathbf{x}_i)^\top \mathbf{G}_n[X_N](\mathbf{w})^{-1} \mathbf{y}_n(\mathbf{x}_i) + \lambda(w_i - v_i)/v_i; \\
H_{i,j} &= \mathrm{Tr}[ \mathbf{G}_n[X_N]({\mathbf{w}})^{-1} \mathbf{y}_n(\mathbf{x}_i)\mathbf{y}_n(\mathbf{x}_i)^\top\mathbf{G}_n[X_N]({\mathbf{w}})^{-1} \mathbf{y}_n(\mathbf{x}_j)\mathbf{y}_n(\mathbf{x}_j)^\top] + \nabla_{w_iw_j}^2 [(\lambda/2)\|\mathbf{w}\|_{\mathcal{V}, 2}^2]\\
&= [\mathbf{y}_n(\mathbf{x}_i)^\top \mathbf{G}_n[X_N](\mathbf{w})^{-1} \mathbf{y}_n(\mathbf{x}_j)]^2 + \lambda \delta_{i,j}/v_i,
\end{align*}
where $\delta_{i,j}$ is the Kronecker delta. To evaluate these quadratic forms efficiently, we perform the Cholesky factorization $\mathbf{G}_n[X_N](\mathbf{w}) = \mathbf{L}\mathbf{L}^\top$. By defining the whitened basis matrix $\mathbf{U} := \mathbf{L}^{-1}\mathbf{Y}_n \in \mathbb{R}^{M \times N}$ and the transformed Gram matrix $\mathbf{M} := \mathbf{U}^\top \mathbf{U} \in \mathbb{R}^{N \times N}$, we recover the exact gradient and Hessian entirely via simple matrix multiplications
\begin{subequations} \label{eq:rank-1-procedure}
    \begin{align}
        \mathbf{g} &= -\text{diag}(\mathbf{M}) + \lambda (\mathbf{w} - \mathbf{v}) \oslash \mathbf{v},\\
        \mathbf{H} &= \mathbf{M} \odot \mathbf{M} + \lambda \, \text{diag}( \mathbf{1} \oslash \mathbf{v}),
    \end{align}
\end{subequations}
where $\mathbf{1} \in \mathbb{R}^N$ is the all-one vector, $\odot$ and $\oslash$ denote Hadamard (element-wise) multiplication and division, respectively. 

\begin{remark}[Computational Complexity]
    A generic interior-point solver requires $O(NM^3+N^2M^2)$ floating-point operations (FLOPs) approximately per iteration to assemble the dense Hessian via congruence transformations and trace inner products (e.g., \cite{vandenberghe2010cvxopt}). In contrast, the oracle \eqref{eq:rank-1-procedure} requires $O(N M^2)$ FLOPs to form the Gram matrix $\mathbf{G}_n[X_N](\mathbf{w})$, $O(M^3)$ to carry out the Cholesky factorization of it, $O(NM^2)$ to solve for $\mathbf{U}$, $O(N^2 M)$ FLOPs to form $\mathbf{M}$, and $O(N^2)$ FLOPs to perform the Hadamard product, resulting in a total of $O(M^3+N^2M+NM^2)$ FLOPs per iteration.
\end{remark}

With these fast gradient and Hessian oracles, the $D$-optimal collocation problem \eqref{eq:MZ-collocation} with the choice $$
J_n(\mathbf{w}) = -\log\det \mathbf{G}_n[X_N](\mathbf{w})
$$
can be solved efficiently by a standard infeasible start Newton method with Armijo rule backtracking line search \cite{boyd2004convex}, detailed in \cref{alg:d_optimal}.
\begin{algorithm}[htbp]
\caption{Infeasible Start Newton Method for $D$-optimal Collocation}
\label{alg:d_optimal}
\begin{algorithmic}[1]
\REQUIRE Initial weights $\mathbf{w}^{(0)}$, Voronoi weights $\mathbf{v}$, regularization strength $\lambda > 0$, tolerance $\epsilon > 0$, Armijo parameter $\alpha \in (0,1/2)$, backtracking parameter $\beta \in(0,1)$.
\STATE Initialize $\mathbf{w} \leftarrow \mathbf{w}^{(0)}$, dual variables $\boldsymbol{\nu} \leftarrow \mathbf{0}$.
\REPEAT
\STATE Form Gram matrix $\mathbf{G}_n[X_N](\mathbf{w}) \leftarrow \mathbf{Y}_n \text{diag}(\mathbf{w}) \mathbf{Y}_n^\top$.
\STATE Compute Cholesky factorization $\mathbf{L} \leftarrow \texttt{chol}(\mathbf{G}_n[X_N](\mathbf{w}) )$.
\STATE Compute auxiliary matrices $\mathbf{U} \leftarrow \mathbf{L}^{-1} \mathbf{Y}_n$, $\mathbf{M} \leftarrow \mathbf{U}^\top \mathbf{U}$.
\STATE Compute $\mathbf{g}$ and $\mathbf{H}$ using \eqref{eq:rank-1-procedure}.
\STATE Compute primal and dual Newton steps $(\Delta \mathbf{w}, \Delta \boldsymbol{\nu})$:
\begin{equation*}
    \left[
\begin{matrix} 
    \mathbf{H} & \mathbf{Y}_{n^+}^\top \\ 
    \mathbf{Y}_{n^+} & \mathbf{0} 
\end{matrix} \right]\left[
\begin{matrix} 
    \Delta \mathbf{w} \\ \Delta \boldsymbol{\nu} 
\end{matrix} 
\right] = - \left[
\begin{matrix} 
\mathbf{g} + \mathbf{Y}_{n^+}^\top \boldsymbol{\nu} \\ \mathbf{Y}_{n^+}\mathbf{w} - \mathbf{b}_{n^+} 
\end{matrix} 
\right]
\end{equation*}
\STATE Backtracking line search on $\|\mathbf{r}(\mathbf{w}, \bm{\nu})\|_2 := \|(\mathbf{g}+\mathbf{Y}_{n^+}^\top \bm{\nu}, \mathbf{Y}_{n^+}\mathbf{w} - \mathbf{b}_{n^+})\|_2$
\begin{itemize}[noitemsep, topsep=0pt]
    \item[] $t := 1$.
    \item[] \textbf{while} $\|\mathbf{r}(\mathbf{w} + t\Delta\mathbf{w}, \bm{\nu} + t\Delta\bm{\nu})\|_2 > (1 - \alpha t)\|\mathbf{r}(\mathbf{w}, \bm{\nu})\|_2$ or $\mathbf{G}_n[X_N](\mathbf{w} + t\Delta\mathbf{w}) \not\succ \mathbf{0}$
    \item[] \quad\quad $t := \beta t$.
    \item[] Update $\mathbf{w} \gets \mathbf{w} + t \Delta\mathbf{w}$, $\bm{\nu} \gets \bm{\nu} + t \Delta\bm{\nu}$.
\end{itemize}
\UNTIL{$\mathbf{Y}_{n^+} \mathbf{w} = \mathbf{b}_{n^+}$ and $\|\mathbf{r}(\mathbf{w}, \bm{\nu})\|_2 \leq \epsilon$}
\end{algorithmic}
\end{algorithm}

The sum-of-rank-one structure can also be exploited in the spectral collocation \eqref{eq:spectral-collocation}, though two barrier matrices are needed for the linear matrix inequality. A path-following interior-point method \cite{vandenberghe1998determinant} could be applied, but solving the resulting barrier subproblems is more expensive than the $D$-optimal surrogate. We omit the details for brevity.
\section{Numerical Experiments} \label{sec:num_exp}
In this section, we numerically validate the efficiency of the proposed optimization approach to weight collocation by the fundamental approximation tasks of numerical integration and hyperinterpolation. All experiments are implemented in MATLAB, and the codes are available on \href{https://github.com/HansEtherious/weights}{GitHub}\footnote{https://github.com/HansEtherious/weights}. The optimization models are formulated via the parser-solver CVX \cite{cvx,gb08} and solved by MOSEK \cite{mosek}, with the exception of our customized interior-point algorithm for the MZ collocation models \eqref{eq:MZ-collocation}. All computations are executed on a Lenovo laptop equipped with an Intel(R) Core(TM) i9-12900H (2.50GHz).

To evaluate our methods across distinct geometric regimes, we benchmark performance on two contrasting spatial distributions: 
\begin{itemize}
    \item \textit{Halton Points} \cite{halton1960efficiency}: A synthetic low-discrepancy sequence generated via the following area-preserving mapping from a 2D Halton sequence $(u_{j,1}, u_{j,2})_{j=1}^N$ in the unit square:
    \begin{equation*}
    \theta_j = \arccos(2u_{j,1} - 1), \quad \phi_j = 2\pi u_{j,2}.
    \end{equation*}
    While Halton points exhibit a relatively small mesh norm $h_{X_N}$, their generation does not enforce a minimum spacing, which allows them to be clustered arbitrarily closely. This yields a tiny separation distance $q_{X_N}$.
    \item \textit{MAGSAT Points} \cite{langel1982results}: A challenging real-world dataset comprising sequential measurements collected by the MAGSAT satellite along its orbit, kindly provided to us by Prof. Alvise Sommariva. From the full trajectory of 10,443 recorded locations, we extract $N$ points sampled uniformly in time. Because the satellite's polar orbit leaves systematic longitudinal gaps at the equator, this dataset suffers from significant spatial clustering at the poles and coverage blind spots at the equator, manifesting as a larger mesh norm $h_{X_N}$.
\end{itemize}

The structural disparity between the Halton and MAGSAT points is exemplified in \cref{fig:point_sets}, where we tracked the mesh norm $h_{X_N}$ and the separation distance $q_{X_N}$ across varying sample sizes $N$. As the number of points increases, the MAGSAT points consistently exhibit a larger mesh norm, while the Halton points maintain a tighter mesh norm but possess substantially smaller separation distances. To provide an intuitive visualization of these geometries, the actual spatial distributions on the sphere are visualized for a representative $N = 1,024$ in \cref{fig:3d-view,fig:2d-view}.

\begin{figure}[htbp]
    \centering
    \includegraphics[width=0.65\linewidth]{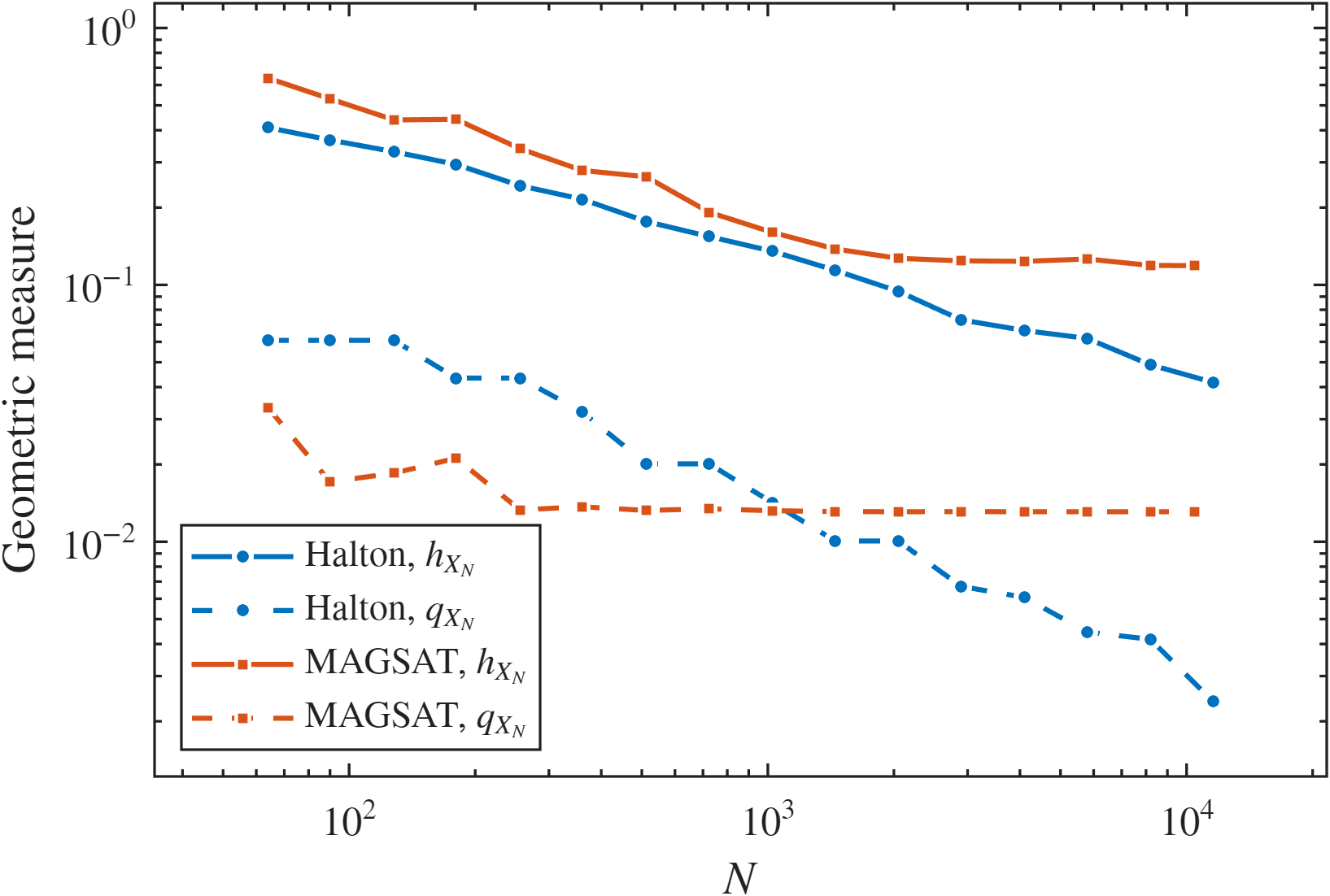}
    \caption{Mesh norm $h_{X_N}$ (solid lines) and separation distance $q_{X_N}$ (dashed lines) 
    versus cardinality $N$ for Halton and MAGSAT points used in the numerical experiments.}
    \label{fig:point_sets}
\end{figure}

\begin{figure}[htbp]
    \centering
    \begin{subfigure}[t]{0.45\textwidth}
        \centering
        \includegraphics[width=\textwidth]{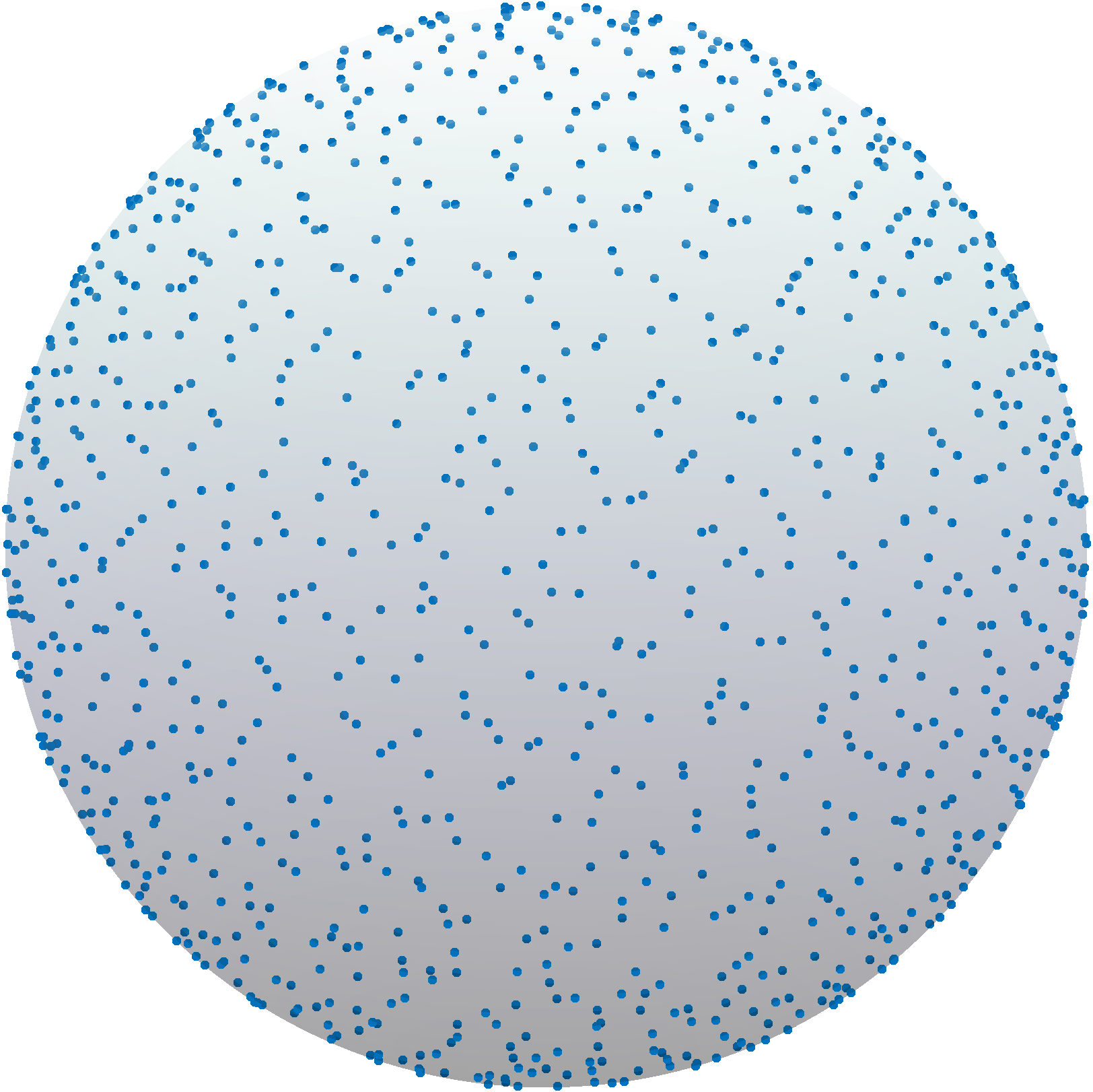}
        \caption{Halton points.}
    \end{subfigure}~\hspace{2mm}~
    \begin{subfigure}[t]{0.45\textwidth}
        \centering
        \includegraphics[width=\textwidth]{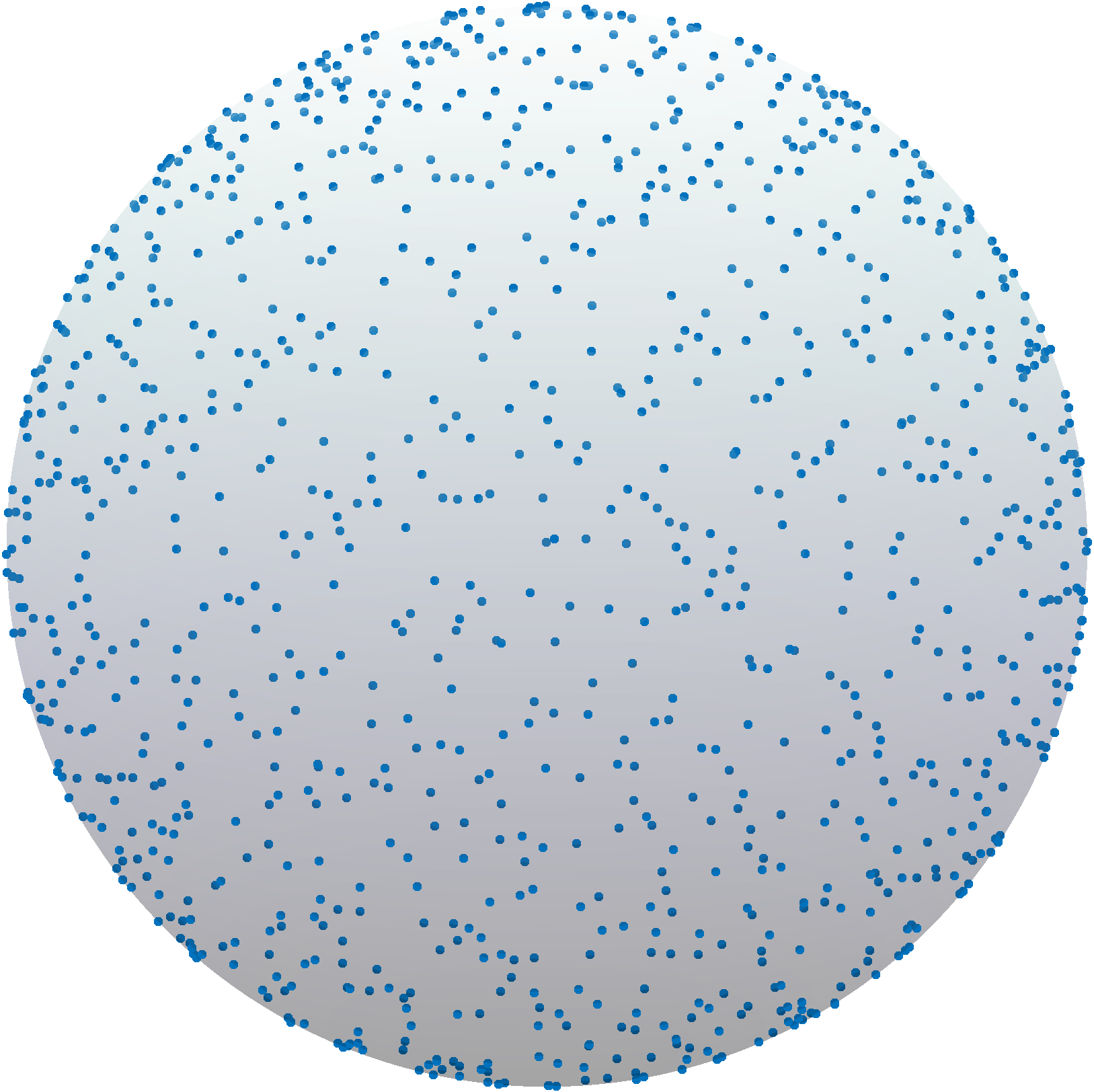}
        \caption{MAGSAT points.}
    \end{subfigure}
    \caption{$N = 1,024$ Halton and MAGSAT points on the unit sphere.}
    \label{fig:3d-view}
\end{figure}

\begin{figure}[htbp]
    \centering
    \begin{subfigure}[t]{0.45\textwidth}
        \centering
        \includegraphics[width=\textwidth]{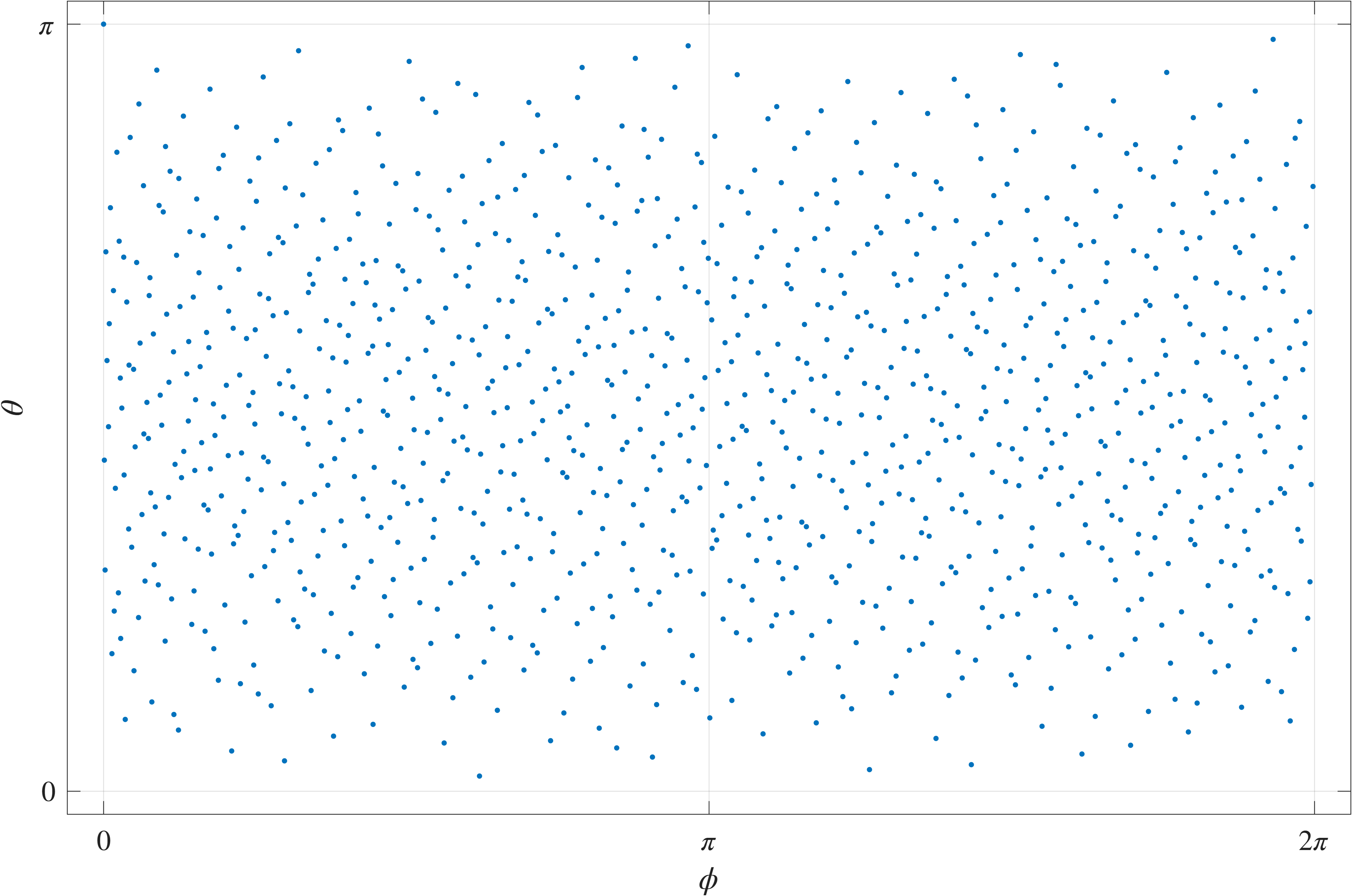}
        \caption{Halton points.}
    \end{subfigure}~\hspace{2mm}~
    \begin{subfigure}[t]{0.45\textwidth}
        \centering
        \includegraphics[width=\textwidth]{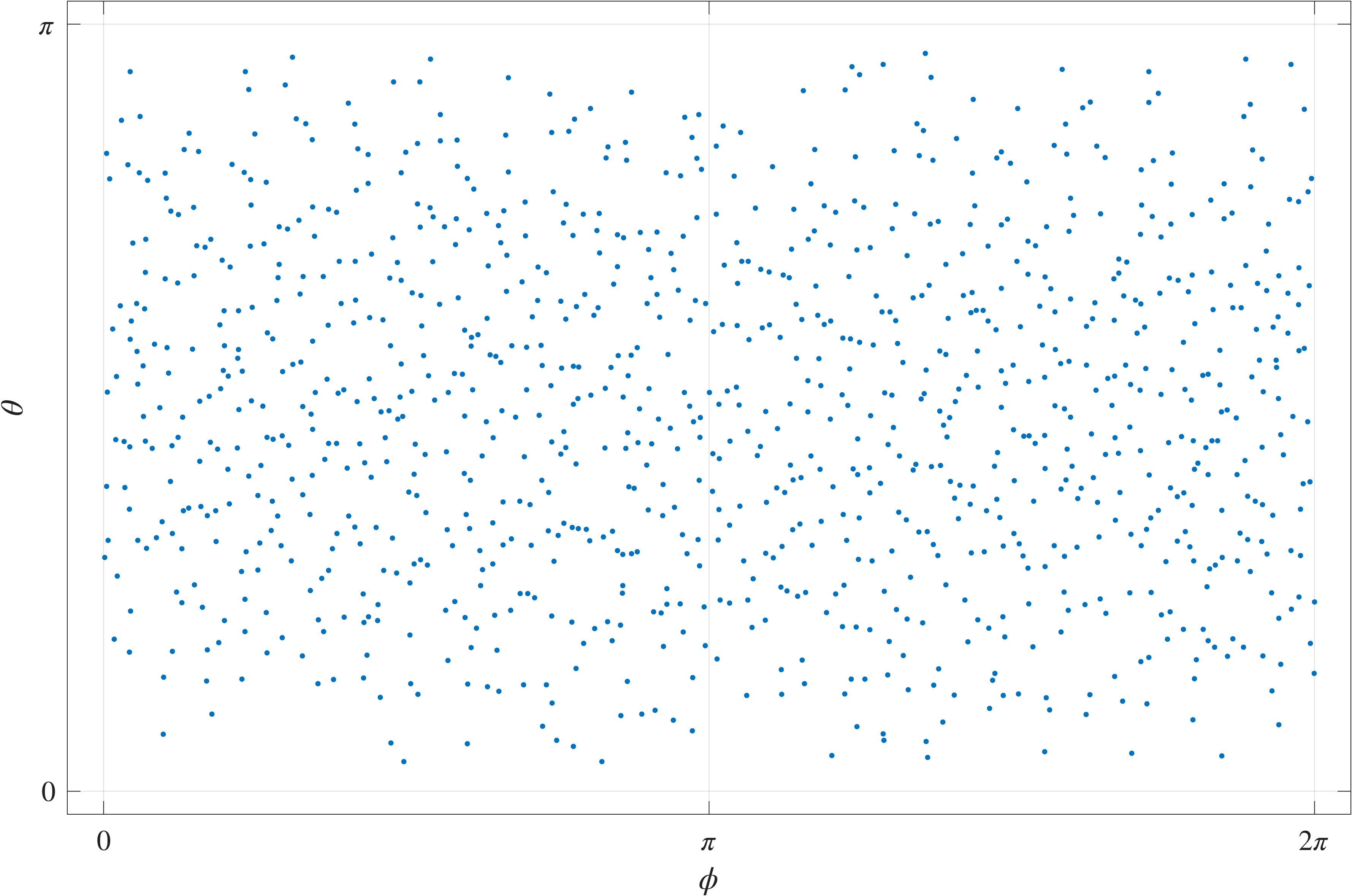}
        \caption{MAGSAT points.}
    \end{subfigure}
    \caption{Node configurations of $N = 1,024$ Halton and MAGSAT points in spherical coordinates.}
    \label{fig:2d-view}
\end{figure}

As discussed in \Cref{sec:framework}, to ensure a fair comparison among approximation methodologies, all collocation methods are benchmarked at the same level of exactness, specified by the critical degree $n^+$ \eqref{eq:critical-exactness}: the maximum degree supported by the spatial distribution of $X_N$ with strictly positive weights. The degree $n^+$ intrinsically reflects the geometric capacity of the scattered sites, and it scales very differently for our two datasets. 

\newpage
To acquire the critical degree $n^+$, we solve the least-squares problem with the constraints $\mathbf{Y}_n\mathbf{w} = \mathbf{b}_n$ and $\mathbf{w} \geq \mathbf{0}$, performing bisections until we find the maximum $n$ that still permits a feasible solution. \cref{fig:n_plus} illustrates the scaling of this critical degree $n^+$ as the number of points $N$ increases for both datasets. We observe that Halton points consistently support a higher $n^+$ than MAGSAT points. The latter prematurely exhaust the linear independence of the basis matrix $\mathbf{Y}_n$, forcing $n^+$ to be a much lower threshold. Since the exactness constraints are fixed at the critical degree $n^+$, the performance differences stem from the choice of the optimization objective rather than algebraic exactness.

\begin{figure}[htbp]
    \centering
    \includegraphics[width=0.65\textwidth]{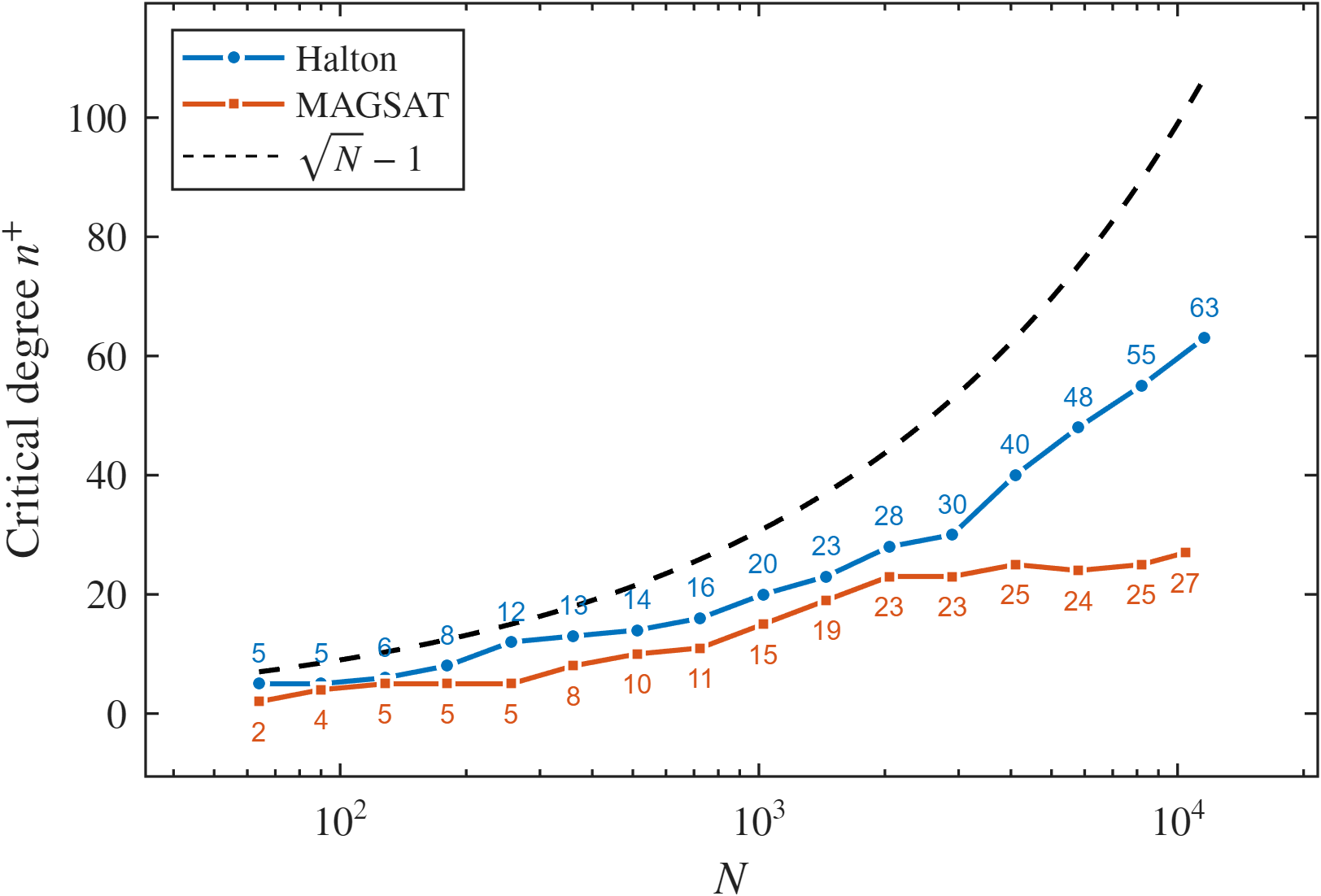}
    \caption{Scaling of the critical degree $n^+$ against the number of points $N$.}
    \label{fig:n_plus}
\end{figure}

\subsection{Validation of Stability}

Before evaluating the numerical performance for standard approximation tasks, we first empirically validate the theoretical assertions governing the stability of our optimization models. Specifically, we examine the spectral behavior of the discrepancy matrix governing the kernel framework (cf. \Cref{sec:spectral-stability}) and the deterioration of the MZ constants governing hyperinterpolation stability (cf. \Cref{sec:perturbation}).

\subsubsection{Conditioning of the Discrepancy Matrix}
\label{sec:conditioning}
In the discrepancy collocation \eqref{eq:discrepancy-collocation}, numerical stability is intrinsically tied to the conditioning of the matrix $\mathbf{K}_s$. Our spectral analysis determines two primary asymptotic behaviors driven by the functional smoothness index $s$: exponential growth of the condition number $\mathrm{cond}(\mathbf{K}_s)$ as $s \to \infty$ (cf. \cref{thm:conditioning}), and a degeneration into $\ell_2$-minimization as $s \to 0^+$ (cf. \Cref{sec:asymptotic-recovery}).

To verify the exponential growth, we compute the extremal eigenvalues $\lambda_1$ and $\lambda_N$, and the resulting condition number of $\mathbf{K}_s$ across a discrete grid of smoothness parameters $s \in [0.05, 1.95]$. For the low-smoothness cases ($0 \leq s \leq 1)$, we employ the truncated discrepancy matrix $\mathbf{K}_s^L$. To guarantee the positive definiteness in this regime, we set the truncation degree as the theoretically justified threshold $L = \lceil 55/q_{X_N} \rceil$ established in \cref{thm:conditioning-truncation}. For the high-smoothness cases ($s > 1$), we utilize the closed-form expression of the generalized distance kernel defined for $1 < s < 2$ (cf. \cref{tab:closed-form-kernels}) to facilitate computation.

The empirical spectral statistics for $N = 1,024$ nodes are plotted in \cref{fig:stability_analysis}. More specifically, \cref{fig:spectral_decay} illustrates the mechanism driving the ill-conditioning: As the assumed smoothness $s$ increases, the extremal eigenvalues diverge, causing a widening of the spectral gap. \cref{fig:condition_growth} provides confirmation of the resulting exponential growth in condition numbers for $s > 1$, in agreement with \cref{thm:conditioning}.
\begin{figure}[htbp]
     \centering
     \begin{subfigure}[b]{0.48\textwidth}
         \centering
         \includegraphics[width=\textwidth]{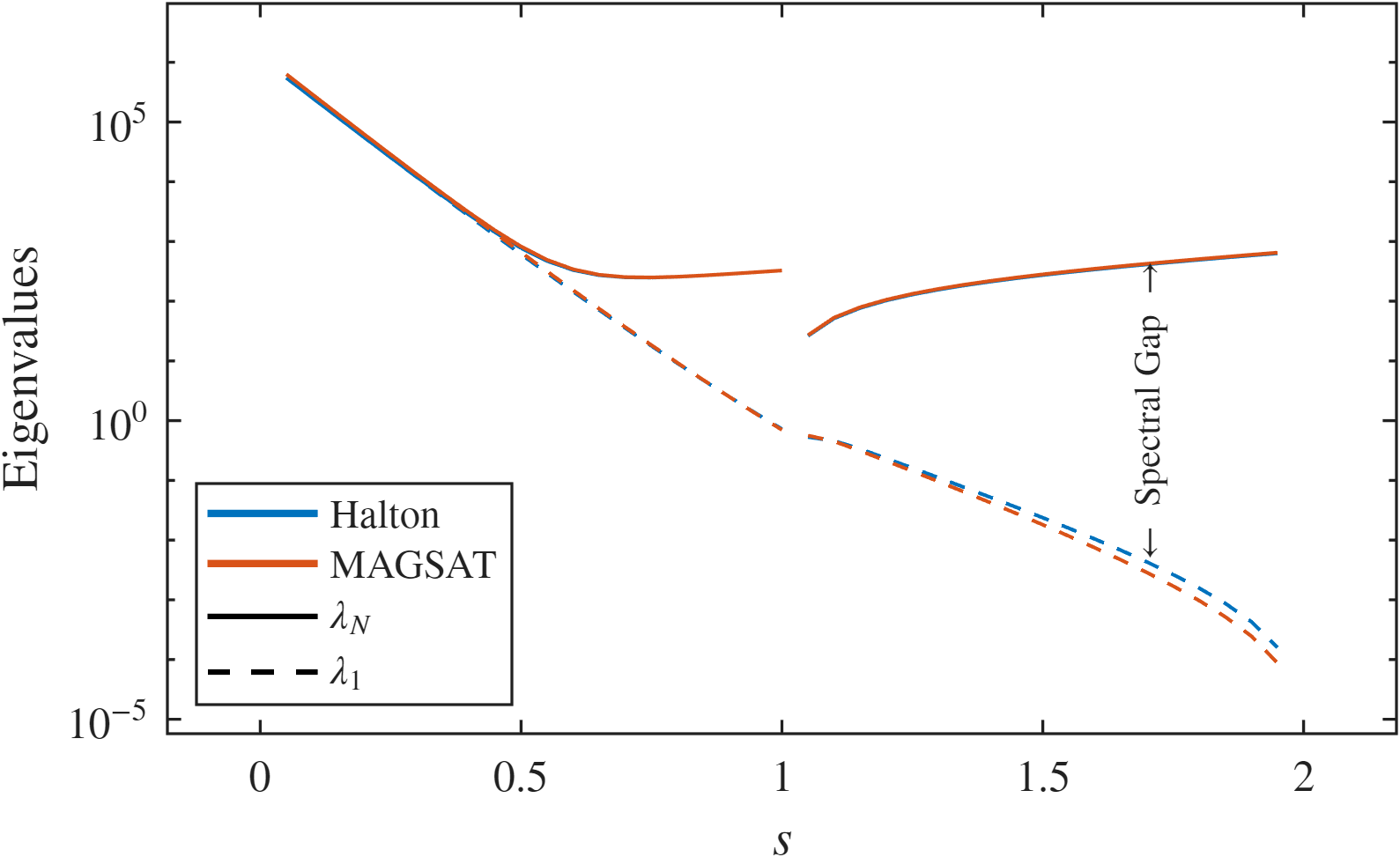}
         \caption{Spectral decay of $\lambda_{\max}$ and $\lambda_{\min}$.}
         \label{fig:spectral_decay}
     \end{subfigure}
     \hfill
     \begin{subfigure}[b]{0.48\textwidth}
         \centering
         \includegraphics[width=\textwidth]{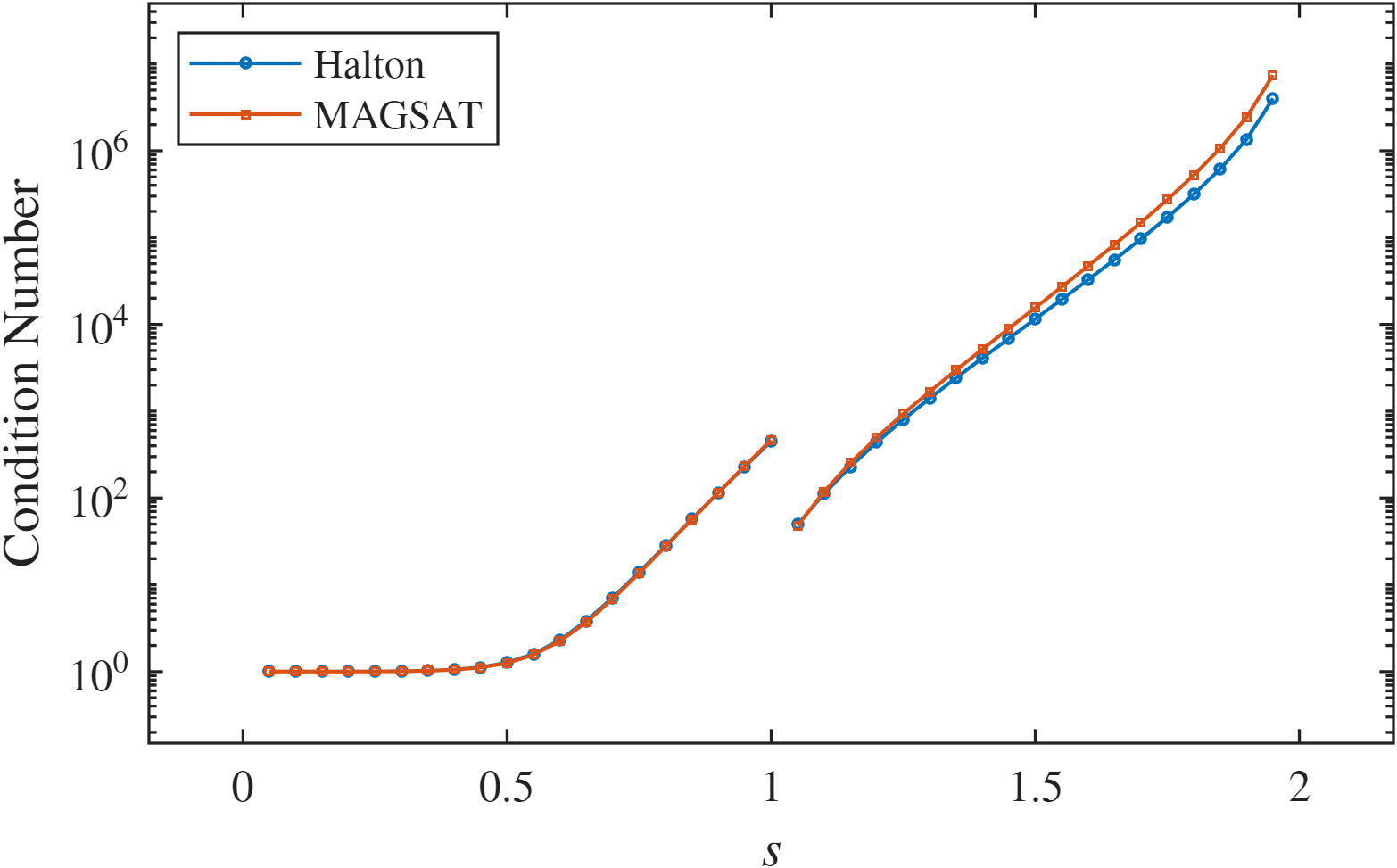}
         \caption{Growth of the condition number.}
         \label{fig:condition_growth}
     \end{subfigure}
     
     \caption{Spectral properties of the discrepancy matrix $\mathbf{K}_s$ for $N = 1,024$ Halton and MAGSAT points.}
     \label{fig:stability_analysis}
\end{figure}

\newpage
Next, we investigate the opposite asymptotic limit as $s \to 0^+$. Let $\mathbf{w}_s$ denote the weights obtained via the discrepancy collocation \eqref{eq:discrepancy-collocation} for a given smoothness $s$, and $\mathbf{w}_{\ell_2}$ the baseline minimum-norm weights obtained via the standard $\ell_2$-minimization \eqref{eq:l2-minimization}. To quantify the limiting behavior in \Cref{sec:asymptotic-recovery}, we track the relative error $\|\mathbf{w}_s - \mathbf{w}_{\ell_2}\|_2/\|\mathbf{w}_{\ell_2}\|_2$. \cref{fig:w-conv} illustrates this metric evaluated on the MAGSAT dataset as $s$ is decreased from $1.95$ down to $0.05$.

\begin{figure}[htbp]
     \centering
     \includegraphics[width=0.6\textwidth]{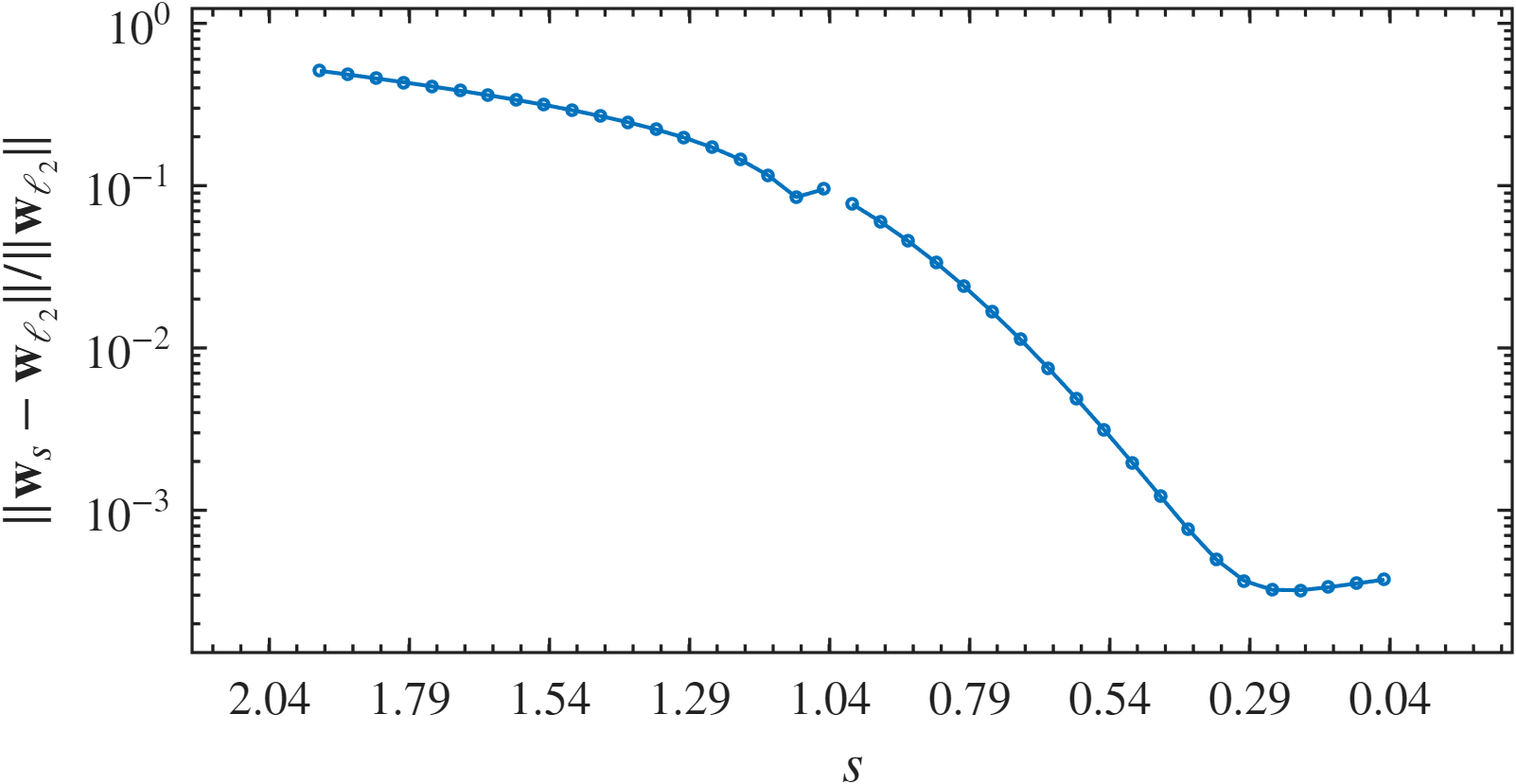}
     \caption{Relative error $\|\mathbf{w}_s - \mathbf{w}_{\ell_2}\|_2 / \|\mathbf{w}_{\ell_2}\|_2$ between the weights $\mathbf{w}_s$ generated by the discrepancy collocation \eqref{eq:discrepancy-collocation} and the $\ell_2$ baseline $\mathbf{w}_{\ell_2}$ \eqref{eq:l2-minimization} on $N=1,024$ MAGSAT points.}
     \label{fig:w-conv}
\end{figure}

As $s$ decreases towards zero, the relative error decays sharply, indicating that the collocated weights converge toward the solution to \eqref{eq:l2-minimization}. The subsequent error plateau near $10^{-4}$ is a numerical artifact: as $s \to 0^+$, the diagonal entries of the truncated discrepancy matrix $\mathbf{K}_s^L$ become large, introducing standard floating-point precision limitations within the optimization solver. Nonetheless, the overall trend is consistent with our theoretical analysis: when functional smoothness is absent, the kernel-based framework defaults to classical $\ell_2$-minimization.

\subsubsection{Effectiveness of the Geometry-Aware Regularizer} \label{sec:efficacy}

Recall that the $\chi^2$-divergence $\|\mathbf{w}\|_{\mathcal{R}, 2}$ is introduced in \Cref{sec:perturbation} as a geometry-aware regularizer to limit the deterioration of MZ constants, when quadrature weights deviate from a good geometric prior (typically chosen as the Voronoi weights $\mathbf{v}$). To illustrate this deterioration, we track the MZ constants $A$, $B$, and $c$, along a linear deformation path $\mathbf{w}(\alpha) = (1-\alpha)\mathbf{v} + \alpha \mathbf{w}_{\ell_2}$ for $\alpha \in [0, 5]$, where $\mathbf{w}_{\ell_2}$ represents the weight obtained from $\ell_2$-minimization \eqref{eq:l2-minimization}. 

\begin{figure}[htbp]
     \centering
     \includegraphics[width=0.5\textwidth]{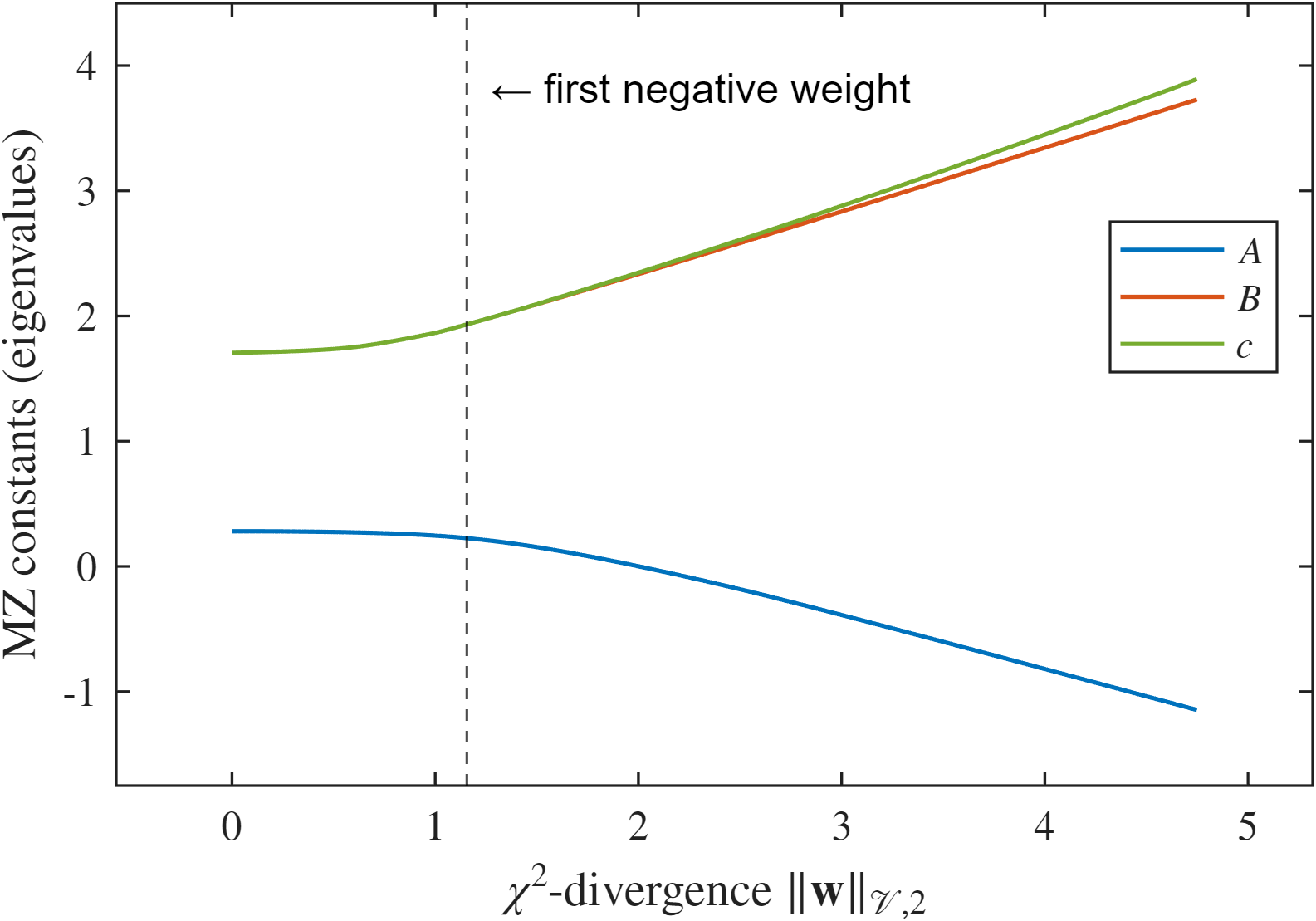}
     \caption{Deterioration of the MZ constants along the deformation path $\mathbf{w}(\alpha) = (1-\alpha)\mathbf{v} + \alpha\mathbf{w}_{\ell_2}$ against the $\chi^2$-divergence $\|\mathbf{w}\|_{\mathcal{V}, 2}$. The vertical dashed line marks the first occurrence of negative weights.}
     \label{fig:mz-deterioration}
\end{figure}

\cref{fig:mz-deterioration} plots the evolution of these MZ constants against the $\chi^2$-divergence $\|\mathbf{w}\|_{\mathcal{\mathcal{V}}, 2}$. As the deviation from the Voronoi prior grows, all constants deteriorate monotonically. As long as the quadrature weights remain strictly non-negative (near $\alpha = 0$), the constant $B$ in the $L^2$ MZ inequality and the constant $c$ in the $L^2$ MZ condition coincide. As the negative weights begin to emerge further along the deformation path, their trajectories diverge. The constant $A$ decline steady linearly. Concurrently, both $B$ and $c$ increase linearly.

It is worth noting that the perturbation bounds derived in \Cref{sec:perturbation} are inherently worst-case guarantees and, as expected, are quantitatively loose. The observed deterioration in \Cref{fig:mz-deterioration} occurs at a rate much slower than what the theoretical bounds indicate. Nevertheless, the numerical results are consistent with our theoretical analysis in the sense that the deterioration of the MZ constants $A$, $B$, and $c$ exhibit a linear relationship with respect to the $\chi^2$-divergence.

To demonstrate the practical effectiveness of the geometry-aware regularizer, we fix $N = 1,024$ and set the degree of the MZ inequalities as $n = 15$, a choice made so that $2n$ is close to $\lfloor\sqrt{N} - 1 \rfloor$. \cref{tab:regularizer_comparison} compares the resulting MZ constants obtained via the $\ell_2$-minimization \eqref{eq:l2-minimization} and those generated by the $\chi^2$-divergence collocation \eqref{eq:voronoi-collocation} with $\mathcal{R}$ chosen as the Voronoi partition $\mathcal{V}$.

\begin{table}[htbp]
\centering
\caption{Comparison of the MZ constants and absolute weight sum $\|\mathbf{w}\|_1$ of the Voronoi weight $\mathbf{v}$, weight obtained by $\ell_2$-minimization \eqref{eq:l2-minimization}, and $\chi^2$-divergence collocation \eqref{eq:voronoi-collocation} with $\mathcal{R}$ chosen as the Voronoi partition $\mathcal{V}$.}
\label{tab:regularizer_comparison}
\begin{tabular}{ccclccccc}
\toprule
Dataset & $n$ & $n^+$ & $P(\mathbf{w})$ & $A$ & $B$ & $c$ & $\eta$ & $\|\mathbf{w}\|_1$ \\
\midrule
\multirow{3}{*}{Halton} & \multirow{3}{*}{1024} & \multirow{3}{*}{20}
& $\mathbf{v}$                        & 0.2806 & 1.7058 & 1.7058 & 0.7194 & 12.5664 \\
\cmidrule{4-9}
& & & $\|\mathbf{w}\|_2$               & 0.2505 & 1.8483 & 1.8483 & 0.8483 & 12.5664 \\
& & & $\|\mathbf{w}\|_{\mathcal{V},2}$ & 0.2558 & 1.7745 & 1.7745 & 0.7745 & 12.5664 \\
\midrule
\multirow{3}{*}{MAGSAT} & \multirow{3}{*}{1024} & \multirow{3}{*}{15}
& $\mathbf{v}$                        & 0.3564 & 1.5808 & 1.5808 & 0.6436 & 12.5664 \\
\cmidrule{4-9}
& & & $\|\mathbf{w}\|_2$               & 0.2737 & 2.0135 & 2.0135 & 1.0135 & 12.5664 \\
& & & $\|\mathbf{w}\|_{\mathcal{V},2}$ & 0.3430 & 1.7844 & 1.7844 & 0.7844 & 12.5664 \\
\bottomrule
\end{tabular}
\end{table}

The results in \cref{tab:regularizer_comparison} demonstrate the effectiveness of the geometric-aware regularizer. Across both datasets, the $\ell_2$ baseline yields the more severe deterioration of the MZ constants, resulting in the larger values of $c$. In contrast, incorporating the geometric regularizer stabilizes $c$, thereby improving the numerical stability of the hyperinterpolation operator. A welcome byproduct of the regularizer derived from the perturbation bound \cref{prop:perturb-III} is that the accuracy constant $\eta$ is also consistently reduced, yielding better approximation quality for polynomials.

\subsection{Performance in Numerical Integration} \label{sec:integration}
We next evaluate the empirical performance of all considered collocation methods on the fundamental task of numerical integration over the sphere. To assess the accuracy of these quadratures on functions with distinct spectral characteristics, we select two generic test functions from \cite{renka1988multivariate}:
\begin{align*}
  f_1(x, y, z) & := 0.75\exp(-(9x-2)^2/4 - (9y-2)^2/4 - (9z-2)^2/4) \\
                 & \quad + 0.75 \exp(-(9x+1)^2/49 - (9y+1)/10 - (9z+1)/10) \\
                 & \quad + 0.5 \exp(-(9x-7)^2/4 - (9y-3)^2/4 - (9z-5)^2/4) \\
                 & \quad - 0.2 \exp(-(9x-4)^2 - (9y-7)^2 - (9z-5)^2), & \mathbf{x} = (x,y,z)^\top \in \mathbb{S}^2.\\
  f_2(x, y, z) & := (1+\tanh(-9x-9y+9z))/9, & \mathbf{x} = (x,y,z)^\top \in \mathbb{S}^2.
\end{align*}
\cref{fig:sph-coeff-int} plots the magnitude of the spherical harmonic coefficients $|\hat{f}_{\ell, k}|$ of the two functions for $\ell$ up to $128$. Both $f_1$ and $f_2$ are $C^\infty$ functions. While the Franke function $f_1$ exhibits an exponential decay of its spherical harmonic coefficients, the coefficients for the function $f_2$ decay at a much slower rate.

\begin{figure}[htbp]
    \centering
    \begin{subfigure}[b]{0.49\textwidth}
        \centering
        \includegraphics[width=\textwidth]{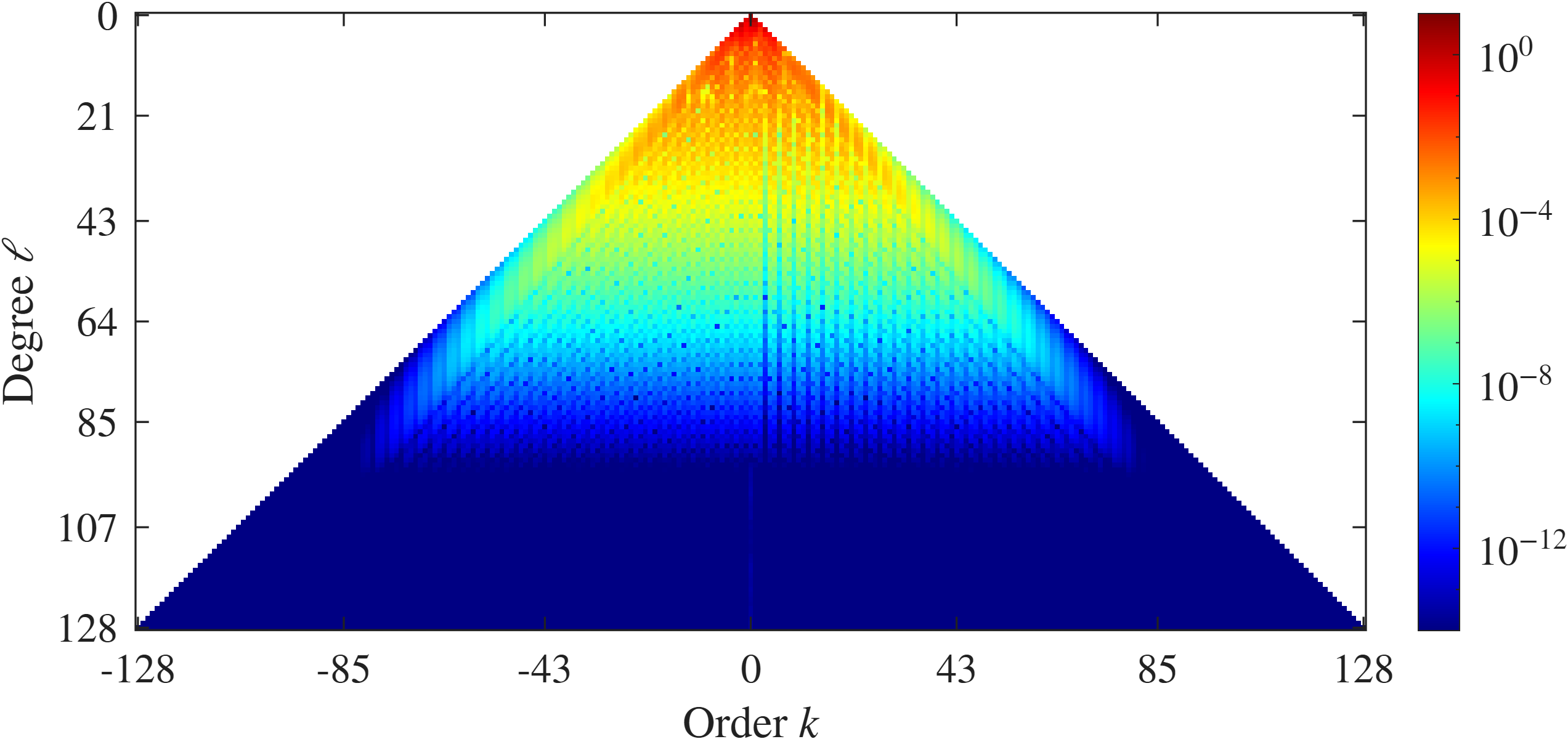}
        \caption{The Franke function $f_1$}
    \end{subfigure}
    \hfill
    \begin{subfigure}[b]{0.49\textwidth}
        \centering
        \includegraphics[width=\textwidth]{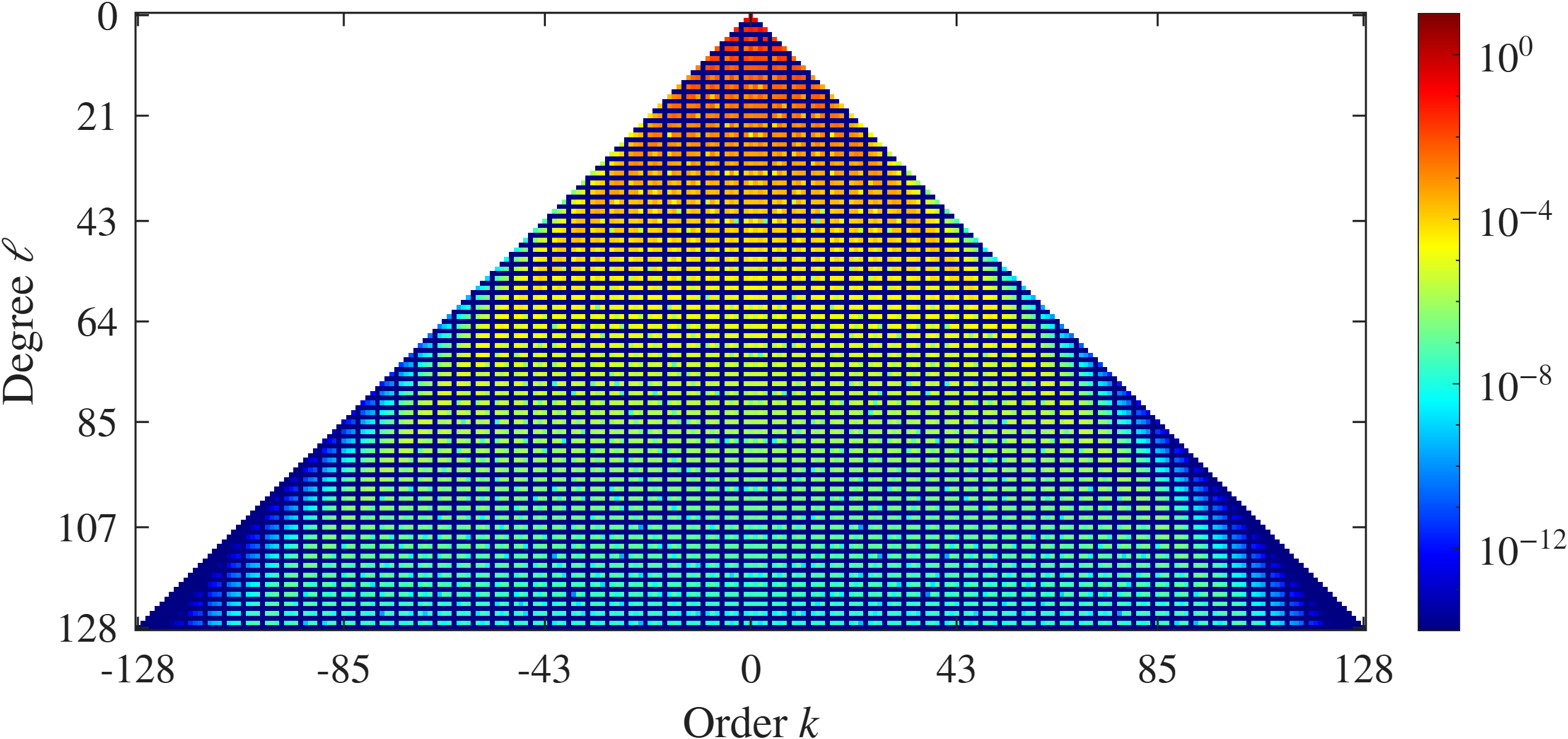}
        \caption{The function $f_2$}
        \label{fig:coeff_f2}
    \end{subfigure}
    \caption{Magnitudes of the spherical harmonic coefficients $|\hat{f}_{\ell,k}|$ for degrees $\ell = 0, \ldots, 128$.}
    \label{fig:sph-coeff-int}
\end{figure}

Note that the integrals of these functions over the unit sphere evaluate to:
\begin{align*}
    I(f_1) & = \int_{\mathbb{S}^2} f_1(\mathbf{x})\,d\omega(\mathbf{x}) \approx 9.95474164740918, \\
    I(f_2) & = \int_{\mathbb{S}^2} f_2(\mathbf{x})\,d\omega(\mathbf{x}) = 4\pi/9.
\end{align*}
To benchmark our collocation methodologies, we utilize the following two classical quadratures as baselines for both the Halton and MAGSAT point sets:
\begin{itemize}[noitemsep]
    \item The first is the purely geometric Voronoi weight $\mathbf{v}$ (Vor), generated directly from the Voronoi partition $\mathcal{V}$. It gives a quadrature exact of degree $0$.
    \item The second baseline ($\ell_2^+$) represents the classical pursuit of positive weights, which is defined as the positive minimum-norm solution to the exactness system
\begin{align} \label{eq:l2-minimization-positive}
    \begin{aligned}
        \min \quad & \|\mathbf{w}\|_2^2 \\
        \text{s.t.} \quad & \mathbf{Y}_{n^+} \mathbf{w} = \mathbf{b}_{n^+}, \quad \mathbf{w} \geq \mathbf{0},
    \end{aligned}
\end{align}
where $n^+$ is the critical degree defined in \eqref{eq:critical-exactness}.
\end{itemize}

We compare these classical baselines with our three proposed kernel collocation methodologies. To ensure a fair comparison, all optimization models are constrained to the critical degree $n^+$:
\begin{itemize}[noitemsep]
    \item \textit{Cui and Freeden discrepancy collocation (CF)}: Formulated via \eqref{eq:cui-freeden-discrepancy-collocation}, representing the smooth regime ($s = 1.5$) with the exact closed-form logarithmic kernel.
    \item \textit{Truncated discrepancy collocation (TD)}: Formulated via \eqref{eq:truncated-kernel-collocation}, targeting a smoother space ($s = 2.5$). To guarantee the positive definiteness of the discrepancy matrix, the series is truncated at the theoretical threshold $L = \lceil 55/q_{X_N}\rceil$ derived in \cref{thm:conditioning-truncation};
    \item \textit{Bandlimited collocation (BL)}: Formulated via \eqref{eq:bandlimited-collocation}, engineered for highly clustered sites. We set the recovery bandwidth as a moderate $L = \lceil 0.75 N \rceil$ and utilize the $\ell_2$-regularizer $R(\mathbf{w}) = \frac{1}{2}\|\mathbf{w}\|_2^2$ with a regularization strength of $\lambda = 0.01$.
\end{itemize}

\cref{tab:quadrature_integration} summarizes the resulting weight vectors for a representative node count of $N = 8,192$. Our kernel collocation methods produce signed quadrature weights, but variance of the weights and the overall $\|\mathbf{w}\|_1$ remains controlled. Note also how the tiny $q_{X_N}$ of Halton points affects the truncation level of (TD). This necessitates the consideration of (BL) for highly clustered scattered sets.

\begin{table}[htbp]
\centering
\caption{Weight statistics for $N = 8,192$ Halton and MAGSAT point sets for numerical integration. In this table, $L$ is the truncation level for (TD) and (BL). The columns also report the minimum (min) and maximum (max) weight, the absolute sum $\|\mathbf{w}\|_1$, variance (var), number of positive weights (pos).}
\label{tab:quadrature_integration}
\begin{adjustbox}{max width=\textwidth}
\begin{tabular}{ccccccccc}
\toprule
Point & Quadrature & $n^+$ & $L$ & min & max & $\|\mathbf{w}\|_1$ & var & pos \\ \midrule
\multirow{5}{*}{Halton} & Vor & $0$ & - & \phantom{-}4.3186e-04 & \phantom{-}3.0574e-03 & 12.5664 & 1.3966e-07 & 8192 \\
 & $\ell_2^+$ & $55$ & - & \phantom{-}1.5180e-04 & \phantom{-}3.4838e-03 & 12.5664 & 2.0278e-07 & 8192 \\
 & CF & $55$ & - & -9.4273e-05 & \phantom{-}3.6096e-03 & 12.5666 & 2.7870e-07 & 8191 \\
 & TD & $55$ & $13185$ & -2.6225e-03 & \phantom{-}4.6479e-03 & 12.9238 & 6.8950e-07 & 7894 \\
 & BL & $55$ & $68$ & -1.6878e-03 & \phantom{-}4.5766e-03 & 12.7028 & 6.6674e-07 & 7981 \\ \midrule
\multirow{5}{*}{MAGSAT} & Vor & $0$ & - & \phantom{-}7.3265e-04 & \phantom{-}6.9502e-03 & 12.5664 & 2.2788e-07 & 8192 \\
 & $\ell_2^+$ & $25$ & - & \phantom{-}8.0965e-05 & \phantom{-}5.6067e-03 & 12.5664 & 1.5921e-07 & 8192 \\
 & CF & $25$ & - & -1.0338e-03 & \phantom{-}1.1095e-02 & 12.5876 & 3.8534e-07 & 8168 \\
 & TD & $25$ & $4202$ & -1.2372e-02 & \phantom{-}2.2455e-02 & 13.1143 & 1.4428e-06 & 7982 \\
 & BL & $25$ & $68$ & -7.5313e-03 & \phantom{-}1.2793e-02 & 13.2511 & 1.4799e-06 & 7649 \\ \bottomrule
\end{tabular}
\end{adjustbox}
\end{table}

\Cref{fig:integration_convergence} reports the relative integration error $e(Q,f) := |Q[X_N,\mathbf{w}](f) - I(f)|/|I(f)|$ across varying sample sizes $N$ for both test functions and datasets. We observe notable improvement of (CF), (TD), and (BL) over the simple (Vor) and ($\ell_2^+$) baselines. This improvement is measurable on the Halton points, but it becomes significant on the MAGSAT dataset. This disparity aligns with our theoretical expectations. Because Halton points inherently form a low-discrepancy sequence, they are closer to quadrature optimality, offering limited gains for discrepancy optimization. However, the performance on the heavily clustered MAGSAT points, particularly for the wider-spectrum function $f_2$, is particularly revealing. In this geometrically challenging regime, our kernel collocation methods (especially (TD) and (BL)) outperform both classical baselines by one to two orders of magnitude.

\newpage
\begin{figure}[htbp]
    \centering
    \begin{subfigure}[b]{0.48\textwidth}
        \centering
        \includegraphics[width=\textwidth]{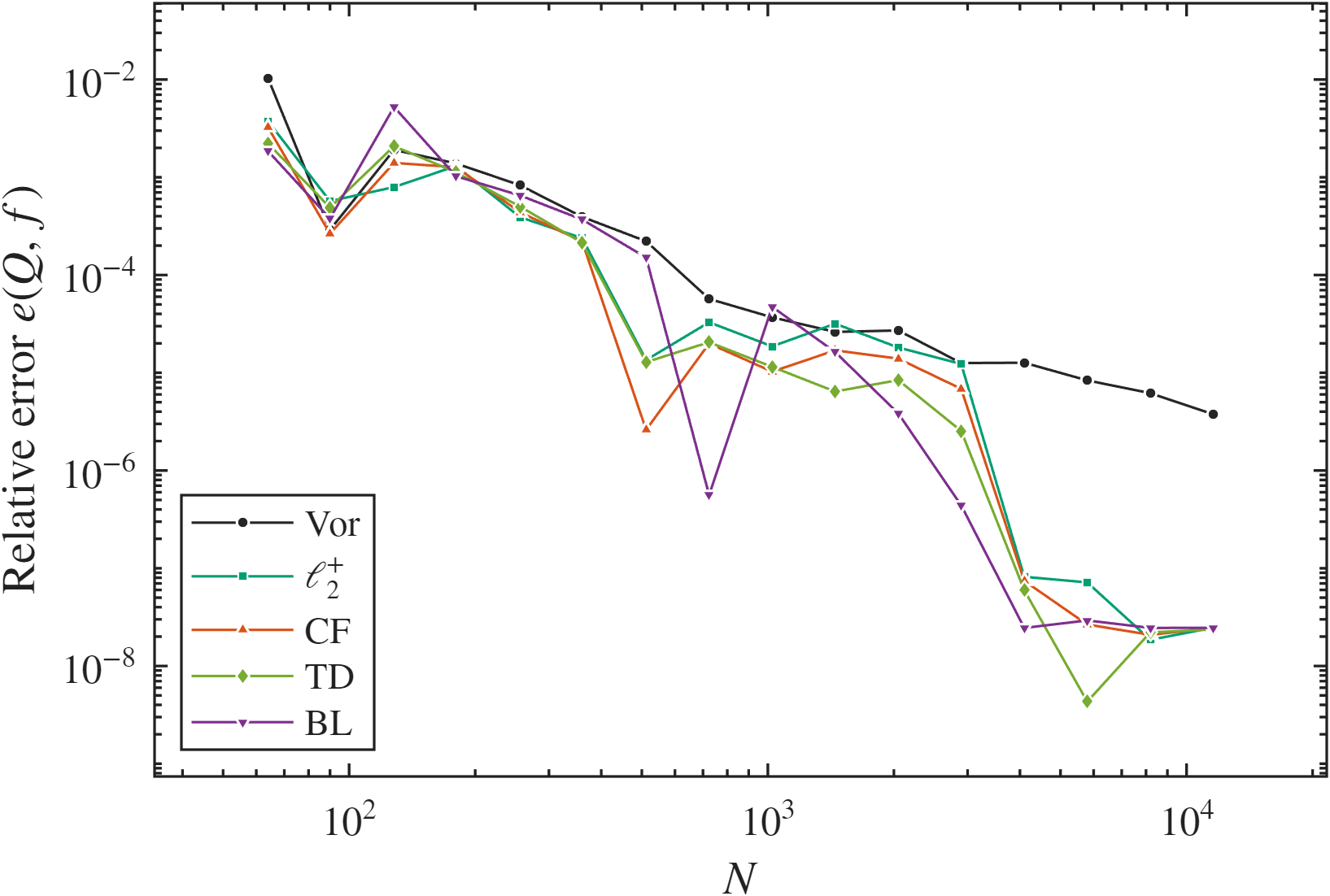}
        \caption{$f_1$, Halton.}
        \label{fig:integration_franke_halton}
    \end{subfigure}
    \hfill
    \begin{subfigure}[b]{0.48\textwidth}
        \centering
        \includegraphics[width=\textwidth]{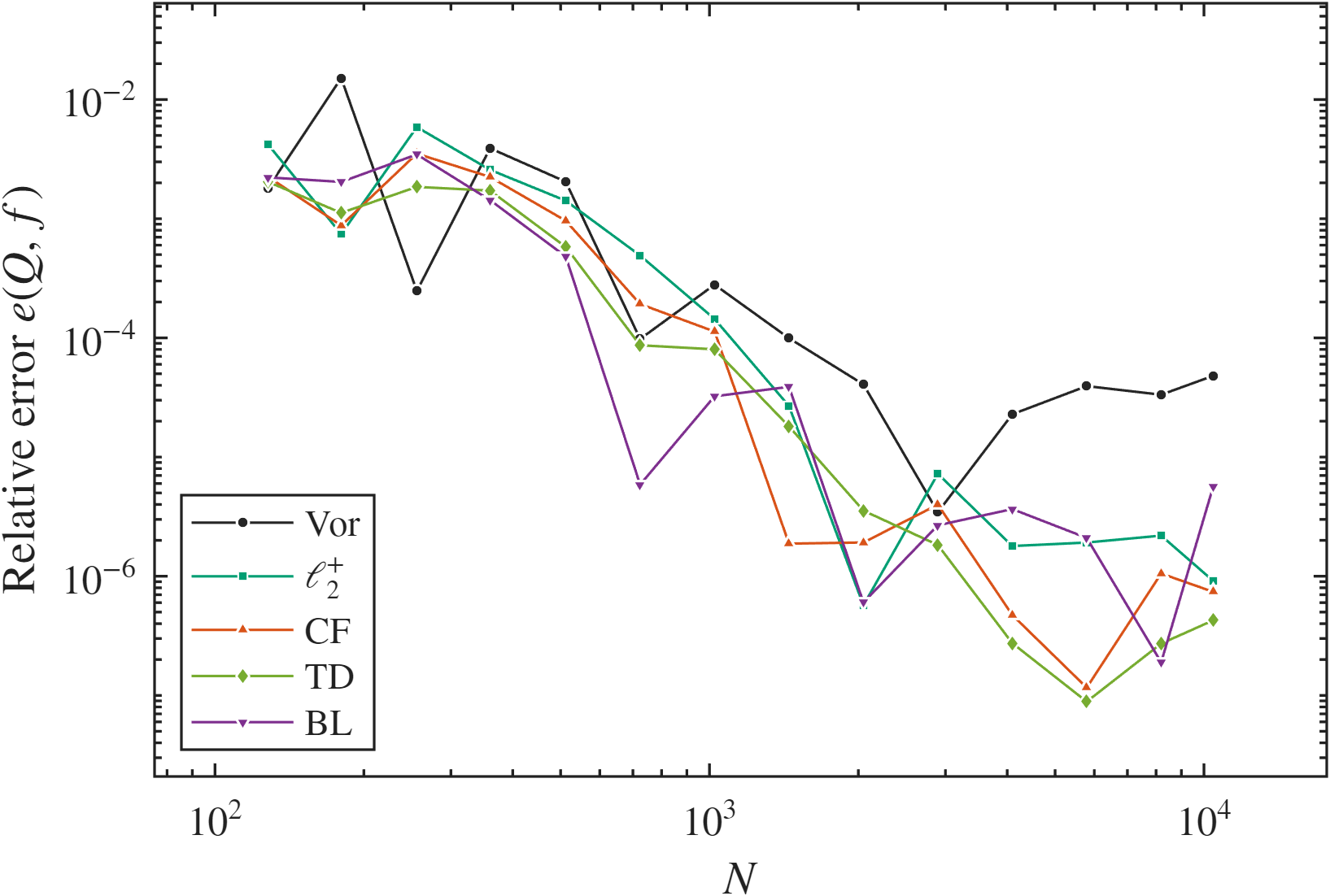}
        \caption{$f_1$, MAGSAT.}
        \label{fig:integration_franke_magsat}
    \end{subfigure}
    \vskip0.75em
    \begin{subfigure}[b]{0.48\textwidth}
        \centering
        \includegraphics[width=\textwidth]{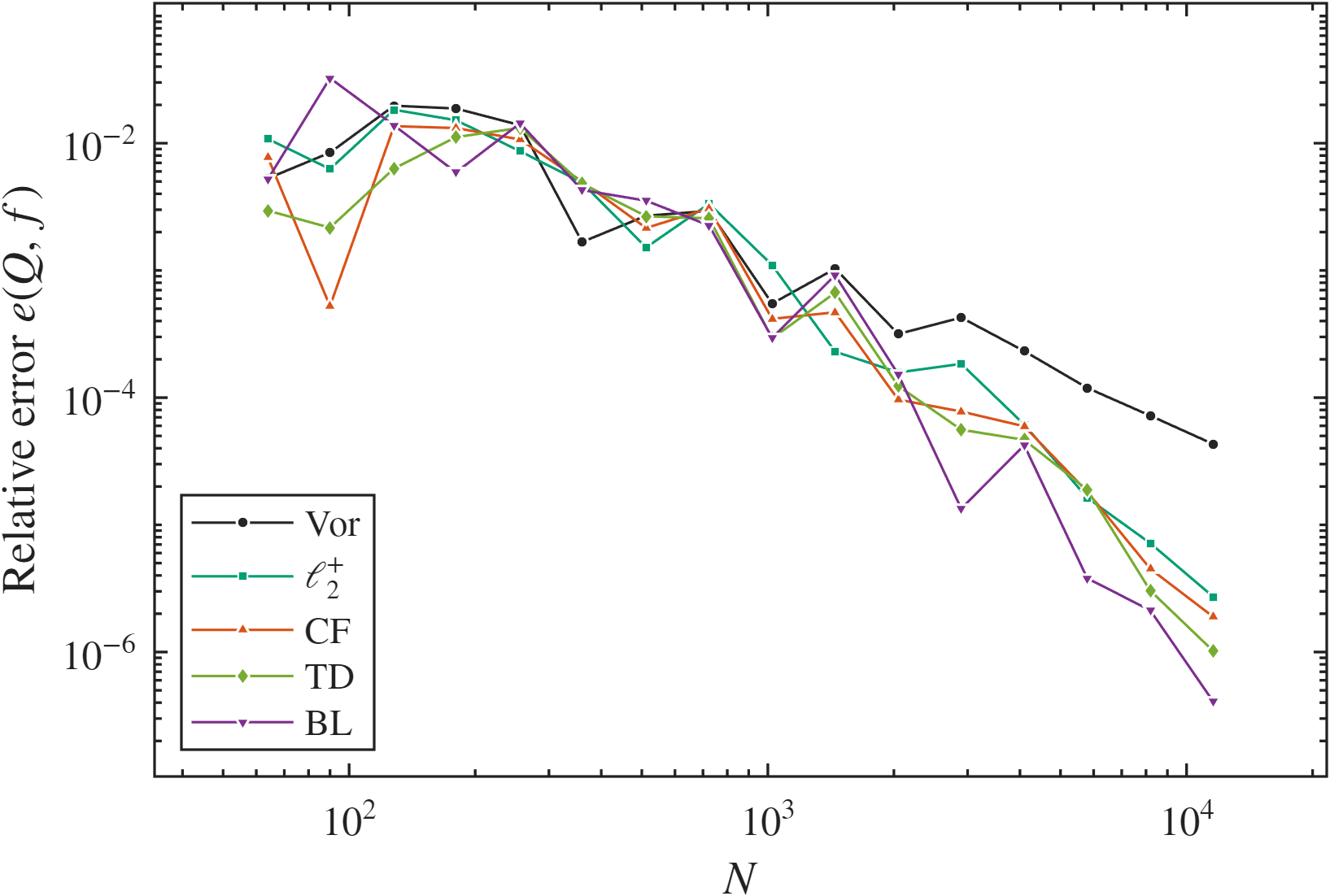}
        \caption{$f_2$, Halton.}
        \label{fig:integration_renka2_halton}
    \end{subfigure}
    \hfill
    \begin{subfigure}[b]{0.48\textwidth}
        \centering
        \includegraphics[width=\textwidth]{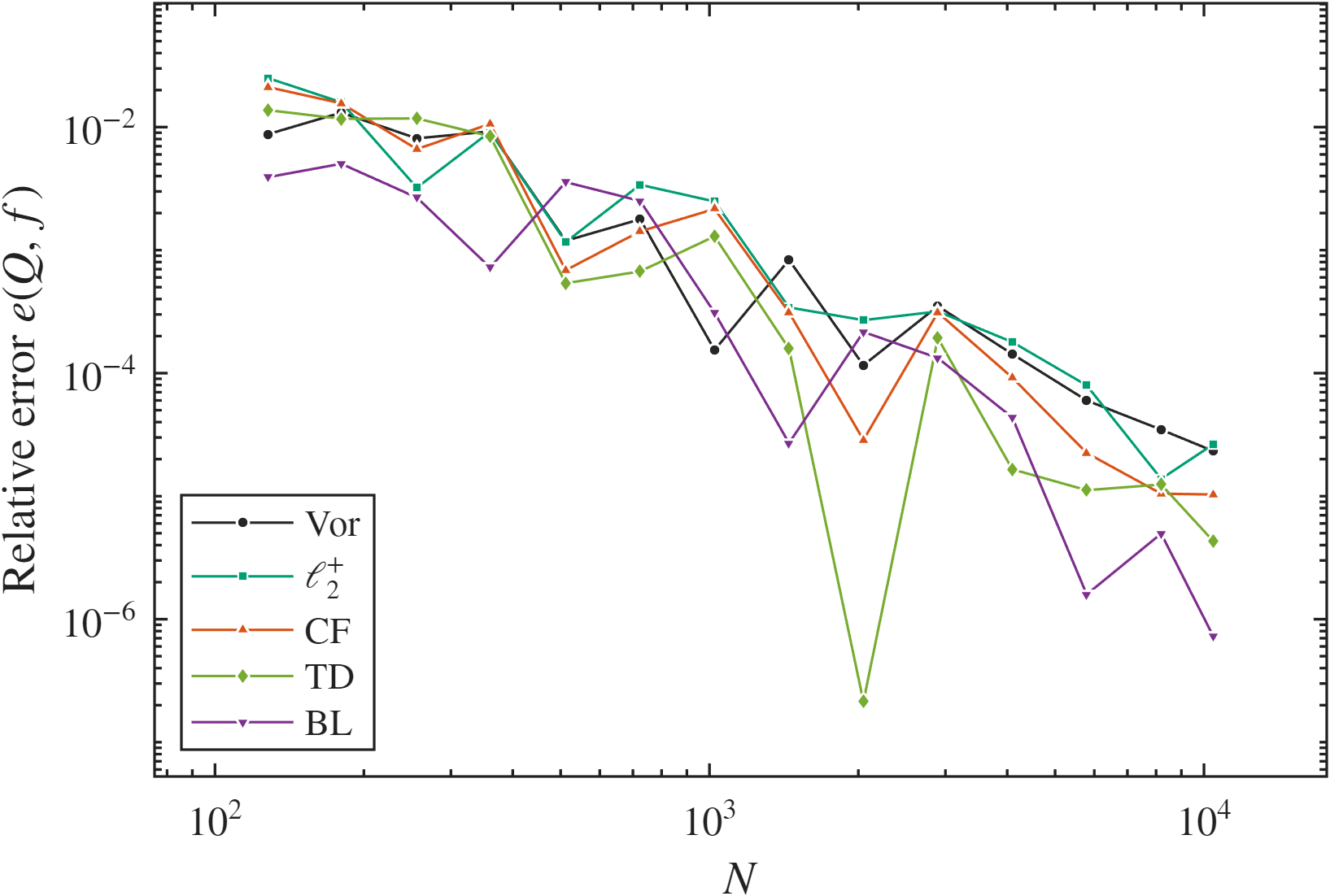}
        \caption{$f_2$, MAGSAT.}
        \label{fig:integration_renka2_magsat}
    \end{subfigure}
    \caption{Convergence of the relative integration error $e(Q,f)$ as a function of the number of points $N$ on Halton and MAGSAT datasets for test functions $f_1$ and $f_2$.}
    \label{fig:integration_convergence}
\end{figure}

These experimental results also reveal that (TD) utilizing $s = 2.5$ frequently outperforms the exact closed-form formulation (CF) utilizing $s = 1.5$ when integrating the infinitely smooth $C^\infty$ functions $f_1$ and $f_2$. This empirical observation motivates the investigation of how the chosen smoothness index $s$ controls the integration accuracy and solver stability of kernel collocation.

\subsubsection{Sensitivity in \texorpdfstring{$s$}{s}}
Consider the (TD) model on $N = 2,048$ Halton points. To observe the effect of $s$ more accurately, the exactness is fixed at the critical degree $n^+$, and the truncation level is enforced at the theoretical positive-definite threshold $L=\lceil 55/q_{X_N}\rceil$. We solve \eqref{eq:truncated-kernel-collocation} across a uniform grid of $s \in [1.1,3.5]$ and evaluate the relative integration error $e(Q, f)$ on the smooth Franke function $f_1$. 

\cref{fig:td_s_tradeoff} compares the theoretical conditioning bounds with the empirical integration errors. More specifically, \cref{fig:td_s_error} demonstrates that increasing $s$ yields a steady improvement in integration accuracy, as the optimization model targets increasingly smoother Sobolev spaces that better match the rapid spectral decay of $f_1$. However, this analytical benefit eventually saturates, and the error increases notably when numerical instability becomes significant. \cref{fig:td_s_condition} confirms the mechanism: as $s$ increases, the condition number $\mathrm{cond}(\mathbf{K}_s^L)$ grows at an exponential rate, in agreement with the theoretical upper bound in \cref{thm:conditioning}. In parallel, the severe ill-conditioning of the discrepancy matrix induces an inflation of the absolute weight sum $\|\mathbf{w}\|_1$.

\begin{figure}[htbp]
    \centering
    \begin{subfigure}[t]{0.42\textwidth}
        \centering
        \includegraphics[width=\textwidth]{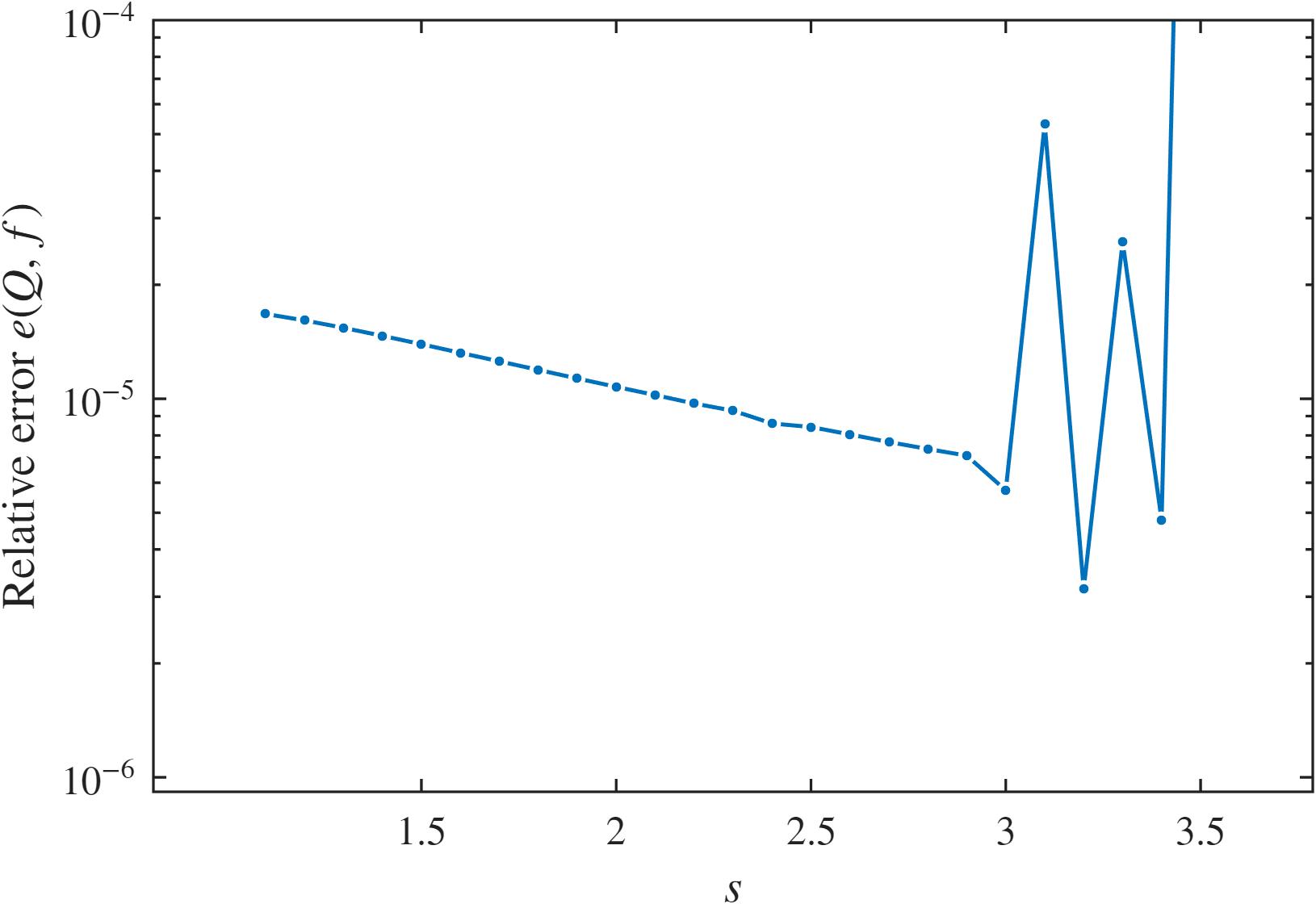}
        \caption{The relative integration error against the smoothness index $s$.}
        \label{fig:td_s_error}
    \end{subfigure}
    ~\hfill~
    \begin{subfigure}[t]{0.47\textwidth}
        \centering
        \includegraphics[width=\textwidth]{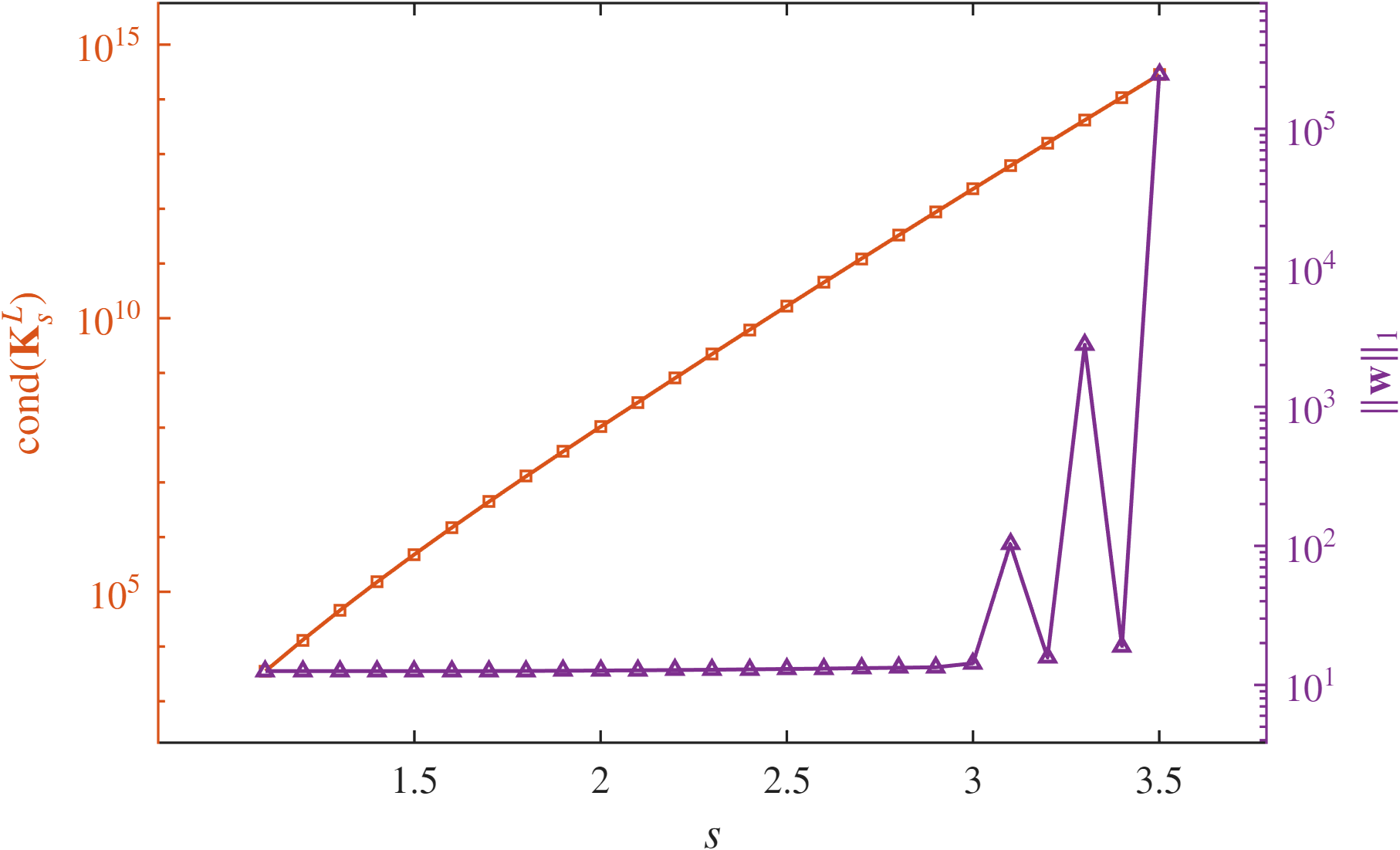}
        \caption{The condition number of the truncated discrepancy matrix $\mathrm{cond}(\mathbf{K}_s^L)$ and $\|\mathbf{w}\|_1$ against $s$.}
        \label{fig:td_s_condition}
    \end{subfigure}
    \caption{Sensitivity of the truncated discrepancy collocation (TD) to the smoothness index $s$. Evaluated on $N = 2,048$ Halton points, with exactness fixed at $n^+$ and truncation at $L=\lceil 54/q_{X_N}\rceil$.}
    \label{fig:td_s_tradeoff}
\end{figure}

These results are consistent with the observation in \cref{rem:exponential-conditioning} and \Cref{sec:conditioning}. Quadrature design inherently benefits from targeting smoother Sobolev spaces ($\mathbb{H}^s(\mathbb{S}^2)$ with larger $s$), provided that the target function possesses sufficient regularity. However, as the conditioning of the discrepancy matrix degrades exponentially, solver precision is compromised, and the quadrature performance degrades substantially. Therefore, $s$ should be chosen \textit{moderately} in the sense that it should be large enough to effectively capture the smoothness of the integrand, but small enough to prevent the exponential growth of ill-conditioning.

\subsection{Performance in Hyperinterpolation} \label{sec:hyperinterpolation}
Our second evaluation targets the accuracy and stability of the hyperinterpolation operator. We utilize the same geometric (Vor) and strictly positive ($\ell_2^+$) baseline quadratures as defined in \Cref{sec:integration}. In this experiment, we compare these baselines with our MZ collocation models, derived from the decomposition of hyperinterpolation error \eqref{eq:hyperinterpolation-error-decomposition}. Specifically, we consider \eqref{eq:MZ-collocation} with two choices of $J_n(\mathbf{w})$ as follows:
\begin{itemize}[noitemsep]
    \item \textit{Spectral Collocation (Spec)}: Utilizing the exact spectral norm $J_n(\mathbf{w}) = \|\mathbf{I} - \mathbf{G}_n[X_N](\mathbf{w})\|_2$, solved via
    a customized primal fixed-reduction path-following interior point method (cf. \cite{vandenberghe1998determinant}) with a regularization strength of $\lambda = 0.001$.
    \item \textit{$D$-optimal Collocation ($D$-opt)}: Utilizing the $D$-optimal surrogate $J_n(\mathbf{w}) = -\log\det \mathbf{G}_n[X_N](\mathbf{w})$, solved via the infeasible start Newton method (cf. \cref{alg:d_optimal}) with a regularization strength of $\lambda = 0.01$.
\end{itemize}

\newpage
To ensure a fair comparison, we again constrain all non-Voronoi quadratures to be exact of degree $n^+$. The degree $n$ of the target $L^2$ MZ inequality is chosen so that $2n$ is close to the limit $\lfloor \sqrt{N} - 1 \rfloor$. \cref{tab:quadrature_hyperinterpolation} provides comparisons of the computed quadrature weights and their resulting MZ constants for $N = 2,048$. The weights collocated by (Spec) and ($D$-opt) achieve smaller $\eta$ at the cost of introducing negative weights and larger variance. Nevertheless, the regularization effectively prevents the total weight norm $\|\mathbf{w}\|_1$ and $c$ from becoming excessively large.

\begin{table}[htbp]
\centering
\setlength{\tabcolsep}{2.5pt}
\caption{Weight statistics for $N = 2,048$ Halton and MAGSAT point sets for hyperinterpolation. In this table, $n$ denotes the degree of the target $L^2$ MZ inequality. The columns also report the minimum (min) and maximum (max) weight, the absolute sum $\|\mathbf{w}\|_1$, variance (var), number of positive weights (pos), MZ constants $A$, $B$, and $c$, and $\eta$ as in \eqref{eq:eta}.}
\label{tab:quadrature_hyperinterpolation}
\begin{adjustbox}{max width=\textwidth}
\begin{tabular}{ccccccccccccc}
\toprule
Point & $n$ & Quadrature & $n^+$ & min & max & $\|\mathbf{w}\|_1$ & var & pos & $c$ & $A$ & $B$ & $\eta$ \\ \midrule
\multirow{4}{*}{Halton} & \multirow{4}{*}{$22$} & Vor & $0$ & \phantom{-}1.64e-03 & \phantom{-}1.25e-02 & 12.57 & 2.22e-06 & 2048 & 1.3621 & 0.5577 & 1.3621 & 0.4423 \\
 & & $\ell_2^+$ & $28$ & \phantom{-}7.91e-05 & \phantom{-}1.48e-02 & 12.57 & 4.87e-06 & 2048 & 1.5036 & 0.5189 & 1.5036 & 0.5036 \\
 & & Spec & $28$ & -5.33e-02 & \phantom{-}5.21e-02 & 19.41 & 1.01e-04 & 1579 & 6.7280 & 0.8856 & 1.1144 & 0.1144 \\
 & & $D$-opt & $28$ & -4.15e-02 & \phantom{-}4.17e-02 & 18.56 & 8.51e-05 & 1602 & 5.2044 & 0.8094 & 1.1836 & 0.1906 \\ \midrule
\multirow{4}{*}{MAGSAT} & \multirow{4}{*}{$21$} & Vor & $0$ & \phantom{-}1.57e-03 & \phantom{-}1.42e-02 & 12.57 & 3.62e-06 & 2048 & 1.3984 & 0.2996 & 1.3984 & 0.7004 \\
 & & $\ell_2^+$ & $23$ & \phantom{-}4.05e-04 & \phantom{-}1.67e-02 & 12.57 & 4.20e-06 & 2048 & 1.6902 & 0.2438 & 1.6902 & 0.7562 \\
 & & Spec & $23$ & -1.26e-01 & \phantom{-}1.23e-01 & 23.48 & 2.34e-04 & 1520 & 17.3447 & 0.7908 & 1.2092 & 0.2092 \\
 & & $D$-opt & $23$ & -8.16e-02 & \phantom{-}7.30e-02 & 21.10 & 1.38e-04 & 1515 & 8.6353 & 0.5446 & 1.3669 & 0.4554 \\ \bottomrule
\end{tabular}
\end{adjustbox}
\end{table}

We evaluate the performance of the resulting hyperinterpolation operators on three distinct target functions. The first function is a smooth $C^\infty$ function adapted from Renka \cite{renka1988multivariate}
\begin{equation*}
    f_3(\mathbf{x}) := \frac{1}{3} \exp \left( -\frac{81}{16} \left( (x - 0.5)^2 + (y - 0.5)^2 + (z - 0.5)^2 \right) \right), \quad \mathbf{x} = (x,y,z)^\top \in \mathbb{S}^2.
\end{equation*}
The second is an oscillatory zonal plane wave
\begin{equation*}
    f_4(\mathbf{x}) := \cos(13 \langle \mathbf{p}, \mathbf{x}\rangle), \quad \mathbf{x} = (x,y,z)^\top \in \mathbb{S}^2,
\end{equation*}
where the defining pole $\mathbf{p}:=(1,2,3)^\top/\sqrt{14}$. The third function possesses lower global regularity
\begin{equation*}
    f_5(\mathbf{x}) = \left| \frac{x^2 - \cos(5xy)/2 + \sin(5z)}{2}\right|^3, \quad \mathbf{x} = (x,y,z)^\top \in \mathbb{S}^2.
\end{equation*}

The numerical performance of hyperinterpolation, utilizing the quadrature weights summarized in \cref{tab:quadrature_hyperinterpolation}, is visualized in \cref{fig:hyper_halton,fig:hyper_magsat}. The numerical results show that the quadratures (Spec) and ($D$-opt) consistently yield smoother and more accurate approximants than both the geometric Voronoi and $\ell_2^+$ baselines. It is also noted that negative weights introduced by the proposed optimization approach do not degrade performance, while they indeed facilitate superior spatial approximation.

\begin{figure}[htbp]
    \centering
    \begin{subfigure}{\textwidth}
        \centering
        \includegraphics[width=0.9\textwidth]{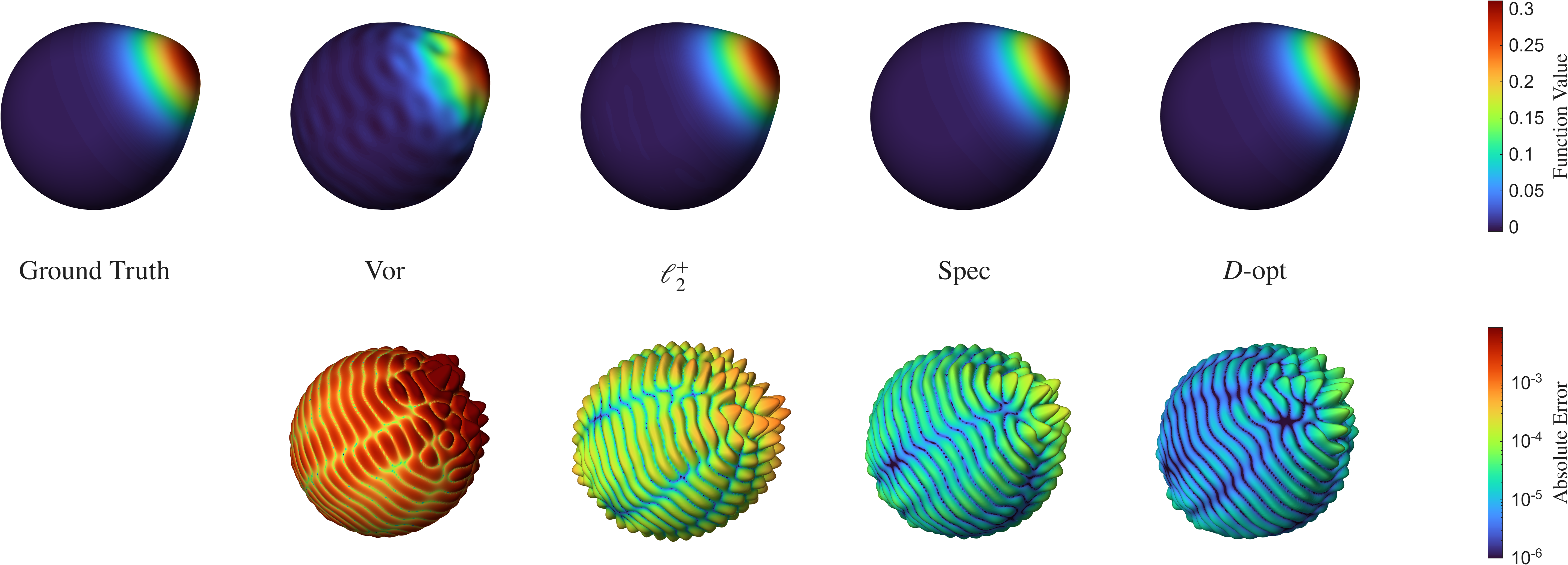}
        \caption{Hyperinterpolation on $f_3$.}
    \end{subfigure}
    
    \vspace{5mm}
    
    \begin{subfigure}{\textwidth}
        \centering
        \includegraphics[width=0.9\textwidth]{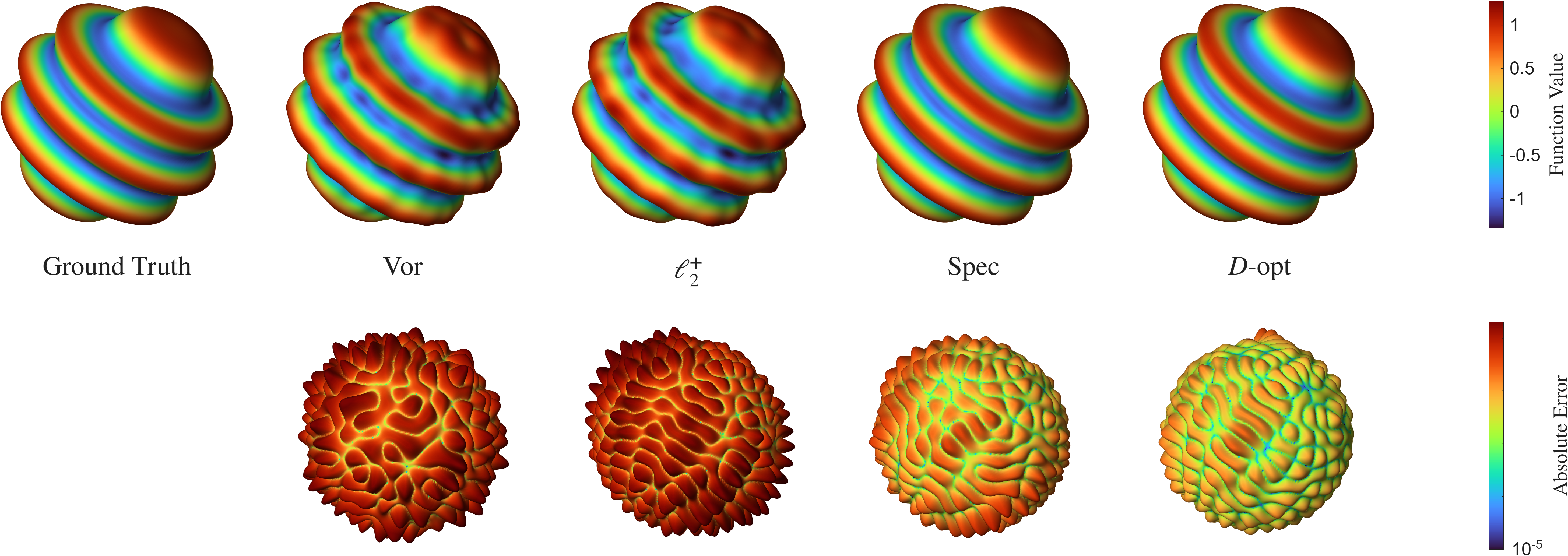}
        \caption{Hyperinterpolation on $f_4$.}
    \end{subfigure}

    \vspace{5mm}
    
    \begin{subfigure}{\textwidth}
        \centering
        \includegraphics[width=0.9\textwidth]{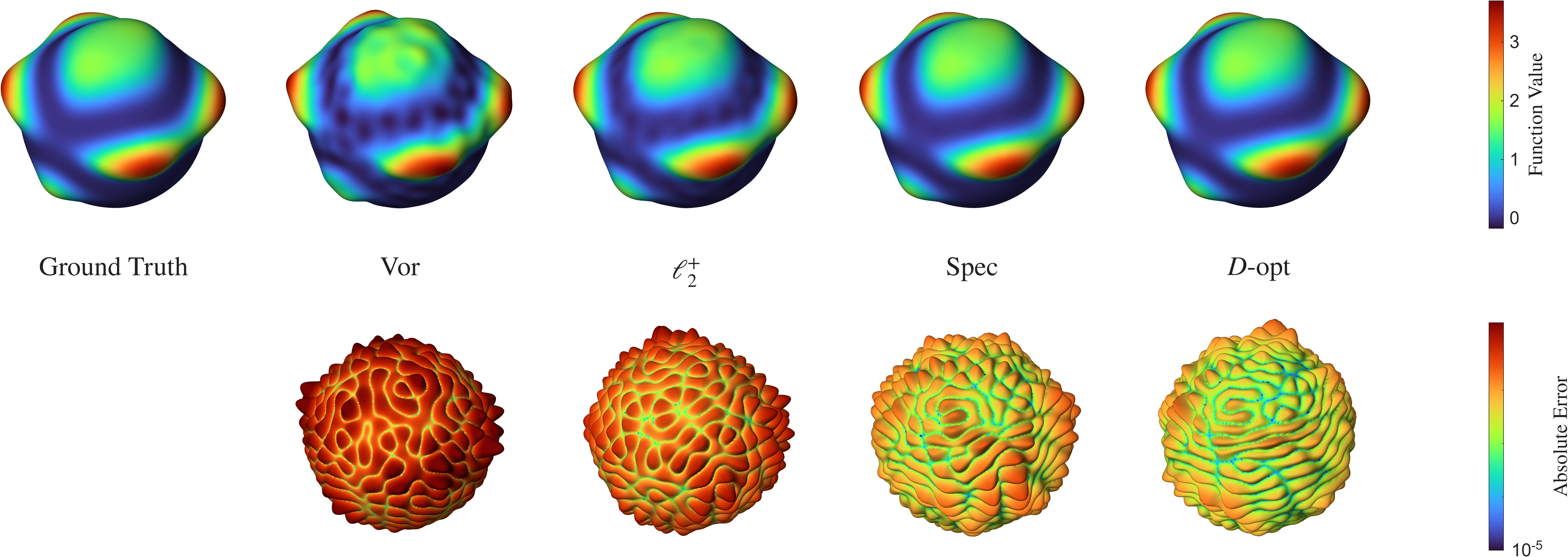}
        \caption{Hyperinterpolation on $f_5$.}
    \end{subfigure}

    \caption{Visual performance of the hyperinterpolation based on $N=2,048$ Halton points. The degree of the target $L^2$ MZ inequality is set as $n = 22$. The top row of each subfigure displays the computed approximant $L_n f$, while the bottom row illustrates the absolute error $|L_n f(\mathbf{x}) - f(\mathbf{x})|$ on the sphere.}
    \label{fig:hyper_halton}
\end{figure}

\begin{figure}[htbp]
    \centering
    \begin{subfigure}{\textwidth}
        \centering
        \includegraphics[width=0.9\textwidth]{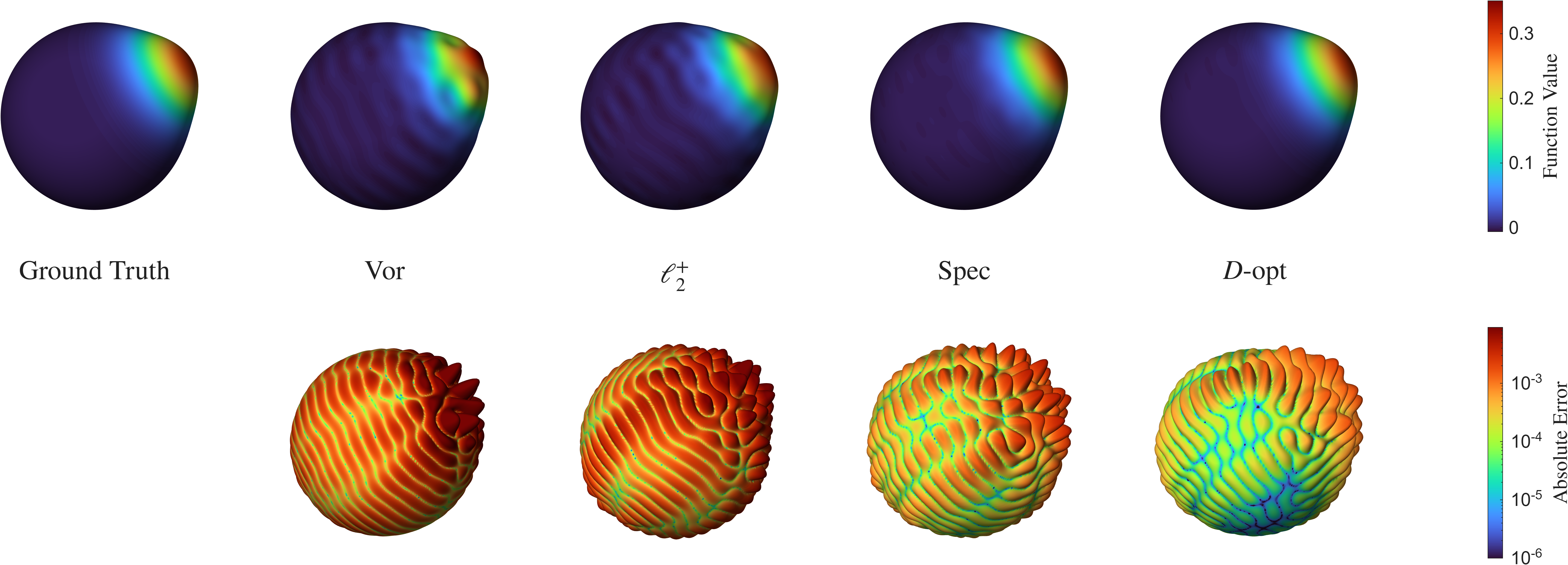}
        \caption{Hyperinterpolation on $f_3$.}
    \end{subfigure}

    \vspace{5mm}
    
    \begin{subfigure}{\textwidth}
        \centering
        \includegraphics[width=0.9\textwidth]{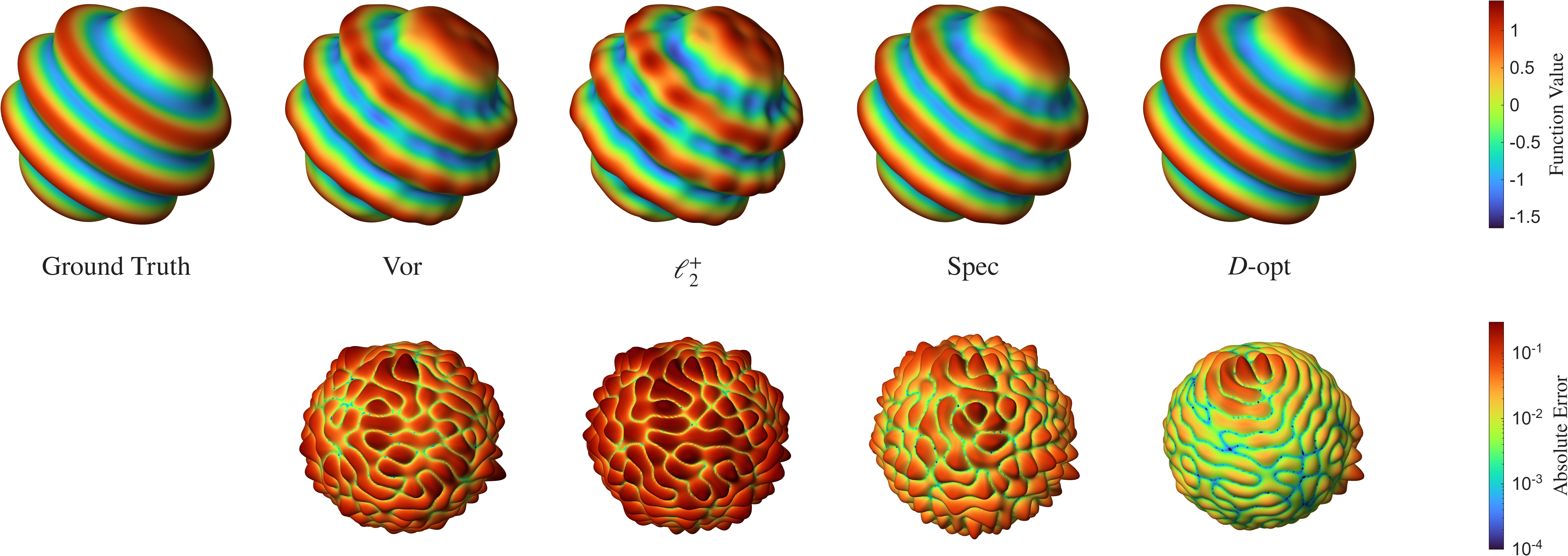}
        \caption{Hyperinterpolation on $f_4$.}
    \end{subfigure}

    \vspace{5mm}

    \begin{subfigure}{\textwidth}
        \centering
        \includegraphics[width=0.9\textwidth]{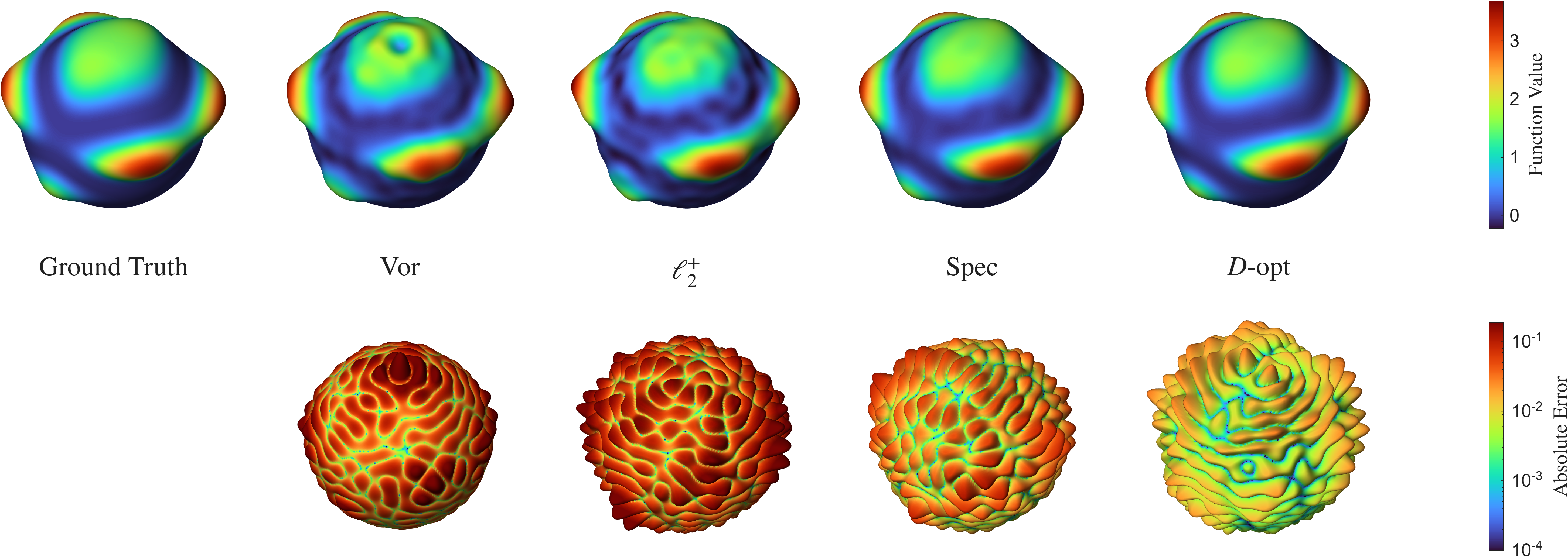}
        \caption{Hyperinterpolation on $f_5$.}
    \end{subfigure}

    \caption{Visual performance of the hyperinterpolation based on $N=2,048$ MAGSAT points. The degree of the target $L^2$ MZ inequality is set as $n = 21$. The top row of each subfigure displays the computed approximant $L_n f$, while the bottom row illustrates the absolute error $|L_n f(\mathbf{x}) - f(\mathbf{x})|$ on the sphere.}
    \label{fig:hyper_magsat}
\end{figure}

\newpage
These visual improvements are further quantified by the $L^2$ and $L^\infty$ hyperinterpolation errors reported in \cref{tab:hyper_error}, respectively. To approximate the continuous $L^2$ and $L^\infty$ errors over the unit sphere, we evaluate the functions on a dense validation set $Y = \{\mathbf{y}_k\}_{k=1}^{N_\text{val}} \subseteq \mathbb{S}^2$ consisting of $N_{\text{val}} = 50,000$ equal-partition points \cite{leopardi2007distributing}, and approximate as
\begin{equation*}
    \|L_nf - f\|_{L^2} \approx \left(\frac{4\pi}{N_{\text{val}}} \sum_{k=1}^{N_\text{Val}} |L_nf(\mathbf{y}_k) - f(\mathbf{y}_k)|^2\right)^{1/2},
\end{equation*}
and
\begin{equation*}
    \|L_nf - f\|_{L^\infty} \approx \max_{1 \leq k \leq N_\text{val}} |L_nf(\mathbf{y}_k) - f(\mathbf{y}_k)|.
\end{equation*}

\begin{table}[htbp]
\centering
\setlength{\tabcolsep}{4pt}
\caption{Approximated $L^\infty$ and $L^2$ hyperinterpolation error.}
\label{tab:hyper_error}
\begin{adjustbox}{max width=\textwidth}
\begin{tabular}{ccccccccc}
\toprule
\multirow{2}{*}{Point} & \multirow{2}{*}{Quadrature} & \multicolumn{2}{c}{$f_3$} & \multicolumn{2}{c}{$f_4$} & \multicolumn{2}{c}{$f_5$} \\
\cmidrule(lr){3-4} \cmidrule(lr){5-6} \cmidrule(lr){7-8}
 & & $\|L_m f - f\|_{L^2}$ & $\|L_m f - f\|_{L^\infty}$ & $\|L_m f - f\|_{L^2}$ & $\|L_m f - f\|_{L^\infty}$ & $\|L_m f - f\|_{L^2}$ & $\|L_m f - f\|_{L^\infty}$ \\ \midrule
\multirow{4}{*}{Halton} & Vor & 1.4425e-02 & 4.2628e-02 & 3.0726e-01 & 3.3278e-01 & 2.8700e-01 & 5.5644e-01 \\
 & $\ell_2^+$ & 6.3151e-04 & 1.0307e-03 & 3.2360e-01 & 3.8512e-01 & 1.1164e-01 & 1.7888e-01 \\
 & Spec & 1.2474e-04 & 2.3553e-04 & 5.8403e-02 & 8.5187e-02 & 2.6443e-02 & 3.5628e-02 \\
 & $D$-opt & 4.6647e-05 & 8.0689e-05 & 2.3681e-02 & 4.0263e-02 & 1.6248e-02 & 2.2968e-02 \\ \midrule
\multirow{4}{*}{MAGSAT} & Vor & 1.6645e-02 & 5.0504e-02 & 3.2144e-01 & 4.9613e-01 & 2.9955e-01 & 9.5901e-01 \\
 & $\ell_2^+$ & 8.9014e-03 & 1.9594e-02 & 4.9438e-01 & 6.4968e-01 & 2.7519e-01 & 3.9640e-01 \\
 & Spec & 2.6437e-03 & 4.6222e-03 & 1.5740e-01 & 1.8390e-01 & 8.5394e-02 & 1.2063e-01 \\
 & $D$-opt & 1.2749e-03 & 2.0977e-03 & 7.4342e-02 & 1.4868e-01 & 3.1267e-02 & 3.0362e-02 \\ \bottomrule
\end{tabular}
\end{adjustbox}
\end{table}

A particularly illuminating dynamic emerges when comparing the spectral and $D$-optimal collocation. The spectral collocation successfully drives $\eta$ to a smaller value than the $D$-optimal surrogate. However, this minimal $\eta$ is achieved at the cost of a larger $c$, resulting in a less stable quadrature. This instability translates directly into practice: ($D$-opt) consistently produces lower overall approximation errors than (Spec) as observed visually in \cref{fig:hyper_halton,fig:hyper_magsat} and numerically in \cref{tab:hyper_error}.

This empirical behavior is consistent with the accuracy-stability decomposition of the hyperinterpolation error established in \eqref{eq:hyperinterpolation-error-decomposition}. While minimizing the MZ constants is paramount, pursuing an absolute minimal $\eta$ often destabilizes the hyperinterpolation operator. Consequently, slightly relaxed and near-optimal $D$-optimal quadratures may provide more robust and superior balances for practical hyperinterpolation tasks.

\section{Conclusions}

Since quadrature nodes for many real-world applications are determined by physical constraints rather than optimal designs, we are inspired to consider the problem of weight collocation for the scattered nodes. Our principle of weight collocation is built upon P\'olya's foundational theorem on quadrature convergence in 1933, and we study this problem via an optimization perspective. More specifically, we identify that the classical pursuits  of weight positivity and high algebraic exactness are not necessary. 
We show that, by relaxing these requirements, the task of constructing quadrature rules for fixed node sets can be characterized by standard convex optimization problems. Moreover, we propose the kernel collocation for the minimization of worst-case integration error, and the MZ collocation for the stability--accuracy decomposition of hyperinterpolation error. 
We refine the spectral stability analysis of the discrepancy matrix, which provides a theoretically justified truncation level, and delineate the sensitivity of the kernel collocation process to functional smoothness. Also, the $2$-optimality of Voronoi partitions eventually leads to a geometry-aware regularizer that controls quadrature weights according to the underlying geometry of scattered sets for hyperinterpolation. The resulting standard optimization models can theoretically underpin numerical integration and hyperinterpolation, and they are computationally solvable by off-the-shelf solvers with high efficiency. Our theoretical assertions are also supported by comprehensive numerical studies. 

In summary, we provide a computationally efficient approach to weight collocation for scattered spherical data from the optimization perspective, and it would be natural to consider extending our methodologies to other Riemannian manifolds, such as higher-dimensional spheres and tori.

\section*{Acknowledgment}
The authors would like to thank Prof. Alvise Sommariva for generously providing the MAGSAT dataset.

\bibliographystyle{siamplainmc}
\small
\bibliography{references}

\end{document}